\documentclass[twoside, 11pt]{article}
\usepackage{amsmath,amssymb,amsthm,mathrsfs}
\usepackage[margin=2.5cm]{geometry} 
\usepackage{lipsum}
\usepackage{titlesec,hyperref}
\usepackage{fancyhdr}
\usepackage[numbers,sort&compress]{natbib}
\usepackage{color}

\fancypagestyle{plain}
{
\fancyhf{}

}

\titleformat{\subsection}{\it}{\thesubsection.\enspace}{1.5pt}{}
\titleformat{\subsubsection}{\it}{\thesubsubsection.\enspace}{1.5pt}{}

\newtheorem{theo}{Theorem}[section]
\newtheorem{lemm}[theo]{Lemma}

\newtheorem{prop}[theo]{Proposition}
\newtheorem{rema}{Remark}[section]

\numberwithin{equation}{section}

\allowdisplaybreaks

\def\th2{\frac{\theta}{2}}

\begin{document}
\title{Uniform Regularity and Vanishing Viscosity Limit for the
Nematic Liquid Crystal Flows in Three Dimensional Domain \hspace{-4mm}}
\author{Jincheng Gao$^\dag$  \quad Boling Guo $^\dag$ \quad Xiaoyu Xi$^\ddag$\\[10pt]
\small {$^\dag$ Institute of Applied Physics and Computational Mathematics,}\\
\small {100088, Beijing, P. R. China}\\[5pt]
\small {$^\ddag $Graduate School of China Academy of Engineering Physics,}\\
\small {100088, Beijing, P. R. China}\\[5pt]
}

\footnotetext{Email: \it gaojc1998@163.com(J.C. Gao),
                     \it gbl@iapcm.ac.cn(B.L. Guo)
                     \it xixiaoyu1357@126.com(X.Y. Xi).}
\date{}

\maketitle

\begin{abstract}
In this paper, we investigate the uniform regularity and vanishing limit
for the incompressible nematic liquid crystal flows in three dimensional
bounded domain.
It is shown that there exists a unique strong solution for the incompressible
nematic liquid crystal flows with boundary condition in a finite
time interval which is independent of the viscosity.
The solution is uniformly bounded in a conormal Sobolev space.
Finally, we also study the convergence rate of the viscous solutions to the inviscid
ones.

\vspace*{5pt}

\vspace*{5pt}
\end{abstract}


\section{Introduction}

\quad In this paper, we consider the incompressible nematic liquid crystal
flows as follows
\begin{equation}\label{eq1}
\left\{
\begin{aligned}
&u^{\varepsilon}_t+u^{\varepsilon}\cdot \nabla u^{\varepsilon}+\nabla p^{\varepsilon}
=\varepsilon \Delta u^{\varepsilon}
-\nabla d^{\varepsilon} \cdot \Delta d^{\varepsilon} ,
& (x, t)\in \Omega \times (0, T),\\
&d^{\varepsilon}_t+u^{\varepsilon}\cdot \nabla d^{\varepsilon}
=\Delta d^{\varepsilon}+|\nabla d^{\varepsilon}|^2 d^{\varepsilon},
& (x, t)\in \Omega \times (0, T),\\
&{\rm div} u^{\varepsilon}=0,
& (x, t)\in \Omega \times (0, T).
\end{aligned}
\right.
\end{equation}
Here $0 < T \le +\infty$ and $\Omega$ is a bounded smooth domain of $\mathbb{R}^3$.
The unknown vector functions
$u^\varepsilon(x, t)=(u^\varepsilon_1(x,t), u^\varepsilon_2(x,t), u^\varepsilon_3(x,t)), d^\varepsilon(x,t)=(d_1^\varepsilon(x,t), d_2^\varepsilon(x,t), d_3^\varepsilon(x,t))$
and scalar function $p^\varepsilon(x,t)$
represent the velocity field of fluid,
the macroscopic average of the nematic liquid crystal orientation field,
and  pressure respectively.
The parameter $\varepsilon>0$ is the inverse of the Reynolds number.
Here $$(u^\varepsilon\cdot \nabla u^\varepsilon)_i=\sum_{j=1}^3 u^\varepsilon_j\partial_j u^\varepsilon_i,~~(\nabla d^\varepsilon \cdot \Delta d^\varepsilon)_i
=\sum_{j,k=1}^3 \partial_i d_j^\varepsilon \partial_{kk}d_j^\varepsilon,$$ for $i=1,2,3$.
The system \eqref{eq1}, first proposed by Lin \cite{Lin}, is a simplified version of the general
Ericksen-Leslie system modeling the hydrodynamic flow of nematic liquid crystal
materials proposed by Ericksen \cite{Ericksen}  and Leslie \cite{Leslie}
during the period between
$1958$ and $1968$. System \eqref{eq1} is a macroscopic continuum description of the time
evolution of the material under the influence of both the fluid velocity field and the
macroscopic description of the microscopic orientation configurations of rod-like
liquid crystals. The interested readers can refer to \cite{{Lin},{Ericksen},{Leslie}},
and Lin-Liu \cite{Lin-Liu} for more details.
Corresponding to the system \eqref{eq1}, we impose the following Navier-slip type
and Neumann boundary
conditions:
\begin{equation}\label{bc1}
u^\varepsilon \cdot n=0, \quad ((Su^\varepsilon)n)_\tau=-(A u^\varepsilon)_\tau,
\quad {\rm and}
\quad \frac{ \partial d^\varepsilon}{\partial n}=0, \quad {\rm on}~\partial \Omega,
\end{equation}
where $A$ is a given smooth symmetric matrix(see \cite{Gie-Kelliher}),
$n$ is the outward unit vector normal to $\partial \Omega$,
$(A u^\varepsilon)_\tau$ represents the tangential part of $A u^\varepsilon$.
The strain tensor $Su^\varepsilon$ is defined by
\begin{equation*}
Su^\varepsilon=\frac{1}{2}\left((\nabla u^\varepsilon)+(\nabla u^\varepsilon)^t\right).
\end{equation*}
For smooth solutions, it is noticed that
$$
(2S(v)n-(\nabla \times v)\times n)_\tau=-(2S(n)v)_\tau,
$$
see \cite{Xiao-Xin1} for detail. Hence, the boundary condition \eqref{bc1} can
be written in the form of the vorticity as
\begin{equation}\label{bc2}
u^{\varepsilon}\cdot n=0,
\quad n\times (\nabla \times u^{\varepsilon})=[Bu^{\varepsilon}]_\tau,
\quad {\rm and}\quad
\frac{\partial d^{\varepsilon}}{\partial n}=0,
\quad {\rm on}~\partial \Omega,
\end{equation}
where $B=2(A-S(n))$ is symmetric matrix.
Actually, it turns out that the form \eqref{bc2} will be more
convenient than \eqref{bc1} in the energy estimates.

In this paper, we are interested in the existence of strong solution of \eqref{eq1}
with uniform bounds on an interval of time independent of viscosity
$\varepsilon\in (0, 1]$ and the vanishing
viscosity  limit to the corresponding
invicid nematic liquid crystal flows as $\varepsilon$ vanishes, i.e.
\begin{equation}\label{eq2}
\left\{
\begin{aligned}
&u_t+u\cdot \nabla u+\nabla p=-\nabla d \cdot \Delta d ,
& (x, t)\in \Omega \times (0, T),\\
&d_t+u\cdot \nabla d=\Delta d+|\nabla d|^2 d,
& (x, t)\in \Omega \times (0, T),\\
&{\rm div} u=0,
& (x, t)\in \Omega \times (0, T),\\
\end{aligned}
\right.
\end{equation}
with the boundary condition
\begin{equation}\label{bc3}
u\cdot n=0,
\quad {\rm and}\quad
\frac{\partial d}{\partial n}=0,
\quad {\rm on}~\partial \Omega.
\end{equation}

When $d$ is a constant vector field, the systems \eqref{eq1} and \eqref{eq2}
are the well-known Navier-Stokes equations and Euler equations respectively.
There is lots of literature on the uniform bounds and the vanishing viscosity
limit for the Navier-Stokes equations when the domain has no boundaries,
for instance, \cite{{Constantin}, {Constantin-Foias},{Kato}, {Masmoudi}}.
The problem is that due to the presence of a boundary the time of existence
$T^\varepsilon$ depends on the viscosity,
and one often cannot prove that it stays bounded away from zero.
Nevertheless, in
a domain with boundaries, for some special types of Navier-slip boundary conditions
or boundaries, some uniform $H^3$ (or $W^{2,p}$, with $p$ large enough) estimates and
a uniform time of existence for Navier-Stokes when the viscosity goes to zero
have recently been obtained (see \cite{{Beira1},{Beira2},{Xiao-Xin}}).
It is easy to see that, for these special
boundary conditions, the main part of the boundary layer vanishes, which allows
this uniform control in some limited regularity Sobolev space.
Recently, Masmoudi and Rousset \cite{Masmoudi-Rousset}
established conormal uniform estimates for three-dimensional general
smooth domains with the Naiver-slip boundary condition, which, in particular,
implies the uniform boundedness of the normal first order derivatives of the velocity
field. This allows the authors \cite{Masmoudi-Rousset}
to obtain the convergence of the viscous solutions to the inviscid ones
by a compact argument. Based on the uniform estimates in \cite{Masmoudi-Rousset},
better convergence with rates have been studied in \cite{Gie-Kelliher}
and \cite{Xiao-Xin2}. In particular, Xiao
and Xin \cite{Xiao-Xin2} have proved the convergence in $L^\infty(0, T ;H^1)$
with a rate of convergence.
Recently, Wang et al. \cite{Wang-Xin-Yong} investigated the
inviscid limit for isentropic compressible
Navier-Stokes equation with general Navier-slip boundary conditions.
The authors in \cite{Wang-Xin-Yong} not only proved that solution
is uniformly bounded  in a conormal Sobolev space and is uniformly bounded in
$W^{1, \infty}$, but also obtained the convergence rate of viscous solutions
to the inviscid ones. As the vanishing viscosity limit for the full compressible
Navier-Stokes equations with general Navier-slip and Neumann boundary conditions,
the readers can refer to \cite{Wang}.

Before stating our main results, we first explain the notations and conventions
used throughout this paper. Similar to \cite{{Masmoudi-Rousset},{Wang-Xin-Yong}},
one assumes that the bounded domain $\Omega \subset \mathbb{R}^3$ has a covering that
\begin{equation}
\Omega \subset \Omega_0 \cup_{k=1}^n \Omega_k,
\end{equation}
where $\overline{\Omega}_0$, and in each $\Omega_k$ there exists a function
$\psi_k$ such that
\begin{equation*}
\begin{aligned}
&\Omega \cap \Omega_k =\{x=(x_1, x_2, x_3)| x_3 > \psi_k(x_1, x_2)\}\cap \Omega_k,\\
&\partial \Omega \cap \Omega_k=\{x_3=\psi_k(x_1, x_2)\}\cap \Omega_k.
\end{aligned}
\end{equation*}
Here, $\Omega$ is said to be $\mathcal{C}^m$ if the functions $\psi_k$
are a $\mathcal{C}^m-$function. To define the conormal Sobolev spaces,
one considers $(Z_k)_{1\le k \le N}$ to be a finite set of generators of vector fields
that are tangential to $\partial \Omega$, and sets
\begin{equation*}
H_{co}^m=\{f\in L^2(\Omega)|Z^I f\in L^2(\Omega), ~ {\rm for}~|I|\le m\},
\end{equation*}
where $I=(k_1,..., k_m)$. The following notations will be used
$$
\begin{aligned}
&\|u\|_m^2=\|u\|_{H^m_{co}}^2=\sum_{j=1}^3 \sum_{|I|\le m}\|Z^I u_j\|_{L^2}^2,\\
&\|u\|_{m, \infty}^2=\sum_{|I|\le m}\|Z^I u\|_{L^\infty}^2,
\end{aligned}
$$
and
$$
\|\nabla Z^m u\|^2=\sum_{|I|=m}\|\nabla Z^I u\|_{L^2}^2.
$$
Noting that by using the covering of $\Omega$, one can always assume that each
vector field $(u, d)$ is supported in one of the $\Omega_i$, and moreover,
in $\Omega_0$ the norm $\|\cdot\|_m$ yields a control of the standard $H^m$ norm,
whereas if $\Omega_i \cap \partial \Omega\neq\emptyset$, there is no control of the
normal derivatives.

Since $\partial \Omega$ is given locally by $x_3=\psi(x_1, x_2)$
(we omit the subscript $j$ of notational convenience), it is convenient to
use the coordinates
$$
\Psi:~(y, z)\mapsto (y, \psi(y)+z)=x.
$$
A  basis is thus given by the vector fields $(e_{y^1}, e_{y^2}, e_z)$,
where $e_{y^1}=(1, 0, \partial_1 \psi)^t,~e_{y^1}=(0, 1, \partial_2 \psi)^t$,
and $e_z=(0, 0, -1)^t$. On the boundary, $e_{y^1}$ and $e_{y^2}$ are tangent to
$\partial \Omega$, and in general, $e_z$ is not a normal vector field.
By using this parametrization, one can take as suitable vector fields compactly
supported in $\Omega_j$ in the definition of the $\|\cdot \|_m$ norms
\begin{equation*}
Z_i=\partial_{y^i}=\partial_i+\partial_i \psi \partial_z,~i=1,2,~~
Z_3=\varphi(z)\partial_z,
\end{equation*}
where $\varphi(z)=\frac{z}{1+z}$ is smooth, supported in $\mathbb{R}_+$
with the property $\varphi(0)=0, \varphi'(0)>0, \varphi(z)>0$ for $z>0$.
It is easy to check that
$$
Z_k Z_j=Z_j Z_k,~~j, k =1,2,3,
$$
and
$$
\partial_z Z_i=Z_i \partial_z,~i=1,2;\quad \partial_z Z_3\neq Z_3 \partial_z.
$$
We shall still denote by $\partial_j,~j=1,2,3,$ or
$\nabla$ the derivatives in the physical space. The coordinates of a vector field
$u$ in the basis $(e_{y^1}, e_{y^2}, e_z)$ will be denoted by $u^i$, and thus
$$
u=u^1 e_{y^1}+u^2 e_{y^2}+u^3 e_{z}.
$$
We shall denote by $u_j$ the coordinates in the standard basis of $\mathbb{R}^3$,
i.e, $u=u_1 e_1+u_2 e_2+u_3 e_3$. Denote by $n$ the unit outward normal in the
physical space which is given locally by
\begin{equation}
n(x)\equiv n(\Psi(y,z))=\frac{1}{\sqrt{1+|\nabla \psi(y)|^2}}
\left(
\begin{array}{c}
\partial_1 \psi(y)\\
\partial_2 \psi(y)\\
-1
\end{array}
\right)
\triangleq \frac{-N(y)}{\sqrt{1+|\nabla \psi(y)|^2}},
\end{equation}
and by $\Pi$ the orthogonal projection
\begin{equation}
\Pi u\equiv\Pi(\Psi(y,z))u=u-[u\cdot n(\Psi(y,z))]n(\Psi(y,z)),
\end{equation}
which gives the orthogonal projection on to the tangent space of the boundary.
Note that $n$ and $\Pi$ are defined in the whole $\Omega_k$ and do not depend on $z$.
For later use and notational convenience, set
\begin{equation*}
\mathcal{Z}^\alpha=\partial_t^{\alpha_0}Z^{\alpha_1}
=\partial_t^{\alpha_0}Z_1^{\alpha_{11}}Z_2^{\alpha_{12}}Z_3^{\alpha_{13}},
\end{equation*}
where $\alpha, \alpha_0$ and  $\alpha_1$ are the differential multi-indices with
$\alpha=(\alpha_0, \alpha_1), \alpha_1=(\alpha_{11}, \alpha_{12}, \alpha_{13})$
and we also use the notation
\begin{equation}\label{sdef2}
\|f(t)\|_{\mathcal{H}^{m}}^2
=\sum_{|\alpha|\le m}\|\mathcal{Z}^\alpha f(t)\|_{L^2_x}^2,
\end{equation}
and
\begin{equation}\label{sdef3}
\|f(t)\|_{\mathcal{H}^{k,\infty}}
=\sum_{|\alpha|\le k}\|\mathcal{Z}^\alpha f(t)\|_{L^\infty_x}
\end{equation}
for smooth space-time function $f(x,t)$.
Throughout this paper, the positive generic constants that are independent
of $\varepsilon$ are denoted by $c, C$.
Denote by $C_k$ a positive constant independent of $\varepsilon \in (0, 1]$
which depends only on the $\mathcal{C}^k-$norm of the functions $\psi_j,~j=1,...,n.$
$\|\cdot\|$ denotes the standard $L^2(\Omega; dx)$ norm,
and $\|\cdot\|_{H^m}(m=1,2,3,...)$ denotes the Sobolev $H^m(\Omega, dx)$ norm.
The notation $|\cdot|_{H^m}$ will be used for the standard Sobolev norm
of functions defined on $\partial \Omega$. Note that this norm involves only
tangential derivatives. $P(\cdot)$ denotes a polynomial function.

Since the boundary layers may appear in the presence of physical boundaries,
in order to obtain the uniform estimation for solutions to the nematic liquid
crystal flows with Navier-slip and Neumann boundary conditions, we needs to find a
suitable functional space. Indeed, it is impossible to get a uniform bound
for $(u, d)$ in a standard Sobolev spaces due to possible boundary layers.
In order to overcome such a difficulty, Masmoudi and Rousset \cite{Masmoudi-Rousset}
used conormal Sobolev  space to investigate the uniform regularity of the solution
for the incompressible Navier-Stokes equations.
More precisely, they defined the functional space
\begin{equation*}
\underset{0\le t\le T}{\sup}\{\|u(t)\|_{H^m_{co}}^2+\|\nabla u(t)\|_{H^{m-1}_{co}}^2
+\|\nabla u\|_{1,\infty}^2\},
\end{equation*}
which implies that the solutions are uniform bounded in $W^{1, \infty}$.
Note that the nematic liquid crystal flows
\eqref{eq1} is a coupling between the incompressible Navier-Stokes equations
and a transported heat flow of harmonic maps into $S^2$.
Then, in the spirit of Masmoudi and Rousset \cite{Masmoudi-Rousset},
we also investigate the solutions of the nematic liquid crystal flows
in Conormal Sobolev space.
Hence, the functional space should include some information for the
direction field $d$. On the other hand, due to the nonlinear higher
order derivatives term $\nabla d\cdot \Delta d$, one should control this term
by using the dissipative term $\Delta d$ on the right hand side of the equation \eqref{eq1}$_2$
which involving the time derivatives term $d_t$. Hence, we also include some
information involving the time derivatives in the functional space.
Therefore, we define the functional space $X_m(T)$ for a pair of function
$(u, d)(x, t)$ as follows
\begin{equation}
X_m(T)=\{(u, d)\in L^\infty([0, T], L^2);
       ~\underset{0\le t\le T}{\rm esssup}\|(u, d)(t)\|_{X_m}<+\infty\},
\end{equation}
where the norm $\|(\cdot, \cdot)\|_{X_m}$ is given by
\begin{equation}
\|(u, d)(t)\|_{X_m}
=\|u\|_{\mathcal{H}^{m}}^2
       +\|d\|_{L^2}^2
       +\|\nabla d\|_{\mathcal{H}^{m}}^2
       +\|\nabla u\|_{\mathcal{H}^{m-1}}^2
       +\|\Delta d\|_{\mathcal{H}^{m-1}}^2
       +\|\nabla u\|_{\mathcal{H}^{1,\infty}}^2.
\end{equation}
In the present paper, we supplement the nematic liquid crystal flows system \eqref{eq1}
with initial data
\begin{equation}\label{ID}
(u^\varepsilon, d^\varepsilon)(x, 0)=(u^\varepsilon_0, d^\varepsilon_0)(x),
\end{equation}
and
\begin{equation}\label{IDB}
\begin{aligned}
\underset{0< \varepsilon \le 1}{\sup}\|(u^\varepsilon_0, d^\varepsilon_0)\|_{X_m}
=
&\underset{0< \varepsilon \le 1}{\sup}\{\|u_0^\varepsilon\|_{\mathcal{H}^{m}}^2
       +\|d_0^\varepsilon\|_{L^2}^2
       +\|\nabla d_0^\varepsilon\|_{\mathcal{H}^{m}}^2
       +\|\nabla u_0^\varepsilon\|_{\mathcal{H}^{m-1}}^2\\
&\quad  \quad \quad     +\|\Delta d_0^\varepsilon\|_{\mathcal{H}^{m-1}}^2
       +\|\nabla u_0^\varepsilon\|_{\mathcal{H}^{1,\infty}}^2\}\le \widetilde{C}_0,
\end{aligned}
\end{equation}
where $\widetilde{C}_0$ is a positive constant independent of $\varepsilon \in (0,1]$,
and the time derivatives of initial data are defined through the equation \eqref{eq1}.
Thus, the initial data $(u_0^\varepsilon, d_0^\varepsilon)$ is assumed to have a
higher space regularity and compatibilities.
Notice that the a priori estimates in Theorem  \ref{Theoream3.1}
below are obtained in the case that the
approximate solution is sufficiently smooth up to the boundary, and therefore,
in order to obtain a selfconstained result, one needs to assume the approximated
initial data satisfies the boundary compatibilities condition \eqref{bc2}.
For the initial data $(u^\varepsilon_0, d^\varepsilon_0)$ satisfying
\eqref{ID}, it is not clear if there exists an approximate sequences
$(u^{\varepsilon,\delta}_0, d^{\varepsilon,\delta}_0)$
($\delta$ being a regularization parameter) which satisfy the boundary
compatibilities and $\|(u^{\varepsilon,\delta}_0-u^\varepsilon_0, d^{\varepsilon,\delta}_0-d^\varepsilon_0)\|_{X_m}\rightarrow 0$ as
$\delta \rightarrow 0$. Therefore, we set
\begin{equation}
\begin{aligned}
X_{n,m}^{ap}
&=\left\{(u, d)\in H^{4m}(\Omega)\times H^{4(m+1)}(\Omega)\right.
|\partial_t^k u, \partial_t d, \partial_t^k \nabla d,
k=1,...,m{\rm~ are~ defined ~through}\\
&\quad \quad \quad \quad \quad \quad \quad \quad \quad \quad \quad \quad
\quad \quad
{\rm equations}~\eqref{eq1}~
{\rm ~the~and}~\partial_t^k u, \partial_t^k \nabla d, k=0,...,m-1\\
&\quad \quad \quad \quad \quad \quad \quad \quad \quad \quad \quad \quad
\quad \quad
{\rm satisfy~the~boundary~compatibility~condition}~\}
\end{aligned}
\end{equation}
and
\begin{equation}
X_{n,m}={\rm~ the ~closure~ of~}X_{n,m}^{ap}~{\rm in~ the~ norm~}\|(\cdot,\cdot)\|_{X_m}.
\end{equation}

Now, we state the first results concerning the uniform regularity for the nematic
liquid crystal flows \eqref{eq1}, \eqref{bc2} and \eqref{ID} as follows.

\begin{theo}[{\textbf {Uniform Regularity}}]\label{Theorem1.1}
Let $m$ be an integer satisfying $m \ge 6$, $\Omega$ be a $\mathcal{C}^{m+2}$ domain,
and $A\in C^{m+1}(\partial \Omega)$. Consider the initial data
$(u_0^\varepsilon, d_0^\varepsilon)\in X_{n,m}$
satisfy \eqref{IDB} and ${\rm div}u^\varepsilon_0=0$,~$|d^\varepsilon_0|=1$
in $\overline{\Omega}$.
Then, there exists a time $T_0>0$ and $\widetilde{C}_1>0$ independent of
$\varepsilon \in (0, 1]$, such that there exits a unique solution
of \eqref{eq1}, \eqref{bc2} and \eqref{ID} which is defined on $[0, T_0]$
and satisfies the estimate
\begin{equation}\label{111}
\begin{aligned}
&\underset{0\le \tau \le t}{\sup}
(\|d^\varepsilon\|_{L^2}^2+\|(u^\varepsilon, \nabla d^\varepsilon)\|_{\mathcal{H}^{m}}^2
+\|(\nabla u^\varepsilon, \Delta d^\varepsilon)\|_{\mathcal{H}^{m-1}}^2
+\|\nabla u^\varepsilon\|_{\mathcal{H}^{1,\infty}}^2)\\
&\quad  +\varepsilon \!\!\int_0^t\!\! (\|\nabla u^\varepsilon\|_{\mathcal{H}^{m}}^2+
\|\nabla^2 u^\varepsilon\|_{\mathcal{H}^{m-1}}^2) d\tau
+\!\!\int_0^t \!\!(\|\Delta d^\varepsilon\|_{\mathcal{H}^{m}}^2
+\|\nabla \Delta d^\varepsilon\|_{\mathcal{H}^{m-1}}^2 )d\tau
\le \widetilde{C}_1,
\end{aligned}
\end{equation}
where $\widetilde{C}_1$ depends only on $\widetilde{C}_0$ and $C_{m+2}$.
\end{theo}

\begin{rema}
For $(u_0^\varepsilon, d_0^\varepsilon)\in X_{n,m}$, it must hold that
$u_0^\varepsilon \cdot n|_{\partial \Omega}=0$,
$((Su_0^\varepsilon)n)_\tau|_{\partial \Omega}=-(A u_0^\varepsilon)_\tau|_{\partial \Omega}$,
and $n\cdot \nabla d_0^\varepsilon|_{\partial \Omega}=0$
in the trace sense for every fixed $\varepsilon \in (0, 1]$.
\end{rema}

\begin{rema}
Indeed, the nematic liquid crystal flows \eqref{eq1} has the following form
\begin{equation}\label{req1}
\left\{
\begin{aligned}
&u_t+u\cdot \nabla u+\nabla p
=\varepsilon \Delta u
-\lambda\nabla d \cdot \Delta d,
& (x, t)\in \Omega \times (0, T),\\
&d_t+u\cdot \nabla d
=\theta(\Delta d+|\nabla d|^2 d),
& (x, t)\in \Omega \times (0, T),\\
&{\rm div} u=0,
& (x, t)\in \Omega \times (0, T).
\end{aligned}
\right.
\end{equation}
Here the parameters $\varepsilon$, $\lambda$ and $\theta$ are positive constants representing the fluid viscosity, the competition between kinetic energy and potential energy, and the macroscopic elastic relaxation time for the molecular orientation field respectively.
If $\theta=0$, system \eqref{req1} for $(u, d)$ is
closely related to the MHD system for $(u,\psi)$ provided we identify $\psi=\nabla \times d$.
There have been many interesting works on global small solutions to the MHD
system recently, refer to \cite{{Lin-Zhang1},{Lin-Zhang2},{Lin-Zhang-Xu}}.
Due to the nonlinear term $\lambda \nabla d\cdot \Delta d$ on the right hand side of
\eqref{req1}$_1$, we can not obtain the uniform regularity estimates for the case that
both $\varepsilon$ and $\theta$ go to zero in this paper.
\end{rema}

The main steps of the proof of Theorem \ref{Theorem1.1} are the following.
First, we obtained a conormal energy estimates for $(u^\varepsilon, \nabla d^\epsilon)$
in $\mathcal{H}^{m}-$norm(see \eqref{sdef2} above for the definition
of $\mathcal{H}^{m}$).
The second step is to give the estimates for $\|\Delta d^\varepsilon\|_{\mathcal{H}^{m-1}}$,
which is easy to obtain since there exists a dissipative term $\Delta d^\varepsilon$
on the right-hand side of \eqref{eq1}$_2$.
The third step is to give the estimate for $\|\partial_n u^\varepsilon\|_{\mathcal{H}^{m-1}}$
and $\|\Delta d^\varepsilon\|_{\mathcal{H}^{m-1}}$.
In order to obtain this estimate by an energy method, $\partial_n u^\varepsilon$a
is not a convenient quantity because it does not vanish on the boundary.
Similar to \cite{Masmoudi-Rousset}, the quantity $\partial_n u^\varepsilon \cdot n$
can be controlled thanks to the divergence free condition of $u^\varepsilon$.
In order to give the estimate for $(\partial_n u^\varepsilon)_\tau$, one choose
the convenient quantity $\eta=w^\varepsilon \times n+(Bu^\varepsilon)_\tau$
with a homogeneous Dirichlet boundary conditions.
The fourth step is to give the estimate for the pressure, which is not transparent
in the estimate since the conormal fields $Z_i$ do not communicate with the
gradient operator. In the spirit of \cite{Masmoudi-Rousset}, we split
the pressure into two parts, the first part have the same regularity as in the
Euler equation and the second part is linked to the Navier condition.
The fifth step is to estimate $\|\Delta d^\varepsilon\|_{W^{1,\infty}}$.
Indeed, this estimate is easy to obtain since there exists a dissipation term $\Delta d^\varepsilon$ on the right-hand side of \eqref{eq1}$_2$.
The last step is to estimate $\|\nabla u^\varepsilon\|_{\mathcal{H}^{1,\infty}}$.
In fact, it suffices to estimate $\|(\partial_n u^\varepsilon)_\tau\|_{\mathcal{H}^{1,\infty}}$
since the other terms can be estimated by the Sobolev embedding.
We choose an equivalent quantity such that it satisfies a homogeneous Dirichlet
condition and solves a convection-diffusion equation at the leading order.
Then Theorem \ref{Theorem1.1} can be proved by these a priori estimates and
a classical iteration method.

Based on the uniform estimates in Theorem \ref{Theorem1.1},
in the spirit of similar arguments \cite{Masmoudi-Rousset},
we can obtain the vanishing viscosity limit of solutions of
\eqref{eq1} to the solutions of \eqref{eq2} in $L^\infty-$norm
by the strong compactness argument, without convergence rate.
In this paper, we hope to prove the vanishing viscosity limit with
rates of convergence, which can be stated as follows.

\begin{theo}[{\textbf{Inviscid Limit}}]\label{Theorem1.2}
Let $(u, d)(t)\in L^\infty(0, T_1; H^3)\times L^\infty(0, T_1; H^4)$
be the smooth solution to the equation
\eqref{eq2} and \eqref{bc2} with initial data $(u_0, d_0)$ satisfying
\begin{equation}\label{121}
(u_0, d_0)\in (H^3 \times H^4) \cap X_{n,m}~{\rm with}~m \ge 6.
\end{equation}
Let $(u^\varepsilon, d^\varepsilon)(t)$ be the solution to the initial boundary
value problem of the nematic liquid crystal flows \eqref{eq1}, \eqref{bc1}
with initial data $(u_0, d_0)$ satisfying \eqref{121}.
Then, there exists $T_2={\rm min}\{T_0, T_1\}>0$, which is independent of
$\varepsilon>0$, such that
\begin{equation}\label{122}
\|(u^\varepsilon-u)(t)\|_{L^2}^2+\|(d^\varepsilon-d)(t)\|_{H^1}^2
\le C\varepsilon^{\frac{3}{2}}, \quad t\in [0, T_2]
\end{equation}
and
\begin{equation}\label{123}
\|(u^\varepsilon-u)\|_{L^\infty(0, T_2; L^\infty(\Omega))}
+\|(d^\varepsilon-d)\|_{L^\infty(0,T_2; W^{1,\infty}(\Omega))}
\le C \varepsilon^{\frac{3}{10}},
\end{equation}
which $C$ depends only on the norm $\|u_0\|_{H^3}, \|d_0\|_{H^4}$
and $\|(u_0, d_0)\|_{X_m}$.
\end{theo}


The rest of the paper is organized as follows:
In section \ref{Preliminary}, we collect some inequalities that will be used later.
In section \ref{estimates}, the a priori estimates in Theorem \ref{Theoream3.1}
are proved. By using these a priori estimates, one give the proof for the
Theorem \ref{Theorem1.1} in section \ref{Proof1}.
Based on the uniform estimates obtained in Theorem \ref{Theorem1.1},
we establish the convergence rate for the solutions from \eqref{eq1} to \eqref{eq2}
and complete the proof for the Theorem \ref{Theorem1.2}.

\section{Preliminaries}\label{Preliminary}
\quad The following lemma \cite{{Xiao-Xin},{Temam}} allows us to control
the $H^m(\Omega)$-norm
of a vector valued function $u$ by its $H^{m-1}-$norm of $\nabla \times u$
and ${\rm div}u$, together with the $H^{m-\frac{1}{2}}(\partial \Omega)$
of $u\cdot n$.

\begin{prop}\label{prop2.1}
Let $m \in \mathbb{N}_+$ be an integer. Let $u\in H^m$ be a vector-valued function.
Then, there exists a constant $C>0$ independent of $u$, such that
\begin{equation}\label{21}
\|u\|_{H^m}\le C(\|\nabla \times u\|_{H^{m-1}}+\|{\rm div}u\|_{H^{m-1}}
                 +\|u\|_{H^{m-1}}+|u\cdot n|_{H^{m-\frac{1}{2}}(\partial \Omega)}),
\end{equation}
and
\begin{equation}\label{22}
\|u\|_{H^m}\le C(\|\nabla \times u\|_{H^{m-1}}+\|{\rm div}u\|_{H^{m-1}}
                 +\|u\|_{H^{m-1}}+|n \times u|_{H^{m-\frac{1}{2}}(\partial \Omega)}).
\end{equation}
\end{prop}

In this paper, one repeatedly use the Gagliardo-Nirenberg-Morser type inequality,
whose proof can be find in \cite{Gues}. First, define the space
\begin{equation}
W^m(\Omega \times [0, T])=\{f(x,t)\in L^2(\Omega \times [0, T])
|\mathcal{Z}^\alpha f \in L^2(\Omega \times [0, T]), |\alpha|\le m\}.
\end{equation}
Then, the Gagliardo-Nirenberg-Morser type inequality  can be stated as follows:
\begin{prop}\label{prop2.2}
For $u, v \in L^\infty(\Omega \times [0, T])\cap \mathcal{W}^m(\Omega \times [0, T])$
with $m \in \mathbb{N}_+$ an integer, it holds that
\begin{equation}\label{23}
\int_0^t \|(\mathcal{Z}^\beta u \mathcal{Z}^\gamma v)(\tau)\|^2 d\tau
\lesssim \|u\|_{L^\infty_{t,x}}^2\int_0^t \|v(\tau)\|_{\mathcal{H}^m}^2 d\tau
+\|v\|_{L^\infty_{t,x}}^2\int_0^t \|u(\tau)\|_{\mathcal{H}^m}^2 d\tau,
~|\beta|+|\gamma|=m.
\end{equation}
\end{prop}

We also need the following anisotropic Sobolev embedding and trace theorems,
refer to \cite{{Wang},{Wang-Xin-Yong}}.
\begin{prop}\label{prop2.3}
Let $m_1 \ge 0, m_2 \ge 0$ be integers and $f\in H_{co}^{m_1}(\Omega)\cap H_{co}^{m_2}(\Omega)$
and $\nabla f\in H_{co}^{m_2}(\Omega)$.\\
$(1)$ The following anisotropic Sobolev embedding holds:
\begin{equation}\label{24}
\|f\|_{L^\infty}^2\le C(\|\nabla f\|_{H^{m_2}_{co}}+\|f\|_{H^{m_2}_{co}})
\cdot \|f\|_{H^{m_1}_{co}},
\end{equation}
provided $m_1+m_2 \ge 3$.\\
$(2)$The following trace estimate holds:\\
\begin{equation}\label{25}
|f|_{H^s(\partial \Omega)}^2 \le  C(\|\nabla f\|_{H^{m_2}_{co}}+\|f\|_{H^{m_2}_{co}})
\cdot \|f\|_{H^{m_1}_{co}},
\end{equation}
provided $m_1+m_2 \ge 2s \ge 0$.
\end{prop}

\section{A priori estimates}\label{estimates}

\quad The aim of this section is to prove the following a priori estimates,
which are crucial to prove Theorem \ref{Theorem1.1}.
For notational convenience, we drop the superscript $\varepsilon$ throughout
this section.

\begin{theo}[a priori estimates]\label{Theoream3.1}
Let $m$ be an integer satisfying $m \ge 6$, $\Omega$ be a
$C^{m+2}$ domain, and $A\in \mathcal{C}^{m+1}(\partial \Omega)$.
For sufficiently smooth solutions defined on $[0, T]$ of
\eqref{eq1} and \eqref{bc1}, then the following a priori estimate holds
\begin{equation}
\begin{aligned}
&\underset{0\le \tau \le t}{\sup}N_m(\tau)
+\varepsilon \int_0^t( \|\nabla u\|_{\mathcal{H}^{m}}^2
+ \|\nabla^2 u\|_{\mathcal{H}^{m-1}}^2 )d\tau
+\int_0^t (\|\Delta d\|_{\mathcal{H}^{m}}^2+\|\nabla \Delta d\|_{\mathcal{H}^{m-1}}^2) d\tau\\
&
\le \widetilde{C}_2C_{m+2}\left\{N_m(0)+P(N_m(t))\int_0^t P(N_m(\tau)) d\tau\right\},
\quad \quad \forall t\in [0, T],
\end{aligned}
\end{equation}
where $\widetilde{C}_2$ depends only on $\frac{1}{c_0},$ $P(\cdot)$
is a polynomial, and
\begin{equation}\label{def1}
N_m(t)\triangleq
       \|u\|_{\mathcal{H}^{m}}^2
       +\|d\|_{L^2}^2
       +\|\nabla d\|_{\mathcal{H}^{m}}^2
       +\|\nabla u\|_{\mathcal{H}^{m-1}}^2
       +\|\Delta d\|_{\mathcal{H}^{m-1}}^2
       +\|\nabla u\|_{\mathcal{H}^{1,\infty}}^2.
\end{equation}
\end{theo}
Since the proof of Theorem \ref{Theoream3.1} is quite lengthy and involved,
we divide the proof into the following several subsections.

\subsection{Conormal Estimates for $u$ and $\nabla d$}

\quad For any smooth function $f$, notice that
\begin{equation*}
\Delta f=\nabla {\rm div}f-\nabla \times (\nabla \times f),
\end{equation*}
and then $\eqref{eq1}_1$ can be written as
\begin{equation}\label{eq1-1}
u_t+u\cdot \nabla u+\nabla p
=-\varepsilon \nabla \times (\nabla \times u)-\nabla d\cdot \Delta d.
\end{equation}

In this subsection, we first give the basic a priori $L^2$ energy estimate which holds
for \eqref{eq1} and \eqref{bc2}.
Let
\begin{equation}\label{def2}
Q(t)\triangleq \underset{0\le \tau \le t}{\sup}
\{\|(u, u_\tau)(\tau)\|_{L^{\infty}}^2
+\|d_\tau(\tau)\|_{W^{1,\infty}}^2
+\|\nabla d(\tau)\|_{W^{1,\infty}}^2
+\|\nabla \Delta d(\tau)\|_{L^\infty}^2
+\|\nabla u(\tau)\|_{\mathcal{H}^{1,\infty}}^2\}.
\end{equation}

\begin{lemm}\label{lemma3.1}
For a smooth solution to \eqref{eq1} and \eqref{bc2}, it holds that
for $\varepsilon \in (0, 1]$
\begin{equation}\label{311}
\begin{aligned}
&\underset{0\le \tau \le t}{\sup}\int(|u|^2+|d|^2+|\nabla d|^2)(\tau)dx
+\varepsilon \int_0^t \int |\nabla u|^2 dx d\tau
+\int_0^t \int |\Delta d|^2 dx d\tau\\
&\le \int(|u_0|^2+|\nabla d_0|^2)dx+C_2[1+Q(t))]\int_0^t N_0(\tau) d\tau.
\end{aligned}
\end{equation}
\end{lemm}
\begin{proof}
Multiplying \eqref{eq1}$_2$ by $d$ and integrating over $\Omega$, one arrives at
\begin{equation*}
\frac{d}{dt}\frac{1}{2}\int (|d|^2-1)dx
+\int u\cdot \nabla (|d|^2-1)dx
=\int \Delta d \cdot d ~dx+\int |\nabla d|^2 |d|^2 dx,
\end{equation*}
which, integrating by part and applying the boundary condition \eqref{bc2}, yields that
\begin{equation}\label{312-1}
\frac{d}{dt}\int(|d|^2-1)dx+2\int (|d|^2-1)|\nabla d|^2 dx=0.
\end{equation}
In view of the Gr\"{o}nwall inequality, one deduce from the identity \eqref{312-1} that
\begin{equation}\label{unit}
|d|=1\quad {\rm in}~\overline{\Omega}.
\end{equation}
Multiplying \eqref{eq1-1} by $u$, integrating by parts
and applying the boundary condition \eqref{bc2}, we find
\begin{equation}\label{312}
\frac{1}{2} \frac{d}{dt}\int |u|^2 dx
+\varepsilon \int \nabla \times (\nabla \times u)\cdot u ~dx
=-\int (u \cdot \nabla) d \cdot \Delta d ~dx.
\end{equation}
Integrating by part and applying the boundary condition \eqref{bc2}, we get
\begin{equation*}
\begin{aligned}
\int \nabla \times (\nabla \times u)u ~dx
&=\int_{\partial \Omega}n\times (\nabla \times u)u~ d\sigma
+\int |\nabla \times u|^2 dx \\
&=\int_{\partial \Omega}[Bu]_\tau \cdot  u_\tau~ d\sigma
+\int |\nabla \times u|^2 dx,
\end{aligned}
\end{equation*}
which, together with \eqref{312}, gives directly
\begin{equation}\label{313}
\frac{1}{2} \frac{d}{dt}\int |u|^2 dx
+\varepsilon \int |\nabla \times u|^2 dx
=-\varepsilon\int_{\partial \Omega}[Bu]_\tau \cdot  u_\tau~ d\sigma
-\int (u \cdot \nabla) d \cdot \Delta d ~dx.
\end{equation}
Multiplying \eqref{eq1}$_3$ by $\Delta d$ and integrating over $\Omega$, one arrives at
\begin{equation}\label{314}
\int(d_t+u \cdot \nabla d)\cdot \Delta d ~dx
=\int |\Delta d|^2 dx+\int |\nabla d|^2 d \cdot \Delta d ~dx.
\end{equation}
Integration by part and application of boundary condition \eqref{bc2} yield directly
\begin{equation}\label{315}
\int d_t \cdot \Delta d ~ dx
=\int_{\partial \Omega}d_t \cdot \frac{\partial d}{\partial n}d\sigma
-\frac{1}{2}\frac{d}{dt}\int |\nabla d|^2 dx
=-\frac{1}{2}\frac{d}{dt}\int |\nabla d|^2 dx.
\end{equation}
By virtue of the basic fact $|d|=1$, we find $\Delta d \cdot d=-|\nabla d|^2$.
Then, the combination of \eqref{314} and \eqref{315} gives
\begin{equation*}
\frac{1}{2}\frac{d}{dt}\int |\nabla d|^2 dx+\int |\Delta d|^2 dx
=\int (u \cdot \nabla) d\cdot \Delta d ~dx+\int |\nabla d|^4 dx,
\end{equation*}
which, together with \eqref{313}, yields directly
\begin{equation}\label{316}
\frac{1}{2}\frac{d}{dt}\int (|u|^2+|\nabla d|^2) dx
+\varepsilon \int |\nabla \times u|^2 dx
+ \int |\Delta d|^2 dx
\le \int |\nabla d|^4 dx
+ \varepsilon C_2 \int_{\partial \Omega}|u|^2 d\sigma.
\end{equation}
The trace theorem in Proposition \ref{prop2.3} implies
\begin{equation}\label{317}
|u|_{L^2(\partial \Omega)}^2\le \delta \|\nabla u\|^2+C_\delta \|u\|^2.
\end{equation}
The application of Proposition \ref{prop2.1} gives immediately
\begin{equation}\label{318}
\|\nabla \times u\|_{L^2}^2\ge c_1 \|\nabla u\|_{L^2}^2-c_2 \|u\|_{L^2}^2.
\end{equation}
Substituting \eqref{317} and \eqref{318} into \eqref{316}
and choosing $\delta $ small enough, one arrives at
\begin{equation*}
\frac{d}{dt}\int (|u|^2+|\nabla d|^2) dx
+\varepsilon \int |\nabla  u|^2 dx
+\int |\Delta d|^2 dx
\le \int |\nabla d|^4 dx
+C_2 \int  |u|^2 dx,
\end{equation*}
which, integrating over $[0, t]$, yields
\begin{equation*}
\begin{aligned}
&\underset{0\le \tau \le t}{\sup}\int (|u|^2+|\nabla d|^2)(\tau) dx
+\varepsilon \int_0^t \int |\nabla  u|^2 dx d\tau
+\int_0^t \int |\Delta d|^2 dx d\tau\\
&\le \int (|u_0|^2+|\nabla d_0|^2) dx
+\|\nabla d\|_{L^\infty_{x, t}}^2\int_0^t \int |\nabla d|^2 dx d\tau
+C_2 \int_0^t \int  |u|^2 dx d\tau.
\end{aligned}
\end{equation*}
Therefore, we complete the proof of the lemma \ref{lemma3.1}.
\end{proof}

However, the above basic energy estimation is insufficient to get the
vanishing viscosity limit. Some conormal derivative estimates are needed.

\begin{lemm}\label{lemma3.2}
For a smooth solution to \eqref{eq1} and \eqref{bc2}, it holds that
for every $m\in \mathbb{N}_+$ and $\varepsilon \in (0, 1]$
\begin{equation}\label{321}
\begin{aligned}
&\underset{0\le \tau \le t}{\sup}\|(u, \nabla d)\|_{\mathcal{H}^{m}}^2
+\varepsilon \int_0^t \|\nabla u\|_{\mathcal{H}^{m}}^2 d\tau
+\int_0^t \|\Delta d\|_{\mathcal{H}^{m}}^2 d\tau\\
&\le C_{m+2}\left\{\|(u_0, \nabla d_0)\|_{\mathcal{H}^{m}}^2
     +\delta \int_0^t \|\nabla \Delta d\|_{\mathcal{H}^{m-1}}^2d\tau
     +\delta \varepsilon^2 \int\|\nabla^2 u\|_{\mathcal{H}^{m-1}}^2 d\tau\right.\\
&\quad \quad \quad \quad
\left. +\int_0^t\!\!(\|\nabla^2 p_1\|_{\mathcal{H}^{m-1}}\|u\|_{\mathcal{H}^{m}}
+\varepsilon^{-1}\|\nabla p_2\|_{\mathcal{H}^{m-1}}^2)d\tau
+C_\delta[1+P(Q(t))]\!\int_0^t\! N_m(\tau)d\tau\right\},
\end{aligned}
\end{equation}
where the pressure $p$ is split as $p=p_1+p_2$, where $p_1$ and $p_2$
are defined by \eqref{3215} and \eqref{3216} respectively.
\end{lemm}
\begin{proof}
The case for $m=0$ is already proved in Lemma \ref{lemma3.1}. Assume that
\eqref{321} is proved for $k=m-1$. We shall prove that holds for $k=m \ge 1$.
Applying the operator $\mathcal{Z}^\alpha(|\alpha_0|+|\alpha_1|=m)$
to the equation \eqref{eq1-1}, we find
\begin{equation}\label{322}
\mathcal{Z}^\alpha u_t+u\cdot \nabla \mathcal{Z}^\alpha u+\mathcal{Z}^\alpha \nabla p
=-\varepsilon \mathcal{Z}^\alpha \nabla \times (\nabla \times u)
-\mathcal{Z}^\alpha(\nabla d \cdot \Delta d)+\mathcal{C}_1^\alpha,
\end{equation}
where
\begin{equation*}
\mathcal{C}_1^\alpha=-[\mathcal{Z}^\alpha, u\cdot \nabla]u.
\end{equation*}
Multiplying \eqref{322} by $\mathcal{Z}^\alpha u$, it is easy to deduce that
\begin{equation}\label{323}
\begin{aligned}
&\frac{1}{2}\frac{d}{dt}\int |\mathcal{Z}^\alpha u|^2 dx
+\int \mathcal{Z}^\alpha \nabla p \cdot \mathcal{Z}^\alpha u ~dx\\
&=-\varepsilon \int \mathcal{Z}^\alpha \nabla \times(\nabla \times u)\cdot \mathcal{Z}^\alpha u ~dx
-\int \mathcal{Z}^\alpha(\nabla d \cdot \Delta d)\cdot \mathcal{Z}^\alpha u~dx
+\int \mathcal{C}_1^\alpha \cdot \mathcal{Z}^\alpha u ~dx.
\end{aligned}
\end{equation}
Integrating by part and applying the boundary condition \eqref{bc2}, one arrives at
\begin{equation}\label{324}
\begin{aligned}
&-\varepsilon \int \mathcal{Z}^\alpha \nabla \times w \cdot \mathcal{Z}^\alpha u~dx\\
&=-\varepsilon \int (\nabla \times (\mathcal{Z}^\alpha w)+[\mathcal{Z}^\alpha, \nabla \times]w)
   \cdot \mathcal{Z}^\alpha u ~dx\\
&=-\varepsilon \int_{\partial \Omega} n \times \mathcal{Z}^\alpha w \cdot \mathcal{Z}^\alpha ud\sigma
-\varepsilon \int \mathcal{Z}^\alpha w \cdot \nabla \times \mathcal{Z}^\alpha u ~dx
-\varepsilon \int [\mathcal{Z}^\alpha, \nabla \times]w\cdot \mathcal{Z}^\alpha u ~dx\\
&=-\varepsilon \int |\nabla \times \mathcal{Z}^\alpha u|^2 dx
-\varepsilon \int [\mathcal{Z}^\alpha, \nabla \times]u \cdot \nabla \times \mathcal{Z}^\alpha u ~dx
-\varepsilon \int [\mathcal{Z}^\alpha, \nabla \times]w \cdot \mathcal{Z}^\alpha u ~dx\\
&\quad
-\varepsilon \int_{\partial \Omega}n \times \mathcal{Z}^\alpha w \cdot \mathcal{Z}^\alpha u ~d\sigma.
\end{aligned}
\end{equation}
Applying the Cauchy inequality, it is easy to deduce that
\begin{equation}\label{325}
\begin{aligned}
&-\varepsilon \int [\mathcal{Z}^\alpha, \nabla \times]u \cdot \nabla \times \mathcal{Z}^\alpha u ~dx
\le \frac{\varepsilon}{4}\|\nabla \times \mathcal{Z}^\alpha u\|^2
     +C\|\nabla u\|_{\mathcal{H}^{m-1}}^2,\\
&-\varepsilon \int [\mathcal{Z}^\alpha, \nabla \times]w \cdot \mathcal{Z}^\alpha u ~dx
\le \delta \varepsilon^2 \|\nabla^2 u\|_{\mathcal{H}^{m-1}}^2
    +C_\delta \|u\|_{\mathcal{H}^{m}}^2.
\end{aligned}
\end{equation}
To deal with the boundary terms involving $\mathcal{Z}^\alpha u$ with $\alpha_{13}=0$.
(If $\alpha_{13}\neq0, \mathcal{Z}^\alpha u=0$ on the boundary be definition).
With the help of trace theorem in Proposition \ref{prop2.3},
one has for $|\alpha_0|+|\alpha_1|=m$
\begin{equation*}
\begin{aligned}
&|n \times \mathcal{Z}^\alpha w|_{L^2(\partial \Omega)}\\
&\le C_{m+2}\left(|\partial_t^{\alpha_0} w|_{H^{|\alpha_1|-1}(\partial \Omega)}
             +|\partial_t^{\alpha_0} u|_{H^{|\alpha_1|}(\partial \Omega)}\right)\\
&\le C_{m+2} \|\nabla \partial_t^{\alpha_0}w\|_{|\alpha_1|-1}^{\frac{1}{2}}
             \|\partial_t^{\alpha_0}w\|_{|\alpha_1|-1}^{\frac{1}{2}}
     +C_{m+2} \|\nabla \partial_t^{\alpha_0}u\|_{|\alpha_1|}^{\frac{1}{2}}
             \|\partial_t^{\alpha_0}u\|_{|\alpha_1|}^{\frac{1}{2}}\\
&\quad +C_{m+2}\left(\|\nabla u\|_{\mathcal{H}^{m-1}}+\|u\|_{\mathcal{H}^{m}}\right)\\
&\le C_{m+2}(\|\nabla^2 u\|_{\mathcal{H}^{m-1}}^{\frac{1}{2}}
           \|\nabla u\|_{\mathcal{H}^{m-1}}^{\frac{1}{2}}
          +\|\nabla u\|_{\mathcal{H}^{m}}^{\frac{1}{2}}
           \|u\|_{\mathcal{H}^{m}}^{\frac{1}{2}}
          +\|\nabla u\|_{\mathcal{H}^{m-1}}+\|u\|_{\mathcal{H}^{m}})\\
&\le C_{m+2}(\|\nabla^2 u\|_{\mathcal{H}^{m-1}}
          +\|\nabla u\|_{\mathcal{H}^{m-1}}
          +\|\nabla u\|_{\mathcal{H}^{m}}
          +\|u\|_{\mathcal{H}^{m}}).
\end{aligned}
\end{equation*}
Note that we use the convention that $\|\cdot\|_{H^k}=0$, for $k<0$
here and below. Then
the term $|\partial_t^{\alpha_0} w|_{H^{|\alpha_1|-1}(\partial \Omega)}$
does not show up when $|\alpha_0|=m$.
Hence, it is easy to deduce that
\begin{equation}\label{326}
\begin{aligned}
&-\varepsilon \int_{\partial \Omega}n \times \mathcal{Z}^\alpha w \cdot \mathcal{Z}^\alpha u ~d\sigma\\
&\le \varepsilon |n\times \mathcal{Z}^\alpha w|_{L^2(\partial \Omega)}
         |\mathcal{Z}^\alpha u|_{L^2(\partial \Omega)}\\
&\le C_{m+2}\varepsilon
          (\|\nabla^2 u\|_{\mathcal{H}^{m-1}}
          +\|\nabla u\|_{\mathcal{H}^{m-1}}
          +\|\nabla u\|_{\mathcal{H}^{m}}
          +\|u\|_{\mathcal{H}^{m}})\\
&\quad \quad \times (\|\nabla u\|_{\mathcal{H}^{m}}^\frac{1}{2}+
         \|u\|_{\mathcal{H}^{m}}^\frac{1}{2})\|u\|_{\mathcal{H}^{m}}^\frac{1}{2}\\
&\le \delta \varepsilon \|\nabla u\|_{\mathcal{H}^{m}}^2
     +\delta \varepsilon^2 \|\nabla^2 u\|_{\mathcal{H}^{m-1}}^2
     +C_\delta C_{m+2}(\|\nabla u\|_{\mathcal{H}^{m-1}}^2+\|u\|_{\mathcal{H}^{m}}^2).
\end{aligned}
\end{equation}
Substituting \eqref{325} and \eqref{326} into \eqref{324}, we find
\begin{equation}\label{327}
\begin{aligned}
&-\varepsilon \int \mathcal{Z}^\alpha \nabla \times (\nabla \times u)\cdot \mathcal{Z}^\alpha u~ dx\\
&\le -\frac{3\varepsilon}{4}\int |\nabla \times \mathcal{Z}^\alpha u|^2 dx
      +\delta_1 \varepsilon \|\nabla u\|_{\mathcal{H}^{m}}^2
      +\delta_1 \varepsilon^2 \|\nabla^2 u\|_{\mathcal{H}^{m-1}}^2\\
&\quad  +C_{\delta_1} C_{m+2}(\|\nabla u\|_{\mathcal{H}^{m-1}}^2
        +\|u\|_{\mathcal{H}^{m}}^2).
\end{aligned}
\end{equation}
On the other hand, it follows from Proposition \ref{prop2.1} that
\begin{equation}\label{328}
\begin{aligned}
\|\nabla \mathcal{Z}^\alpha u\|^2
&\le  C(\|\nabla \times \mathcal{Z}^\alpha u\|^2+\|{\rm div} \mathcal{Z}^\alpha u\|^2
        +\|\mathcal{Z}^\alpha u\|^2
        +|\mathcal{Z}^\alpha u \cdot n|_{H^{\frac{1}{2}}(\partial \Omega)})\\
&\le  \|\nabla \times \mathcal{Z}^\alpha u\|^2+\|{\rm div} \mathcal{Z}^\alpha u\|^2
        +C_{m+2}(\|u\|_{\mathcal{H}^{m}}^2+\|\nabla u\|_{\mathcal{H}^{m-1}}^2),
\end{aligned}
\end{equation}
where we have used the fact that
\begin{equation*}
\begin{aligned}
&|\mathcal{Z}^\alpha u \cdot n|_{H^{\frac{1}{2}}(\partial \Omega)}
\le C_{m+2}|\partial^{\alpha_0}_t u |_{H^{m-\alpha_0-\frac{1}{2}}(\partial \Omega)}\\
&\le C_{m+2}(\|\nabla\partial^{\alpha_0}_t u\|_{m-\alpha_0-1}
       +\|\partial^{\alpha_0}_t u\|_{m-\alpha_0})\\
&\le C_{m+2}(\|\nabla u\|_{\mathcal{H}^{m-1}}+\|u\|_{\mathcal{H}^{m}}),
\end{aligned}
\end{equation*}
which is a consequence of boundary condition \eqref{bc2} and \eqref{21}.
The combination of \eqref{327} and \eqref{328} gives directly
\begin{equation}\label{329}
\begin{aligned}
&-\varepsilon \int \mathcal{Z}^\alpha \nabla \times (\nabla \times u)\cdot \mathcal{Z}^\alpha u dx\\
&\le -\frac{3\varepsilon}{4}\int |\nabla \mathcal{Z}^\alpha u|^2 dx
      +\delta_1 \varepsilon \|\nabla u\|_{\mathcal{H}^{m}}^2
      +\delta_1 \varepsilon^2 \|\nabla^2 u\|_{\mathcal{H}^{m-1}}^2\\
&\quad  +C_{\delta_1} C_{m+2}(\|\nabla u\|_{\mathcal{H}^{m-1}}^2
        +\|u\|_{\mathcal{H}^{m}}^2).
\end{aligned}
\end{equation}
Here we have used the divergence free condition of the velocity.
Integrating \eqref{323} over $[0, t]$ and substituting \eqref{329} into the resulting
identity, one arrives at
\begin{equation}\label{3210}
\begin{aligned}
&\frac{1}{2}\int |\mathcal{Z}^\alpha u|^2 dx
+\frac{3\varepsilon}{4}\int_0^t \int |\nabla \mathcal{Z}^\alpha u|^2 dx d\tau\\
&\le \frac{1}{2}\int |\mathcal{Z}^\alpha u_0|^2 dx
     +\delta_1 \varepsilon \int_0^t\|\nabla u\|_{\mathcal{H}^{m}}^2 d\tau
      +\delta_1 \varepsilon^2 \int_0^t\|\nabla^2 u\|_{\mathcal{H}^{m-1}}^2d\tau\\
&\quad   +C_{\delta_1} C_{m+2}\int_0^t(\|\nabla u\|_{\mathcal{H}^{m-1}}^2
        +\|u\|_{\mathcal{H}^{m}}^2)d\tau
     +\int_0^t \int \mathcal{C}_1^\alpha \cdot \mathcal{Z}^\alpha u ~dxd\tau\\
&\quad
 -\int_0^t \int \mathcal{Z}^\alpha \nabla p \cdot \mathcal{Z}^\alpha u ~dx d\tau
-\int_0^t \int \mathcal{Z}^\alpha(\nabla d \cdot \Delta d)\cdot \mathcal{Z}^\alpha u~dxd\tau.
\end{aligned}
\end{equation}
To deal with the term $\int_0^t \int C_1^\alpha \cdot \mathcal{Z}^\alpha u ~dxd\tau$.
Indeed, it is easy to deduce
\begin{equation}\label{3211}
\begin{aligned}
\int_0^t \int C_1^\alpha \cdot \mathcal{Z}^\alpha u dxd\tau
&=-\int_0^t \int \left([\mathcal{Z}^\alpha, u]\cdot \nabla u
         +u \cdot [\mathcal{Z}^\alpha, \nabla]u\right)\cdot\mathcal{Z}^\alpha udxd\tau\\
&\le \int_0^t\left(\|[\mathcal{Z}^\alpha, u]\cdot \nabla u\|
            +\|u \cdot [\mathcal{Z}^\alpha, \nabla]u\| \right) \|\mathcal{Z}^\alpha u\|d\tau.
\end{aligned}
\end{equation}
The help of Proposition \ref{prop2.2} yields directly
\begin{equation}\label{3212}
\begin{aligned}
&\int_0^t \|[\mathcal{Z}^\alpha, u]\cdot \nabla u\|^2 d\tau\\
&\le \sum_{|\beta|+|\gamma|=m,|\beta|\ge1}C_{\beta, \gamma}
     \int_0^t \|\mathcal{Z}^\beta u\cdot \mathcal{Z}^\gamma \nabla u\|^2 d\tau\\
&\le  \|\mathcal{Z}u\|_{L^\infty_{x,t}}^2\int_0^t \|\nabla u\|_{\mathcal{H}^{m-1}}^2 d\tau
+C\|\nabla u\|_{L^\infty_{x,t}}^2\int_0^t \|\mathcal{Z} u\|_{\mathcal{H}^{m-1}}^2 d\tau\\
&\le C_1 Q(t)\int_0^t (\|\nabla u\|_{\mathcal{H}^{m-1}}^2+\|u\|_{\mathcal{H}^{m}}^2) d\tau
\end{aligned}
\end{equation}
and
\begin{equation}\label{3213}
\begin{aligned}
\int_0^t \|u \cdot [\mathcal{Z}^\alpha, \nabla]u\|^2 d\tau
\le \sum_{|\beta|\le m-1}\|u\|_{L^\infty_{x,t}}^2
     \int_0^t \|\mathcal{Z}^\beta \nabla u\|^2 d\tau
\le C \| u\|_{L^\infty_{x,t}}^2\int_0^t \|\nabla u\|_{\mathcal{H}^{m-1}}^2 d\tau.
\end{aligned}
\end{equation}
The combination of \eqref{3211}-\eqref{3213} yields that
\begin{equation}\label{3214}
\int_0^t \int C_1^\alpha \cdot \mathcal{Z}^\alpha u~ dxd\tau
\le C_1 Q(t)\int_0^t(\|\nabla u\|_{\mathcal{H}^{m-1}}^2+\|u\|_{\mathcal{H}^{m}}^2)d\tau.
\end{equation}
To deal with the pressure term
$-\int_0^t \int \mathcal{Z}^\alpha \nabla p \cdot \mathcal{Z}^\alpha u ~dx d\tau$.
Since $\mathcal{Z}^\alpha u\cdot n$ does not vanish on the boundary, we follow the idea as
Masmoudi and Rousset \cite{Masmoudi-Rousset}
to split the pressure $p$ into the Euler part and Navier-Stokes part.
More precisely, the pressure $p$ is split as $p=p_1+p_2$, where $p_1$ is the
"Euler" part of the pressure which solves
\begin{equation*}
\Delta p_1=-{\rm div}(u\cdot \nabla u+\nabla d \cdot \Delta d),~x\in \Omega,~
\partial_n p_1=-(u\cdot \nabla u+\nabla d \cdot \Delta d)\cdot n,~x\in \partial \Omega,
\end{equation*}
and $p_2$ is the "Navier-Stokes part" which solves
\begin{equation*}
\Delta p_2=0,~x\in \Omega,
~\partial_n p_2=(\varepsilon \Delta u-u_t)\cdot n,~x\in \Omega,
\end{equation*}
which, together with the boundary condition \eqref{bc2}, gives directly
\begin{equation}\label{3215}
\Delta p_1=-{\rm div}(u\cdot \nabla u+\nabla d \cdot \Delta d),~x\in \Omega,~
\partial_n p_1=-u\cdot \nabla u\cdot n,~x\in \partial \Omega,
\end{equation}
and
\begin{equation}\label{3216}
\Delta p_2=0,~x\in \Omega,
~\partial_n p_2=\varepsilon \Delta u\cdot n,~x\in \Omega.
\end{equation}
Then, it is easy to deduce that
\begin{equation}\label{3217}
-\int_0^t \int \mathcal{Z}^\alpha \nabla p \cdot \mathcal{Z}^\alpha u dxd\tau
\le \int_0^t\|\nabla^2 p_1\|_{\mathcal{H}^{m-1}}\|u\|_{\mathcal{H}^{m}} d\tau
   +\left|\int_0^t \int \mathcal{Z}^\alpha \nabla p_2\cdot \mathcal{Z}^\alpha u dxd\tau\right|.
\end{equation}
For the last term, we have
\begin{equation}\label{3218}
\begin{aligned}
&\left|\int_0^t \int \mathcal{Z}^\alpha \nabla p_2\cdot \mathcal{Z}^\alpha u dxd\tau\right|\\
&=\left|\int_0^t\int \nabla \mathcal{Z}^\alpha p_2 \cdot \mathcal{Z}^\alpha u dxd\tau
+\int_0^t\int [\mathcal{Z}^\alpha, \nabla]p_2 \cdot \mathcal{Z}^\alpha u dxd\tau\right|\\
&\le \left|\int_0^t\int \nabla \mathcal{Z}^\alpha p_2 \cdot \mathcal{Z}^\alpha u dxd\tau\right|
+C_{m+1}\int_0^t\|\nabla p_2\|_{\mathcal{H}^{m-1}} \|u\|_{\mathcal{H}^{m}} d\tau.
\end{aligned}
\end{equation}
On the other hand, integrating by part and applying the divergence free of
velocity, we find
\begin{equation}\label{3219}
\begin{aligned}
&\left|\int_0^t \int \nabla \mathcal{Z}^\alpha p_2 \cdot \mathcal{Z}^\alpha u dxd\tau\right|\\
&=\left|\int_0^t\int_{\partial \Omega}\mathcal{Z}^\alpha p_2 \mathcal{Z}^\alpha u \cdot nd\sigma d\tau
+\int_0^t \int \mathcal{Z}^\alpha p_2 [\mathcal{Z}^\alpha, \nabla \cdot]u dx\tau\right|\\
&\le \left|\int_0^t\int_{\partial \Omega}\mathcal{Z}^\alpha p_2 \mathcal{Z}^\alpha u \cdot nd\sigma d\tau\right|
+\int_0^t\|\nabla p_2\|_{\mathcal{H}^{m-1}}\|\nabla u\|_{\mathcal{H}^{m-1}}d\tau.
\end{aligned}
\end{equation}
First, we deal with the boundary term when $\alpha_{13}=0$
(for $\alpha_{13}\neq 0, \mathcal{Z}^\alpha u=0$ on the boundary)
on the right hand side of \eqref{3219}.
If $|\alpha_0|=|\alpha|$, it is easy to deduce that
\begin{equation}\label{3220}
\int_0^t\int_{\partial \Omega}\mathcal{Z}^\alpha p_2
\mathcal{Z}^\alpha u \cdot nd\sigma d\tau=0.
\end{equation}
If $|\alpha_1|\ge 1$,
we integrating by parts along the boundary to obtain
\begin{equation}\label{3220-1}
\left|\int_0^t\int_{\partial \Omega}\mathcal{Z}^\alpha p_2 \mathcal{Z}^\alpha u \cdot nd\sigma d\tau\right|
\le C_2 \int_0^t |\partial_t^{\alpha_0}Z_y^{\beta} p_2|_{L^2(\partial \Omega)}
|\mathcal{Z}^\alpha u \cdot n|_{H^1(\partial \Omega)}d\tau,
\end{equation}
where $|\beta|= m-1-\alpha_0$. Applying the trace Theorem in Proposition \ref{prop2.3}, we find
\begin{equation}\label{3221}
|\partial_t^{\alpha_0}Z_y^{\beta} p_2|_{L^2(\partial \Omega)}
\le C \|\nabla \partial_t^{\alpha_0}Z_y^{\beta} p_2\|
      +\|\partial_t^{\alpha_0}Z_y^{\beta} p_2\|_1
\le C_1 \|\nabla p_2\|_{\mathcal{H}^{m-1}},
\end{equation}
and
\begin{equation}\label{3222}
|\mathcal{Z}^\alpha u \cdot n|_{H^1(\partial \Omega)}
\le C_{m+2}(\|\nabla u\|_{\mathcal{H}^{m}}+\|u\|_{\mathcal{H}^{m}}).
\end{equation}
Substituting \eqref{3220}-\eqref{3222} into \eqref{3219}, we find
\begin{equation*}
\left|\int_0^t \int \mathcal{Z}^\alpha \nabla p_2 \cdot \mathcal{Z}^\alpha u dxd\tau \right|
\le C_{m+2}\int_0^t\|\nabla p_2\|_{\mathcal{H}^{m-1}}
      (\|\nabla u\|_{\mathcal{H}^{m}}+\|u\|_{\mathcal{H}^{m}})d\tau,
\end{equation*}
which, together with \eqref{3217}, reads
\begin{equation}\label{3223}
\begin{aligned}
&\left|\int_0^t\int \mathcal{Z}^\alpha \nabla p \cdot \mathcal{Z}^\alpha u dxd\tau\right|\\
&\le C_{m+2}\!\!\int_0^t\!\! \|\nabla^2 p_1\|_{\mathcal{H}^{m-1}}\|u\|_{\mathcal{H}^{m}} d\tau
+C_{m+2}\!\!\int_0^t\!\!\|\nabla p_2\|_{\mathcal{H}^{m-1}}
      (\|\nabla u\|_{\mathcal{H}^{m}}+\|u\|_{\mathcal{H}^{m}})d\tau.
\end{aligned}
\end{equation}
We apply the Proposition \ref{prop2.2} to deduce that
\begin{equation*}\label{3225}
\begin{aligned}
\int_0^t \|\mathcal{Z}^\alpha(\nabla d\cdot \Delta d)\|^2 d\tau
\le \|\nabla d\|_{L^\infty_{x,t}}^2\int_0^t \|\Delta d\|_{\mathcal{H}^{m}}^2d\tau
    +\|\Delta d\|_{L^\infty_{x,t}}^2\int_0^t \|\nabla d\|_{\mathcal{H}^{m}}^2d\tau,
\end{aligned}
\end{equation*}
which implies that
\begin{equation}\label{3225}
\begin{aligned}
&\left|-\int_0^t \int \mathcal{Z}^\alpha(\nabla d \cdot \Delta d)\cdot \mathcal{Z}^\alpha u~dxd\tau\right|\\
&\le \delta_1 \int_0^t \|\Delta d\|_{\mathcal{H}^{m}}^2d\tau
+C_{\delta_1}(\|\nabla d\|_{L^\infty_{x,t}}^2+\|\Delta d\|_{L^\infty_{x,t}}^2)
\int_0^t(\|u\|_{\mathcal{H}^{m}}^2+\|\nabla d\|_{\mathcal{H}^{m}}^2)d\tau.
\end{aligned}
\end{equation}
Substituting \eqref{3214}, \eqref{3223} and \eqref{3225} into \eqref{3210}, it is easy to deduce
\begin{equation}\label{3227}
\begin{aligned}
&\frac{1}{2}\int |\mathcal{Z}^\alpha u|^2 dx
+\frac{3\varepsilon}{4}\int_0^t \int |\nabla \mathcal{Z}^\alpha u|^2 dx d\tau\\
&\le \frac{1}{2}\int |\mathcal{Z}^\alpha u_0|^2 dx
     +\delta_1 \varepsilon \int_0^t\|\nabla u\|_{\mathcal{H}^{m}}^2 d\tau
      +\delta_1 \int_0^t\|\Delta d\|_{\mathcal{H}^{m}}^2d\tau
      +\delta \varepsilon^2 \int_0^t\|\nabla^2 u\|_{\mathcal{H}^{m-1}}^2d\tau\\
&\quad
+C_\delta(1+Q(t))\int_0^t N_m (\tau)d\tau
+C_{m+2}C_\delta\int_0^t (\|\nabla^2 p_1\|_{\mathcal{H}^{m-1}}
      \|u\|_{\mathcal{H}^{m}} +\varepsilon^{-1}\|\nabla p_2\|_{\mathcal{H}^{m-1}}^2)d\tau.
\end{aligned}
\end{equation}
Applying the operator $\mathcal{Z}^\alpha \nabla(|\alpha_0|+|\alpha_1|=m$)
to the equation \eqref{eq1}$_2$, we find
\begin{equation}\label{3228}
\mathcal{Z}^\alpha \nabla d_t -\mathcal{Z}^\alpha \nabla \Delta d
=-\mathcal{Z}^\alpha \nabla(u\cdot \nabla d)
 +\mathcal{Z}^\alpha \nabla(|\nabla d|^2 d).
\end{equation}
Multiplying \eqref{3228} by $\mathcal{Z}^\alpha \nabla d$,
it is easy to deduce that
\begin{equation}\label{3229}
\begin{aligned}
&\frac{1}{2}\frac{d}{dt}\int |\mathcal{Z}^\alpha \nabla d|^2 dx
-\int \mathcal{Z}^\alpha \nabla \Delta d \cdot \mathcal{Z}^\alpha \nabla d~dx\\
&=-\int \mathcal{Z}^\alpha \nabla(u\cdot \nabla d)\cdot \mathcal{Z}^\alpha \nabla d~ dx
+\int \mathcal{Z}^\alpha \nabla(|\nabla d|^2 d)\cdot \mathcal{Z}^\alpha \nabla d ~dx.
\end{aligned}
\end{equation}
Integrating by part, we find
\begin{equation*}
\begin{aligned}
&-\int \mathcal{Z}^\alpha \nabla \Delta d \cdot \mathcal{Z}^\alpha \nabla d~dx\\
&=-\int \nabla \mathcal{Z}^\alpha  \Delta d \cdot \mathcal{Z}^\alpha \nabla d~dx
  -\int [\mathcal{Z}^\alpha, \nabla] \Delta d \cdot \mathcal{Z}^\alpha \nabla d~dx\\
&=-\int_{\partial \Omega} \mathcal{Z}^\alpha  \Delta d \cdot \mathcal{Z}^\alpha \nabla d
    \cdot n~d\sigma
  +\int \mathcal{Z}^\alpha  \Delta d \cdot {\rm div}(\mathcal{Z}^\alpha \nabla d)~dx
  -\int [\mathcal{Z}^\alpha, \nabla] \Delta d \cdot \mathcal{Z}^\alpha \nabla d~dx\\
&=-\int_{\partial \Omega} \mathcal{Z}^\alpha  \Delta d \cdot \mathcal{Z}^\alpha \nabla d
    \cdot n~d\sigma
  +\int |{\rm div}(\mathcal{Z}^\alpha \nabla d)|^2dx
+\int [\mathcal{Z}^\alpha, {\rm div}]\nabla d \cdot {\rm div}(\mathcal{Z}^\alpha \nabla d)~dx\\
&\quad -\int [\mathcal{Z}^\alpha, \nabla] \Delta d \cdot \mathcal{Z}^\alpha \nabla d~dx.
\end{aligned}
\end{equation*}
This, together with \eqref{3229}, reads
\begin{equation}\label{3230}
\begin{aligned}
&\frac{1}{2}\int |\mathcal{Z}^\alpha \nabla d(t)|^2 dx
+\int_0^t \int |{\rm div}(\mathcal{Z}^\alpha \nabla d)|^2dxd\tau\\
&=\frac{1}{2}\int |\mathcal{Z}^\alpha \nabla d_0|^2 dx
-\int_0^t \!\!\int \mathcal{Z}^\alpha \nabla(u\cdot \nabla d)\cdot \mathcal{Z}^\alpha \nabla d~ dxd\tau\\
&\quad
+\int_0^t\!\!\int \mathcal{Z}^\alpha \nabla(|\nabla d|^2 d)\cdot \mathcal{Z}^\alpha \nabla d ~dxd\tau
-\int_0^t\int [\mathcal{Z}^\alpha, {\rm div}]\nabla d \cdot {\rm div}(\mathcal{Z}^\alpha \nabla d)~dxd\tau\\
&\quad
+\int_0^t\int [\mathcal{Z}^\alpha, \nabla] \Delta d \cdot \mathcal{Z}^\alpha \nabla d~dxd\tau
+\int_0^t\int_{\partial \Omega} \mathcal{Z}^\alpha  \Delta d \cdot \mathcal{Z}^\alpha \nabla d\cdot n~d\sigma d\tau.
\end{aligned}
\end{equation}
Integrating by part, it is easy to deduce
\begin{equation}\label{3231}
\begin{aligned}
&-\int_0^t \int \mathcal{Z}^\alpha \nabla(u\cdot \nabla d)\cdot \mathcal{Z}^\alpha \nabla d~ dx d\tau\\
&=-\int_0^t\int  \nabla \mathcal{Z}^\alpha(u\cdot \nabla d)\cdot \mathcal{Z}^\alpha \nabla d~ dxd\tau
  -\int_0^t\int  [\mathcal{Z}^\alpha, \nabla](u\cdot \nabla d)\cdot \mathcal{Z}^\alpha \nabla d~ dxd\tau\\
&=-\int_0^t\int_{\partial \Omega } \mathcal{Z}^\alpha(u\cdot \nabla d)\cdot
    \mathcal{Z}^\alpha \nabla d\cdot n~ d\sigma d\tau
-\int_0^t\int \mathcal{Z}^\alpha(u\cdot \nabla d)\cdot{\rm div} (\mathcal{Z}^\alpha \nabla d)~ dxd\tau\\
&\quad -\int_0^t\int[\mathcal{Z}^\alpha, \nabla](u\cdot \nabla d)\cdot \mathcal{Z}^\alpha \nabla d~ dxd\tau.
\end{aligned}
\end{equation}
To estimate the boundary term on the right hand side of \eqref{3231}.
If $|\alpha_0|=m$ or $|\alpha_{13}|\ge 1$, we obtain
$$
-\int_0^t\int_{\partial \Omega } \mathcal{Z}^\alpha(u\cdot \nabla d)\cdot
    \mathcal{Z}^\alpha \nabla d\cdot n~ d\sigma d\tau=0.
$$
For $|\beta|=m-1-\alpha_0(|\alpha_0|\le m-1)$,
we integrating by part along the boundary to deduce that
\begin{equation}\label{3232}
-\int_0^t\int_{\partial \Omega } \mathcal{Z}^\alpha(u\cdot \nabla d)\cdot
    \mathcal{Z}^\alpha \nabla d\cdot n~ d\sigma d\tau
\le \int_0^t |\partial_t^{\alpha_0}Z_y^\beta (u\cdot \nabla d)|_{L^2{(\partial \Omega)}}
      |\mathcal{Z}^\alpha \nabla d\cdot n|_{H^1(\partial \Omega)}d\tau.
\end{equation}
Applying the trace theorem in Proposition \ref{prop2.3} and Proposition \ref{prop2.2},
one arrives at
\begin{equation}\label{3233}
\begin{aligned}
&\int_0^t |\partial_t^{\alpha_0}Z_y^\beta (u\cdot \nabla d)|_{L^2{(\partial \Omega)}}^2d\tau\\
&\le C\int_0^t (\|\nabla \partial_t^{\alpha_0} (u\cdot \nabla d)\|_{m-1-\alpha_0}^2
              +\| \partial_t^{\alpha_0} (u\cdot \nabla d)\|_{m-1-\alpha_0}^2)d\tau\\
&\le C\int_0^t(\|\nabla(u\cdot \nabla d)\|_{\mathcal{H}^{m-1}}^2
              +\|u\cdot \nabla d\|_{\mathcal{H}^{m-1}}^2)d\tau\\
&\le CQ(t)\int_0^t(\|(u,\nabla d)\|_{\mathcal{H}^{m-1}}^2
     +\|\nabla(u,\nabla d)\|_{\mathcal{H}^{m-1}}^2)d\tau.
\end{aligned}
\end{equation}
With the help of boundary condition \eqref{bc2}
and trace theorem in Proposition \ref{prop2.3}, we find
\begin{equation}\label{3234}
\int_0^t |\mathcal{Z}^\alpha \nabla d\cdot n|_{H^1(\partial \Omega)}^2d\tau
\le C_{m+2}\int_0^t (\|\nabla^2 d\|_{\mathcal{H}^{m}}^2
+\|\nabla d\|_{\mathcal{H}^{m}}^2)d\tau.
\end{equation}
The combination of \eqref{3231}-\eqref{3234} and Young inequality, it is easy to deduce that
\begin{equation}\label{3234}
\begin{aligned}
&-\int_0^t\int_{\partial \Omega } \mathcal{Z}^\alpha(u\cdot \nabla d)\cdot
    \mathcal{Z}^\alpha \nabla d\cdot n~ d\sigma d\tau\\
&\le \delta_1  \!\int_0^t \! \|\nabla^2 d\|_{\mathcal{H}^{m}}^2 d\tau
     +C_{\delta_1}C_{m+2}(1+Q(t))
     \!\ \! \int_0^t \! \!(\|(u,\nabla d)\|_{\mathcal{H}^{m}}^2+\|\nabla(u,\nabla d)\|_{\mathcal{H}^{m-1}}^2)d\tau.
\end{aligned}
\end{equation}
Applying the Young inequality and the Proposition \ref{prop2.2},
one arrives at
\begin{equation}\label{3235}
\begin{aligned}
&-\int_0^t\int \mathcal{Z}^\alpha(u\cdot \nabla d)\cdot{\rm div} (\mathcal{Z}^\alpha \nabla d)~ dxd\tau\\
&\le \delta_1\int_0^t\|{\rm div} (\mathcal{Z}^\alpha \nabla d)\|^2d\tau
     +C_{\delta_1} \|u\|_{L^\infty_{x,t}}^2\int_0^t \|\nabla d\|_{\mathcal{H}^{m}}^2 d\tau
+C_{\delta_1} \|\nabla d\|_{L^\infty_{x,t}}^2\int_0^t \|u\|_{\mathcal{H}^{m}}^2 d\tau\\
&\le \delta_1\int_0^t\|{\rm div} (\mathcal{Z}^\alpha \nabla d)\|^2d\tau
     +C_{\delta_1} C_1 Q(t)
     \int_0^t (\|u\|_{\mathcal{H}^{m}}^2
      +\|\nabla d\|_{\mathcal{H}^{m}}^2) d\tau\\
\end{aligned}
\end{equation}
and
\begin{equation}\label{3236}
\begin{aligned}
&-\int_0^t\int[\mathcal{Z}^\alpha, \nabla](u\cdot \nabla d)\cdot \mathcal{Z}^\alpha \nabla d~ dxd\tau\\
&\le  \sum_{|\beta|\le m-1}\int_0^t \|\mathcal{Z}^\beta(\nabla u\cdot \nabla d+
        u\cdot \nabla^2 d)\| \|\mathcal{Z}^\alpha \nabla d\|dx\\
&\le C\int_0^t \|\nabla d\|_{\mathcal{H}^{m}}^2 d\tau
      +C(\|\nabla u\|_{L^\infty_{x,t}}^2+\|\nabla d\|_{L^\infty_{x,t}}^2)
      \int_0^t (\|\nabla u\|_{\mathcal{H}^{m-1}}^2+\|\nabla d\|_{\mathcal{H}^{m-1}}^2)d\tau\\
&\quad  +C(\|u\|_{L^\infty_{x,t}}^2+\|\nabla^2 d\|_{L^\infty_{x,t}}^2)
      \int_0^t (\|u\|_{\mathcal{H}^{m-1}}^2+\|\nabla^2 d\|_{\mathcal{H}^{m-1}}^2)d\tau\\
&\le C(1+Q(t))
      \int_0^t (\|u\|_{\mathcal{H}^{m-1}}^2+\|\nabla u\|_{\mathcal{H}^{m-1}}^2
       +\|\nabla d\|_{\mathcal{H}^{m}}^2
      +\|\nabla^2 d\|_{\mathcal{H}^{m-1}}^2)d\tau.
\end{aligned}
\end{equation}
Substituting \eqref{3234}-\eqref{3236} into \eqref{3231}, we obtain
\begin{equation}\label{3237}
\begin{aligned}
&-\int_0^t \int \mathcal{Z}^\alpha \nabla(u\cdot \nabla d)\cdot \mathcal{Z}^\alpha \nabla d~ dx d\tau\\
&\le \delta_1 \int_0^t \|\nabla^2 d\|_{\mathcal{H}^{m}}^2 d\tau
     +C_{\delta_1}C_1(1+Q(t))
     \int_0^t(\|(u,\nabla d)\|_{\mathcal{H}^{m}}^2+\|\nabla(u,\nabla d)\|_{\mathcal{H}^{m-1}}^2)d\tau.
\end{aligned}
\end{equation}
On the other hand, it is complicated to deal with the term
$\int_0^t\int \mathcal{Z}^\alpha \nabla(|\nabla d|^2 d)\cdot \mathcal{Z}^\alpha \nabla d ~dxd\tau$. Indeed, by integrating by part, one arrives at
\begin{equation}\label{3238}
\begin{aligned}
&\int_0^t\int \mathcal{Z}^\alpha \nabla(|\nabla d|^2 d)\cdot \mathcal{Z}^\alpha \nabla d ~dxd\tau\\
&=\int_0^t\int \nabla \mathcal{Z}^\alpha (|\nabla d|^2 d)\cdot \mathcal{Z}^\alpha \nabla d ~dxd\tau
+\int_0^t\int [\mathcal{Z}^\alpha, \nabla ](|\nabla d|^2 d)\cdot \mathcal{Z}^\alpha \nabla d ~dxd\tau\\
&=
-\int_0^t\int \mathcal{Z}^\alpha (|\nabla d|^2 d)\cdot {\rm div}(\mathcal{Z}^\alpha \nabla d) ~dxd\tau
+\int_0^t\int [\mathcal{Z}^\alpha, \nabla ](|\nabla d|^2 d)\cdot \mathcal{Z}^\alpha \nabla d ~dxd\tau\\
&\quad
+\int_0^t\int_{\partial \Omega} \mathcal{Z}^\alpha (|\nabla d|^2 d)
  \cdot \mathcal{Z}^\alpha \nabla d \cdot n~d\sigma d\tau.
\end{aligned}
\end{equation}
It is easy to deduce that
\begin{equation}\label{3239}
\begin{aligned}
&-\int_0^t\int \mathcal{Z}^\alpha (|\nabla d|^2 d)\cdot {\rm div}(\mathcal{Z}^\alpha \nabla d) ~dxd\tau\\
&=-\underset{|\beta| \ge 1}{\sum}\int_0^t\int\mathcal{Z}^\gamma(|\nabla d|^2)\mathcal{Z}^\beta d\cdot {\rm div}(\mathcal{Z}^\alpha \nabla d) ~dxd\tau
-\int_0^t\int \mathcal{Z}^\alpha(|\nabla d|^2)d \cdot {\rm div}(\mathcal{Z}^\alpha \nabla d) ~dxd\tau.
\end{aligned}
\end{equation}
By virtue of the Proposition \ref{prop2.2}, we obtain
\begin{equation}\label{3240}
\begin{aligned}
&\underset{|\beta| \ge 1}{\sum}\int_0^t
\|\mathcal{Z}^\gamma(|\nabla d|^2)\mathcal{Z}^\beta d\|^2d\tau\\
&\le  \|\mathcal{Z}d\|_{L^\infty_{x,t}}^2\int_0^t \||\nabla d|^2\|_{\mathcal{H}^{m-1}}^2 d\tau
      +\||\nabla d|^2\|_{L^\infty_{x,t}}^2
      \int_0^t \| \mathcal{Z} d\|_{\mathcal{H}^{m-1}}^2 d\tau\\
&\le \|\mathcal{Z}d\|_{L^\infty_{x,t}}^2\|\nabla d\|_{L^\infty_{x,t}}^2
     \int_0^t \|\nabla d\|_{\mathcal{H}^{m-1}}^2 d\tau
      +\|\nabla d\|_{L^\infty_{x,t}}^4 \int_0^t \|\mathcal{Z}d\|_{\mathcal{H}^{m-1}}^2d\tau\\
&\le \|\nabla d\|_{L^\infty_{x,t}}^4 \int_0^t \|\partial_td\|_{\mathcal{H}^{m-1}}^2d\tau
     +C_1(1+P(Q(t)))\int_0^t N_m(\tau)d\tau.
\end{aligned}
\end{equation}
By virtue of equation \eqref{eq1}$_2$, we find
\begin{equation}\label{3242}
\begin{aligned}
\int_0^t \|\partial_td\|_{\mathcal{H}^{m-1}}^2d\tau
&\le \|u\|_{L^\infty_{x,t}}^2\int_0^t \|\nabla d\|_{\mathcal{H}^{m-1}}^2d\tau
     +\|\nabla d\|_{L^\infty_{x,t}}^2\int_0^t \|u\|_{\mathcal{H}^{m-1}}^2d\tau\\
&\quad +\int_0^t \|\Delta d\|_{\mathcal{H}^{m-1}}^2d\tau
       +\int_0^t \||\nabla d|^2 d\|_{\mathcal{H}^{m-1}}^2 d\tau.
\end{aligned}
\end{equation}
By virtue of the Proposition \ref{prop2.2}, we obtain
\begin{equation}\label{3243}
\begin{aligned}
\int_0^t \||\nabla d|^2 d\|_{\mathcal{H}^{m-1}}^2 d\tau
&\le \sum_{|\gamma|\ge 1,|\beta|+|\gamma|\le m-1}
     \int_0^t \|\mathcal{Z}^\beta(|\nabla d|^2)\mathcal{Z}^\gamma d\|^2 d\tau
     +\int_0^t \||\nabla d|^2\|_{\mathcal{H}^{m-1}}^2d\tau\\
&\le \|\mathcal{Z}d\|_{L^\infty_{x,t}}^2 \|\nabla d\|_{L^\infty_{x,t}}^2
     \int_0^t \|\nabla d\|_{\mathcal{H}^{m-2}}^2d\tau
     +\|\nabla d\|_{L^\infty_{x,t}}^4\int_0^t \|\nabla d\|_{\mathcal{H}^{m-2}}^2 d\tau\\
&\quad +\|\nabla d\|_{L^\infty_{x,t}}^4\int_0^t \|\partial_t d\|_{\mathcal{H}^{m-2}}^2 d\tau
       +\|\nabla d\|_{L^\infty_{x,t}}^2\int_0^t \|\nabla d\|_{\mathcal{H}^{m-1}}^2 d\tau.
\end{aligned}
\end{equation}
Substituting \eqref{3243} into \eqref{3242}, we obtain
\begin{equation}\label{3244}
\begin{aligned}
\int_0^t \|d_t\|_{\mathcal{H}^{m-1}}^2d\tau
&\le \|\nabla d\|_{L^\infty_{x,t}}^4\int_0^t \|\partial_t d\|_{\mathcal{H}^{m-2}}^2 d\tau
    +\int_0^t \|\Delta d\|_{\mathcal{H}^{m-1}}^2d\tau\\
&\quad    +C_1(1+P(Q(t)))\int_0^t (\|u\|_{\mathcal{H}^{m-1}}^2
    +\|\nabla d\|_{\mathcal{H}^{m-1}}^2)d\tau.
\end{aligned}
\end{equation}
On the other hand, it is easy to deduce that
\begin{equation}\label{3245}
\int_0^t \|d_t\|^2 d\tau
\le \int_0^t \|\Delta d\|^2 d\tau+(1+\|\nabla d\|_{L^\infty_{x,t}}^2)
    \int_0^t (\|u\|_{L^2}^2+\|\nabla d\|_{L^2}^2)d\tau.
\end{equation}
The combination  of  \eqref{3244} and \eqref{3245} yields directly
\begin{equation}\label{3246}
\begin{aligned}
\int_0^t  \|d_t\|_{\mathcal{H}^{m-1}}^2 d\tau
&\le C_1(1+P(Q(t)))\int_0^t\!\! \|\Delta d\|_{\mathcal{H}^{m-1}}^2 d\tau\\
&\quad +C_1(1+P(Q(t)))\int_0^t \!\!(\|u\|_{\mathcal{H}^{m-1}}^2
    +\|\nabla d\|_{\mathcal{H}^{m-1}}^2)d\tau,
\end{aligned}
\end{equation}
which, together with \eqref{3240}, gives directly
\begin{equation}\label{3247}
\begin{aligned}
&\left|-\underset{|\beta| \ge 1}{\sum}\int_0^t\int\mathcal{Z}^\gamma(|\nabla d|^2)\mathcal{Z}^\beta d\cdot {\rm div}(\mathcal{Z}^\alpha \nabla d) ~dxd\tau\right|\\
&\le \delta_1 \int_0^t\|{\rm div}(\mathcal{Z}^\alpha \nabla d)\|d\tau
+C_{\delta_1}C_1(1+P(Q(t)))\int_0^t N_m(\tau)d\tau.
\end{aligned}
\end{equation}
In view of the Proposition \ref{prop2.2} and Cauchy inequality, we obtain
\begin{equation}\label{3248}
\begin{aligned}
&\left|-\int_0^t\int \mathcal{Z}^\alpha(|\nabla d|^2)d \cdot {\rm div}(\mathcal{Z}^\alpha \nabla d) ~dxd\tau\right|\\
&\le \delta_1 \int_0^t\|{\rm div}(\mathcal{Z}^\alpha \nabla d)\|d\tau
      +C_{\delta_1}\|\nabla d\|_{L^\infty_{x,t}}^2\int_0^t \|\nabla d\|_{\mathcal{H}^{m}}^2d\tau.
\end{aligned}
\end{equation}
Then combination of  \eqref{3247} and \eqref{3248} yields immediately
\begin{equation}\label{3249}
\begin{aligned}
&-\int_0^t\int \mathcal{Z}^\alpha (|\nabla d|^2 d)\cdot {\rm div}(\mathcal{Z}^\alpha \nabla d) ~dxd\tau\\
&\le \delta_1 \int_0^t\|{\rm div}(\mathcal{Z}^\alpha \nabla d)\|d\tau
+C_1C_{\delta_1}(1+P(Q(t)))\int_0^t N_m(\tau)d\tau.
\end{aligned}
\end{equation}
On the other hand, we find that
\begin{equation}\label{3250}
\begin{aligned}
&\int_0^t \int [\mathcal{Z}^\alpha, \nabla ](|\nabla d|^2 d)\cdot \mathcal{Z}^\alpha \nabla d ~dxd\tau\\
&=\sum_{|\beta|\le m-1}\int_0^t \int
   d \cdot \mathcal{Z}^{\beta} (\nabla d\cdot \nabla^2 d)
   \cdot \mathcal{Z}^\alpha \nabla d ~dxd\tau\\
&\quad +\sum_{|\beta|+|\gamma|\le m-1}^{|\beta|\ge 1|}\int_0^t\int
   \mathcal{Z}^\beta d \cdot \mathcal{Z}^\gamma (\nabla d\cdot \nabla^2 d)
   \cdot \mathcal{Z}^\alpha \nabla d ~dxd\tau\\
&\quad  +\underset{|\beta|+|\gamma|\le m-1}{\sum} \int_0^t\int
   \mathcal{Z}^\beta (|\nabla d|^2)\mathcal{Z}^\gamma \nabla d
   \cdot \mathcal{Z}^\alpha \nabla d ~dxd\tau\\
&=I_1+I_2+I_3.
\end{aligned}
\end{equation}
In view of the Proposition \ref{prop2.2}, we find
\begin{equation}\label{3251}
I_1\le  \|\nabla d\|_{L^\infty_{x,t}}^2\int_0^t \|\nabla^2 d\|_{\mathcal{H}^{m-1}}^2d\tau
        +\|\nabla^2 d\|_{L^\infty_{x,t}}^2\int_0^t \|\nabla d\|_{\mathcal{H}^{m-1}}^2d\tau
        +\int_0^t \|\nabla d\|_{\mathcal{H}^{m}}^2 d\tau.
\end{equation}
and
\begin{equation}\label{3252}
\begin{aligned}
I_2
&\le  C\|\nabla d\|_{W^{1,\infty}_{x,t}}^4
       \int_0^t \|\mathcal{Z} d\|_{\mathcal{H}^{m-2}}^2d\tau
      +C\|\mathcal{Z} d\|_{L^\infty_{x,t}}^2\|\nabla d\|_{L^\infty_{x,t}}^2
        \int_0^t \|\nabla^2 d\|_{\mathcal{H}^{m-2}}^2d\tau\\
&\quad +C\|\mathcal{Z} d\|_{L^\infty_{x,t}}^2\|\nabla^2 d\|_{L^\infty}^2
        \int_0^t \|\nabla d\|_{\mathcal{H}^{m-2}}^2d\tau
       +C\int_0^t \|\nabla d\|_{\mathcal{H}^{m}}^2 d\tau\\
&\le C_1(1+P(Q(t)))\int_0^t N_m(\tau)d\tau,
\end{aligned}
\end{equation}
where we have used the estimate \eqref{3246}.
Similarly, it is easy to deduce that
\begin{equation}\label{3253}
I_3
\le C \|\nabla d\|_{L^\infty_{x,t}}^4 \int_0^t \|\nabla d\|_{\mathcal{H}^{m-1}}^2d\tau
+C\int_0^t \|\nabla d\|_{\mathcal{H}^{m}}^2 d\tau.
\end{equation}
Substituting \eqref{3251}-\eqref{3223} into \eqref{3250}, one arrives at
\begin{equation}\label{3254}
\int_0^t\int [\mathcal{Z}^\alpha, \nabla ](|\nabla d|^2 d)\cdot \mathcal{Z}^\alpha \nabla d ~dxd\tau
\le C_1(1+P(Q(t)))\int_0^t N_m(\tau)d\tau.
\end{equation}
Deal with the boundary term on the right hand side of \eqref{3238}.
If $|\alpha_0|=m$ or $|\alpha_{13}|\ge 1$, we obtain
\begin{equation}\label{3255-1}
\int_0^t\int_{\partial \Omega} \mathcal{Z}^\alpha (|\nabla d|^2 d)
\cdot \mathcal{Z}^\alpha \nabla d \cdot n~d\sigma d\tau
=0.
\end{equation}
Indeed,
it is easy to deduce that for $|\beta|=m-1-\alpha_0$
\begin{equation}\label{3255}
\int_0^t\int_{\partial \Omega} \mathcal{Z}^\alpha (|\nabla d|^2 d)
  \cdot \mathcal{Z}^\alpha \nabla d \cdot n~d\sigma d\tau
\le \int_0^t |\partial_t^{\alpha_0}Z_y^{\beta} (|\nabla d|^2 d)|_{L^2(\partial \Omega)}
    |\mathcal{Z}^\alpha \nabla d \cdot n|_{H^1(\partial \Omega)}d\tau.
\end{equation}
By virtue of the trace theorem in Proposition \ref{prop2.3},
we find for $|\beta|=m-1-\alpha_0$
\begin{equation}\label{3256}
\begin{aligned}
&|\partial_t^{\alpha_0}Z_y^{\beta} (|\nabla d|^2 d)|_{L^2(\partial \Omega)}^2\\
&\le C\|\nabla \partial_t^{\alpha_0}(|\nabla d|^2 d)\|_{m-1-\alpha_0}
       \|\partial_t^{\alpha_0}(|\nabla d|^2 d)\|_{m-1-\alpha_0}
     +C\|\partial_t^{\alpha_0}(|\nabla d|^2 d)\|_{m-1-\alpha_0}^2\\
&\le C\|\nabla(|\nabla d|^2 d)\|_{\mathcal{H}^{m-1}}
       \||\nabla d|^2 d\|_{\mathcal{H}^{m-1}}
    +C\||\nabla d|^2 d\|_{\mathcal{H}^{m-1}}^2,
\end{aligned}
\end{equation}
and
\begin{equation}\label{3257}
|\mathcal{Z}^\alpha \nabla d \cdot n|_{H^1(\partial \Omega)}^2
\le C_{m+2}(\|\nabla^2 d\|_{\mathcal{H}^{m}}^2+\|\nabla d\|_{\mathcal{H}^{m}}^2).
\end{equation}
On the other hand,  we obtain just following the idea as \eqref{3239}
and \eqref{3250} that
\begin{equation}\label{3258}
\int_0^t\||\nabla d|^2 d\|_{\mathcal{H}^{m-1}}^2 d\tau
\le C(1+P(Q(t)))\int_0^t N_m(\tau)d\tau
\end{equation}
and
\begin{equation}\label{3259}
\int_0^t\|\nabla(|\nabla d|^2 d)\|_{\mathcal{H}^{m-1}}^2 d\tau
\le C(1+P(Q(t)))\int_0^t N_m(\tau)d\tau.
\end{equation}
The combination of \eqref{3255}-\eqref{3259} gives directly
\begin{equation}\label{3260}
\begin{aligned}
&\int_0^t\int_{\partial \Omega} \mathcal{Z}^\alpha (|\nabla d|^2 d)
  \cdot \mathcal{Z}^\alpha \nabla d \cdot n~d\sigma d\tau\\
&\le \delta_1 \int_0^t\|\nabla^2 d\|_{\mathcal{H}^{m}}^2d\tau
     +C_{\delta_1}C_{m+2}(1+P(Q(t)))\int_0^t N_m(\tau)d\tau.
\end{aligned}
\end{equation}
Substituting \eqref{3247}, \eqref{3254} and \eqref{3260} into \eqref{3238}, we attains
\begin{equation}\label{3261}
\begin{aligned}
&\int_0^t\int \mathcal{Z}^\alpha \nabla(|\nabla d|^2 d)\cdot \mathcal{Z}^\alpha \nabla d ~dxd\tau\\
&\le \delta_1 \!\int_0^t \!\|{\rm div}(\mathcal{Z}^\alpha \nabla d)\|^2 d\tau
  +\delta_1 \!\int_0^t \!\|\nabla^2 d\|_{\mathcal{H}^{m}}^2d\tau
     +C_{\delta_1}C_{m+2}(1+P(Q(t)))\int_0^t N_m(\tau)d\tau.
\end{aligned}
\end{equation}
In view of the Cauchy inequality, it is easy to deduce that
\begin{equation}\label{3262}
-\int_0^t\int [\mathcal{Z}^\alpha, {\rm div}]\nabla d \cdot {\rm div}(\mathcal{Z}^\alpha \nabla d)~dxd\tau
\le \delta_1 \int_0^t\|{\rm div}(\mathcal{Z}^\alpha \nabla d)\|^2d\tau
     +C_{\delta_1} \int_0^t\|\nabla^2 d\|^2_{\mathcal{H}^{m-1}}d\tau,
\end{equation}
and
\begin{equation}\label{3263}
\int_0^t\int [\mathcal{Z}^\alpha, \nabla] \Delta d \cdot \mathcal{Z}^\alpha \nabla d~dx
\le \delta \int_0^t\|\nabla \Delta d\|_{\mathcal{H}^{m-1}}^2d\tau
    +C_\delta \int_0^t\|\nabla d\|_{\mathcal{H}^{m}}^2 d\tau.
\end{equation}
Finally, we deal with the boundary term on the right hand side of \eqref{3230}.
If $|\alpha_0|=m$ or $|\alpha_{13}|\ge 1$, we obtain
\begin{equation}\label{3264-1}
\int_0^t \int_{\partial \Omega}\mathcal{Z}^\alpha \Delta d
  \cdot \mathcal{Z}^\alpha \nabla d\cdot n d\sigma d\tau
=0.
\end{equation}
On the other hand, integrating by part along the boundary, we have for $|\beta|=m-1-\alpha_0$
\begin{equation}\label{3264}
\int_0^t \int_{\partial \Omega}\mathcal{Z}^\alpha \Delta d
  \cdot \mathcal{Z}^\alpha \nabla d\cdot n d\sigma d\tau
\le C\int_0^t |\partial_t^{\alpha_0}Z_y^\beta \Delta d|_{L^2(\partial \Omega)}
    |\mathcal{Z}^\alpha \nabla d\cdot n|_{H^1(\partial \Omega)}d\tau.
\end{equation}
By virtue of the trace theorem in Proposition \ref{prop2.3}, one arrives at
\begin{equation}\label{3265}
\begin{aligned}
|\partial_t^{\alpha_0}Z_y^\beta \Delta d|_{L^2(\partial \Omega)}
&\le C(\|\nabla \partial_t^{\alpha_0} Z_y^\beta \Delta d\|^{\frac{1}{2}}
      +\|\partial_t^{\alpha_0} Z_y^\beta \Delta d\|^{\frac{1}{2}})
      \| \partial_t^{\alpha_0} Z_y^\beta \Delta d\|^{\frac{1}{2}}\\
&\le C(\|\nabla \Delta d\|_{\mathcal{H}^{m-1}}^{\frac{1}{2}}
      +\|\Delta d\|_{\mathcal{H}^{m-1}}^{\frac{1}{2}})
      \|\Delta d\|_{\mathcal{H}^{m-1}}^{\frac{1}{2}}.
\end{aligned}
\end{equation}
Similarly, with the help of boundary condition
\eqref{bc2} and trace theorem, one arrives at
\begin{equation}\label{3266}
|\mathcal{Z}^\alpha \nabla d\cdot n|_{H^1(\partial \Omega)}
\le C_{m+2}(\|\nabla^2 d\|_{\mathcal{H}^{m}}^{\frac{1}{2}}
     +\|\nabla d\|_{\mathcal{H}^{m}}^{\frac{1}{2}})
     \|\nabla d\|_{\mathcal{H}^{m}}^{\frac{1}{2}}.
\end{equation}
Substituting  \eqref{3265} and \eqref{3266}  into \eqref{3264}
and applying the Young inequality, we find
\begin{equation}\label{3267}
\begin{aligned}
&\int_0^t \int_{\partial \Omega}\mathcal{Z}^\alpha \Delta d
  \cdot \mathcal{Z}^\alpha \nabla d\cdot n d\sigma d\tau\\
&\le \!\delta \!\!\int_0^t \!\! \|\nabla \Delta d\|_{\mathcal{H}^{m-1}}^2 d\tau
     +\delta_1 \!\! \int_0^t\!\!  \|\nabla^2 d\|_{\mathcal{H}^{m}}^2 d\tau
     +C_{\delta, \delta_1}C_{m+2}
     \!\!\int_0^t\!\!(\|\nabla d\|_{\mathcal{H}^{m}}^2
      +\|\Delta d\|_{\mathcal{H}^{m-1}}^2)d\tau.
\end{aligned}
\end{equation}
Substituting \eqref{3234}, \eqref{3237}, \eqref{3261}-\eqref{3263}
and \eqref{3259} into \eqref{3230}
and choosing $\delta_1$ small enough, one arrives at
\begin{equation}\label{3268}
\begin{aligned}
&\frac{1}{2}\int |\mathcal{Z}^\alpha \nabla d(t)|^2 dx
+\frac{3}{4}\int_0^t \int |\mathcal{Z}^\alpha \Delta d|^2dxd\tau\\
&\le \frac{1}{2}\int |\mathcal{Z}^\alpha \nabla d_0|^2 dx
  +\delta_2 \!\int_0^t \!\|\nabla^2 d\|_{\mathcal{H}^{m}}^2d\tau
  +\delta \int_0^t \|\nabla \Delta d\|_{\mathcal{H}^{m-1}}^2 d\tau\\
&\quad +C_{\delta, \delta_2}C_{m+2}(1+P(Q(t)))\int_0^t N_m(\tau)d\tau.
\end{aligned}
\end{equation}
In view of the standard elliptic regularity results with Neumann boundary condition,
we get that
\begin{equation}\label{3269}
\begin{aligned}
&\|\nabla^2 d\|_{\mathcal{H}^{m}}^2
=\|\nabla^2 \partial_t^{\alpha_0}d\|_{m-\alpha_0}^2\\
&\le C_{m+2}(\|\nabla \partial_t^{\alpha_0}d\|^2
+\|\Delta\partial_t^{\alpha_0}d\|_{m-\alpha_0}^2)\\
&\le C_{m+2}(\|\nabla d\|_{\mathcal{H}^{m}}^2+\|\Delta d\|_{\mathcal{H}^{m}}^2).
\end{aligned}
\end{equation}
Hence, the combination of \eqref{3227}, \eqref{3268} and \eqref{3269} yields directly
\begin{equation*}
\begin{aligned}
&\underset{0\le \tau \le t}{\sup}\|(u, \nabla d)\|_{\mathcal{H}^{m}}^2
+\varepsilon \int_0^t \|\nabla u\|_{\mathcal{H}^{m}}^2 d\tau
+\int_0^t \|\Delta d\|_{\mathcal{H}^{m}}^2 d\tau\\
&\le C_{m+2}\left\{\|(u_0, \nabla d_0)\|_{\mathcal{H}^{m}}^2
     +\delta \int_0^t \|\nabla \Delta d\|_{\mathcal{H}^{m-1}}^2d\tau
     +\delta \varepsilon^2 \int\|\nabla^2 u\|_{\mathcal{H}^{m-1}}^2 d\tau\right.\\
&\quad \quad    \left. +\int_0^t(\|\nabla^2 p_1\|_{\mathcal{H}^{m-1}}
\|u\|_{\mathcal{H}^{m}}
+\varepsilon^{-1}\|\nabla p_2\|_{\mathcal{H}^{m-1}}^2)d\tau
+C_{\delta}(1+P(Q(t)))\int_0^t N_m(\tau)d\tau\right\}.
\end{aligned}
\end{equation*}
Therefore, we complete the proof of lemma \ref{lemma3.2}.
\end{proof}

\subsection{Conormal Estimates for $\Delta d$}
\quad
In this subsection, we shall get some uniform estimates for $\Delta d$
in conormal Sobolev space.

\begin{lemm}\label{lemma3.3-1}
For a smooth solution to \eqref{eq1} and \eqref{bc2}, it holds that
for $\varepsilon \in (0, 1]$
\begin{equation}\label{331-1}
\underset{0\le \tau \le t}{\sup}\|\Delta d(\tau)\|^2
+\int_0^t\|\nabla \Delta d\|^2 d\tau
\le \|\Delta d_0\|^2
+C(1+Q(t)^2)\int_0^t N_1(\tau)d\tau.
\end{equation}
\end{lemm}
\begin{proof}
Taking $\nabla$ operator to the equation \eqref{eq1}$_2$, one arrives at
\begin{equation}\label{338}
\nabla d_t-\nabla \Delta d=-\nabla(u\cdot \nabla d)+\nabla(|\nabla d|^2 d).
\end{equation}
Multiplying \eqref{338} by $-\nabla \Delta d$ and integrating over $\Omega$, we find
\begin{equation}\label{339}
\begin{aligned}
&-\int \nabla d_t\cdot \nabla \Delta d ~dx +\int |\nabla \Delta d|^2 dx\\
&=\int\nabla(u\cdot \nabla d)\cdot \nabla \Delta d~dx
-\int \nabla(|\nabla d|^2 d)\cdot \nabla \Delta d~dx.
\end{aligned}
\end{equation}
By integrating by parts and applying the Neumann boundary condition \eqref{bc2}, we get
\begin{equation}\label{3310}
\begin{aligned}
-\int \nabla d_t \cdot \nabla \Delta d ~dx
=-\int_{\partial \Omega}n\cdot \nabla d_t \cdot \Delta d~d\sigma
 +\int \Delta d_t \cdot \Delta d ~dx
=\frac{1}{2}\frac{d}{dt}\int |\Delta d|^2 dx.
\end{aligned}
\end{equation}
In view of the Cauchy inequality, we obtain
\begin{equation}\label{3311}
\begin{aligned}
&\int\nabla(u\cdot \nabla d)\cdot \nabla \Delta d~dx
\le \delta \|\nabla \Delta d\|^2
+ C_\delta \|u\|_{W^{1,\infty}}^2(\|\nabla d\|^2+\|\nabla^2 d\|^2),\\
&-\int \nabla(|\nabla d|^2 d)\cdot \nabla \Delta d~dx
\le \delta \|\nabla \Delta d\|^2
+\|\nabla d\|_{L^\infty}^4 \|\nabla d\|^2
+\|\nabla d\|^2 \|\nabla^2 d\|^2.
\end{aligned}
\end{equation}
Substituting \eqref{3311} into \eqref{3310}, choosing $\delta$ small enough
and integrating over $[0, t]$, one attains
\begin{equation*}
\begin{aligned}
&\frac{1}{2}\int |\Delta d|^2(t) dx
+\frac{3}{4}\int |\nabla \Delta d|^2 dx\\
&\le \int |\Delta d_0|^2 dx
+C(\|u\|_{W^{1,\infty}}^2+\|\nabla d\|_{L^\infty}^4)
\int_0^t(\|\nabla d\|^2+\|\nabla^2 d\|^2)d\tau.
\end{aligned}
\end{equation*}
Therefore, we complete the proof of lemma \ref{lemma3.3-1}.
\end{proof}

Next, we can establish the following conormal estimates for the quantity $\Delta d$.

\begin{lemm}\label{lemma3.4-1}
For $m \ge  1$ and a smooth solution to \eqref{eq1} and \eqref{bc2}, it holds that
for $\varepsilon \in (0, 1]$
\begin{equation}\label{341-1}
\begin{aligned}
&\underset{0\le \tau \le t}{\sup}
\|\Delta d(\tau)\|_{\mathcal{H}^{m-1}}^2
+\int_0^t \|\nabla \Delta d\|_{\mathcal{H}^{m-1}}^2 d\tau\\
&\le\!\! C_{m+2}\left\{
      \|\Delta d_0\|_{\mathcal{H}^{m-1}}^2
      +\delta \!\!\int_0^t \!\!\| \Delta d\|_{\mathcal{H}^{m}}^2 d\tau
      +C_\delta\left(1+P(Q(t))\right)\int_0^t N_m(t)d\tau\right\}.
\end{aligned}
\end{equation}
\end{lemm}
\begin{proof}
The case for $m=1$ is already proved in Lemma \ref{lemma3.3-1}. Assume that
\eqref{341-1} is proved for $k=m-2$. We shall prove that is holds for $k=m-1 \ge 1$.
For $|\alpha|=m-1$,
multiplying \eqref{338} by $-\nabla \mathcal{Z}^\alpha \Delta d$, we find
\begin{equation}\label{3412}
\begin{aligned}
&-\int \mathcal{Z}^\alpha \nabla d_t \cdot \nabla \mathcal{Z}^\alpha \Delta d ~dx
+\int \mathcal{Z}^\alpha \nabla \Delta d\cdot \nabla \mathcal{Z}^\alpha \Delta d ~dx\\
&=\int \mathcal{Z}^\alpha \nabla(u\cdot \nabla d)\cdot
  \nabla \mathcal{Z}^\alpha \Delta d ~dx
  +\int \mathcal{Z}^\alpha \nabla(|\nabla d|^2 d)\cdot
  \nabla \mathcal{Z}^\alpha \Delta d ~dx.
\end{aligned}
\end{equation}
Integrating by part, it is easy to deduce that
\begin{equation}\label{3413}
\begin{aligned}
&-\int \mathcal{Z}^\alpha \nabla d_t \cdot \nabla \mathcal{Z}^\alpha \Delta d ~dx\\
&=-\int_{\partial \Omega}n\cdot\mathcal{Z}^\alpha \nabla d_t\cdot \mathcal{Z}^\alpha \Delta d~d\sigma
+\int {\nabla \cdot}(\mathcal{Z}^\alpha \nabla d_t)\cdot \mathcal{Z}^\alpha \Delta d~dx\\
&=-\int_{\partial \Omega}n\cdot\mathcal{Z}^\alpha \nabla d_t\cdot \mathcal{Z}^\alpha \Delta d~d\sigma
+\frac{1}{2}\frac{d}{dt}\int |\mathcal{Z}^\alpha \Delta d|^2dx
-\int [\mathcal{Z}^\alpha, {\nabla \cdot}] \nabla d_t\cdot \mathcal{Z}^\alpha \Delta d~dx.
\end{aligned}
\end{equation}
It is easy to deduce that
\begin{equation}\label{3414}
\int \mathcal{Z}^\alpha \nabla \Delta d\cdot \nabla \mathcal{Z}^\alpha \Delta d ~dx
=\int |\nabla \mathcal{Z}^\alpha \Delta d|^2dx
+\int [\mathcal{Z}^\alpha, \nabla] \Delta d\cdot \nabla \mathcal{Z}^\alpha \Delta d ~dx.
\end{equation}
Substituting \eqref{3413} and \eqref{3414} into \eqref{3412} and
integrating over $[0, t]$, one attains that
\begin{equation}\label{3415}
\begin{aligned}
&\frac{1}{2}\int |\mathcal{Z}^\alpha \Delta d(t)|^2dx
+\int_0^t \int |\nabla \mathcal{Z}^\alpha \Delta d|^2dx d\tau\\
&=
\frac{1}{2}\int |\mathcal{Z}^\alpha \Delta d_0|^2dx
+\int_0^t \int_{\partial \Omega}n\cdot\mathcal{Z}^\alpha \nabla d_t\cdot \mathcal{Z}^\alpha \Delta d~d\sigma d\tau
+\int_0^t \int [\mathcal{Z}^\alpha, {\nabla \cdot}] \nabla d_t\cdot \mathcal{Z}^\alpha \Delta d~dxd\tau\\
&\quad -\int_0^t \int [\mathcal{Z}^\alpha, \nabla] \Delta d\cdot \nabla \mathcal{Z}^\alpha \Delta d ~dxd\tau
+\int_0^t\int \mathcal{Z}^\alpha \nabla(u\cdot \nabla d)\cdot
  \nabla \mathcal{Z}^\alpha \Delta d ~dxd\tau\\
&\quad +\int_0^t \int \mathcal{Z}^\alpha \nabla(|\nabla d|^2 d)\cdot
  \nabla \mathcal{Z}^\alpha \Delta d ~dxd\tau\\
&:=II_1+II_2+II_3+II_4+II_5+II_6.
\end{aligned}
\end{equation}
To deal with the boundary term on the right hand side of \eqref{3415}.
If $|\alpha_0|=m-1$, then we have
\begin{equation}\label{3416-1}
\int_0^t \int_{\partial \Omega}n\cdot\mathcal{Z}^\alpha \nabla d_t\cdot \mathcal{Z}^\alpha \Delta d~d\sigma d\tau=0.
\end{equation}
On the other hand, it is easy to deduce that for $|\alpha_0|\le m-2$
\begin{equation}\label{3416}
\int_0^t \int_{\partial \Omega}n\cdot\mathcal{Z}^\alpha \nabla d_t\cdot \mathcal{Z}^\alpha \Delta d~d\sigma d\tau
\le \int_0^t|n\cdot\mathcal{Z}^\alpha \nabla d_t|_{L^2(\partial \Omega)}
    |\mathcal{Z}^{{\alpha}}\Delta d|_{L^2(\partial \Omega)}d\tau,
\end{equation}
The application of trace inequality in Proposition \ref{prop2.3}
and the boundary condition \eqref{bc2} implies
\begin{equation}\label{3417}
\begin{aligned}
&|\mathcal{Z}^{{{\alpha}}}\Delta d|_{L^2(\partial \Omega)}
=|\partial_t^{\alpha_0} \Delta d|_{H^{m-1-|\alpha_0|}(\partial \Omega)}\\
&\le C\|\nabla \partial_t^{\alpha_0} \Delta d\|_{m-1-|\alpha_0|}
     +C\|\partial_t^{\alpha_0} \Delta d\|_{m-|\alpha_0|}\\
&\le C\|\nabla \Delta d\|_{\mathcal{H}^{m-1}}
     + C\|\Delta d\|_{\mathcal{H}^{m}},
\end{aligned}
\end{equation}
and
\begin{equation}\label{3418}
\begin{aligned}
&|n\cdot\mathcal{Z}^\alpha \nabla d_t|_{L^2(\partial \Omega)}
\le C_{m}|\partial_t^{\alpha_0}\nabla d_t|_{H^{m-2-|\alpha_0|}(\partial \Omega)}\\
&\le C_{m}\|\partial_t^{\alpha_0}\nabla^2 d_t\|_{m-2-|\alpha_0|}
     +C_{m}\|\partial_t^{\alpha_0}\nabla d_t\|_{m-1-|\alpha_0|}\\
&\le C_{m}\|\nabla^2 d\|_{\mathcal{H}^{m-1}}
     +C_{m}\|\nabla d\|_{\mathcal{H}^{m}}.
\end{aligned}
\end{equation}
Substituting \eqref{3417} and \eqref{3418} into \eqref{3416} and applying
the Cauchy inequality, one attains
\begin{equation}\label{3419}
\begin{aligned}
II_2
&\le \delta_1 \int_0^t \|\nabla \Delta d\|_{\mathcal{H}^{m-1}}^2 d\tau
     +\delta \int_0^t \| \Delta d\|_{\mathcal{H}^{m}}^2 d\tau\\
&\quad +C_m\left\{C_{\delta,\delta_1} \int_0^t \|\nabla^2 d\|_{\mathcal{H}^{m-1}}^2d\tau
     +C_{\delta,\delta_1} \int_0^t\|\nabla d\|_{\mathcal{H}^{m}}^2d\tau\right\}.
\end{aligned}
\end{equation}
By virtue of the Cauchy inequality, one arrives at
\begin{equation}\label{3420}
\begin{aligned}
&II_3
\le C\int_0^t \|\Delta d_t\|_{\mathcal{H}^{m-2}}\|\Delta d\|_{\mathcal{H}^{m-1}}d\tau
\le C\int_0^t \|\Delta d\|_{\mathcal{H}^{m-1}}^2d\tau,\\
&II_4
\le \delta_1 \int_0^t\|\nabla \mathcal{Z}^\alpha \Delta d\|^2 d\tau
+C_{\delta_1} \int_0^t \|\nabla \Delta d\|_{\mathcal{H}^{m-2}}^2 d\tau.
\end{aligned}
\end{equation}
The application of Proposition \ref{prop2.2} yields directly
\begin{equation}\label{3421}
\begin{aligned}
II_5
&=\int_0^t\int \mathcal{Z}^\alpha (\nabla u\cdot \nabla d)\cdot
  \nabla \mathcal{Z}^\alpha \Delta d ~dxd\tau
  +\int_0^t\int \mathcal{Z}^\alpha (u\cdot \nabla^2 d)\cdot
  \nabla \mathcal{Z}^\alpha \Delta d ~dxd\tau\\
&\le \delta_1 \int_0^t\|\nabla \mathcal{Z}^\alpha \Delta d\|^2 d\tau
     +C_\delta\|\nabla u\|_{L^\infty_{x,t}}^2
     \int_0^t\|\nabla d\|_{\mathcal{H}^{m-1}}^2 d\tau
     +C_\delta\|\nabla d\|_{L^\infty_{x,t}}^2
     \int_0^t\|\nabla u\|_{\mathcal{H}^{m-1}}^2 d\tau\\
&\quad  +C_\delta\|u\|_{L^\infty_{x,t}}^2
     \int_0^t\|\nabla^2 d\|_{\mathcal{H}^{m-1}}^2 d\tau
     +C_\delta\|\nabla^2 d\|_{L^\infty_{x,t}}^2
     \int_0^t\|u\|_{\mathcal{H}^{m-1}}^2 d\tau\\
&\le \delta_1 \int_0^t\|\nabla \mathcal{Z}^\alpha \Delta d\|^2 d\tau
     +C_{\delta_1}(1+P(Q(t)))\int_0^t N_m(\tau)d\tau.
\end{aligned}
\end{equation}
It is easy to deduce that
\begin{equation}\label{3422}
\begin{aligned}
II_6
&=\sum_{|\beta|\ge 1}\int_0^t \int
   \mathcal{Z}^\gamma (\nabla d\cdot \nabla^2 d)\cdot\mathcal{Z}^\beta d
    \cdot \nabla \mathcal{Z}^\alpha \Delta d ~dxd\tau\\
&\quad  +\int_0^t \int\mathcal{Z}^\alpha (\nabla d\cdot \nabla^2 d)\cdot d \cdot
     \nabla \mathcal{Z}^\alpha \Delta d ~dxd\tau\\
&\quad +\sum_{|\beta|+|\gamma|=m-1}
     \int_0^t\int \mathcal{Z}^\gamma (|\nabla d|^2)\mathcal{Z}^\beta \nabla d
         \cdot \nabla \mathcal{Z}^\alpha \Delta d~dxd\tau\\
&=II_{61}+II_{62}+II_{63}.
\end{aligned}
\end{equation}
By virtue of the Proposition \ref{prop2.2} and Cauchy inequality, one arrives at
\begin{equation}\label{3423}
\begin{aligned}
II_{61}
&\le \delta \int_0^t \|\nabla \mathcal{Z}^\alpha \Delta d\|^2 d\tau
     +C_\delta\|\mathcal{Z}d\|_{L^{\infty}_{x,t}}^2
     \int_0^t \|\nabla d\cdot \nabla^2 d\|_{\mathcal{H}^{m-2}}^2 d\tau\\
&\quad +C_\delta\|\nabla d\cdot \nabla^2 d \|_{L^{\infty}_{x,t}}^2
     \int_0^t \|\mathcal{Z}d\|_{\mathcal{H}^{m-2}}^2 d\tau\\
&\le \delta_1 \int_0^t \|\nabla \mathcal{Z}^\alpha \Delta d\|^2 d\tau
     +C_1C_{\delta_1}(1+P(Q(t)))\int_0^t N_m(\tau)d\tau.
\end{aligned}
\end{equation}
Similarly, it is easy to deduce that
\begin{equation}\label{3424}
\begin{aligned}
&II_{62}\le \delta \int_0^t \|\nabla \mathcal{Z}^\alpha \Delta d\|^2 d\tau
     +C_\delta \|\nabla d\|_{W^{1,\infty}_{x,t}}^2
     \int_0^t (\|\nabla^2 d\|_{\mathcal{H}^{m-1}}^2+\|\nabla d\|_{\mathcal{H}^{m-1}}^2) d\tau,\\
&II_{63}\le \delta \int_0^t \|\nabla \mathcal{Z}^\alpha \Delta d\|^2 d\tau
     +C_\delta \|\nabla d\|_{L^{\infty}_{x,t}}^4
     \int_0^t (\|\nabla^2 d\|_{\mathcal{H}^{m-1}}^2+\|\nabla d\|_{\mathcal{H}^{m-1}}^2) d\tau.
\end{aligned}
\end{equation}
Substituting \eqref{3423} and \eqref{3424} into \eqref{3422}, we obtain
\begin{equation}\label{3425}
II_6
\le \delta \int_0^t \|\nabla \mathcal{Z}^\alpha \Delta d\|^2 d\tau
  +C_1C_\delta(1+P(Q(t)))\int_0^t N_m(\tau)d\tau.
\end{equation}
Substituting \eqref{3419}-\eqref{3421} and \eqref{3425} into \eqref{3415}
and choosing $\delta_1$ small enough, we find
\begin{equation*}
\begin{aligned}
&\frac{1}{2}\int |\mathcal{Z}^\alpha \Delta d(t)|^2dx
+\int_0^t \int |\nabla \mathcal{Z}^\alpha \Delta d|^2dx d\tau\\
&\le C_m\left\{\frac{1}{2}\int |\mathcal{Z}^\alpha \Delta d_0|^2dx
     +\delta \int_0^t \| \Delta d\|_{\mathcal{H}^{m}}^2 d\tau
     +C\int_0^t \|\nabla \Delta d\|_{\mathcal{H}^{m-2}}^2 d\tau\right\}\\
&\quad   +C_m C_\delta [1+Q(t)^2]\int_0^t N_m(\tau)d\tau.
\end{aligned}
\end{equation*}
By the induction assumption, one can eliminate the term
$\int_0^t \|\nabla \Delta d\|_{\mathcal{H}^{m-2}}^2d\tau$.
Therefore, we complete the proof of the Lemma \ref{lemma3.4-1}.
\end{proof}

\subsection{Normal Derivatives Estimates}

\quad We shall now provide an estimate for $\|\nabla u\|_{\mathcal{H}^{m-1}}$.
Of course, the only difficulty is to estimate $\|\chi \partial_n u\|_{\mathcal{H}^{m-1}}$,
where $\chi$ is compactly supported in one of the $\Omega_i$ and with value
one in a vicinity of the boundary.
Indeed, we have by definition of the norm that
$\|\chi \partial_{y_i} u\|_{\mathcal{H}^{m-1}}\le C \|u\|_{\mathcal{H}^{m}}, ~i=1,2$.
We shall thus use the local coordinates.

At first, thanks to the divergence free condition, which reads
\begin{equation}\label{301}
{\rm div}u=\partial_n u \cdot n+(\Pi \partial_{y_1}u)^1
+(\Pi \partial_{y_2}u)^2=0,
\end{equation}
to get that
\begin{equation}\label{302}
\|\chi \partial_n u \cdot n\|_{\mathcal{H}^{m-1}}\le C_m \|u\|_{\mathcal{H}^{m}}.
\end{equation}
Then, it remains to estimate $\|\chi \Pi(\partial_n u)\|_{\mathcal{H}^{m-1}}$.
We extend the smooth symmetric matrix $A$ in \eqref{bc1} to be
$
A(y, z)=A(y).
$
Let us set
\begin{equation}\label{303}
\eta \triangleq \chi\{w\times n-\Pi(Bu)\}.
\end{equation}
In view of the Navier-slip type boundary condition \eqref{bc2}, we obviously have that
$\eta$ satisfies a homogeneous Dirichlet boundary condition
on the boundary
\begin{equation}\label{304}
\left.\eta\right|_{\partial \Omega}=0.
\end{equation}
By virtue of $w\times n=(\nabla u-(\nabla u)^t)\cdot n$, then
$\eta$ can be written as
\begin{equation}\label{305}
\eta=\chi\Pi\{\partial_n u-\nabla(u\cdot n)+(\nabla n)^t u+Bu\},
\end{equation}
which yields
\begin{equation}\label{306}
\|\chi \Pi (\partial_n u)\|_{\mathcal{H}^{m-1}}
\le C_{m+1}(\|\eta\|_{\mathcal{H}^{m-1}}+\|u\|_{\mathcal{H}^{m}}).
\end{equation}
Then, we find
\begin{equation}\label{307}
\|\nabla u\|_{\mathcal{H}^{m-1}}
\le C_{m+1}(\|\eta\|_{\mathcal{H}^{m-1}}+\|u\|_{\mathcal{H}^{m}}).
\end{equation}
On the other hand, it is easy to deduce that
\begin{equation}\label{308}
\|\eta\|_{\mathcal{H}^{m-1}} \le
C_{m+1}(\|\nabla u\|_{\mathcal{H}^{m-1}}+\|u\|_{\mathcal{H}^{m}}).
\end{equation}

Now we shall establish the estimates for the equivalent quantity $\eta$.
\begin{lemm}\label{lemma3.3}
For a smooth solution to \eqref{eq1} and \eqref{bc2}, it holds that
for $\varepsilon \in (0, 1]$
\begin{equation}\label{331}
\begin{aligned}
&\underset{0\le \tau \le t}{\sup} \|\eta(\tau)\|^2
+\varepsilon \int_0^t \|\nabla \eta\|^2 d\tau\\
&\le \int |\eta_0|^2 dx
+\int_0^t\|\nabla p\|\|\eta\|d\tau
+\delta \varepsilon^2 \int_0^t\|\nabla^2 u\|^2d\tau
+\delta \int_0^t\|\nabla \Delta d\|^2d\tau\\
&\quad +C_3 C_\delta(1+Q(t))
\int_0^t N_1(\tau)d\tau.
\end{aligned}
\end{equation}
\end{lemm}
\begin{proof}
For $w\triangleq \nabla \times u$, notice that
\begin{equation*}
\nabla \times((u\cdot \nabla)u)
=(u\cdot \nabla)w-(w\cdot \nabla)u+w{\rm div}u,
\end{equation*}
which, together with \eqref{eq1-1}, reads
\begin{equation}\label{332}
w_t+(u\cdot \nabla)w-\varepsilon \Delta w=
w\cdot \nabla u-\nabla \times(\nabla d\cdot \Delta d)
\end{equation}
Consequently, the system for $\eta$ is
\begin{equation}\label{333}
\eta_t+u\cdot \nabla \eta-\varepsilon \Delta \eta
=\chi[F_1 \times n+\Pi(BF_2)]+\chi F_3+F_4+\varepsilon \Delta(\Pi B)\cdot u,
\end{equation}
where
\begin{equation*}
\begin{aligned}
F_1&=w\cdot \nabla u-\nabla \times (\nabla d\cdot \Delta d),\\
F_2&=-\nabla p-\nabla d\cdot \Delta d,\\
F_3&=-2\varepsilon \sum_{i=1}^3\partial_i w \times \partial_i n
      -\varepsilon w\times \Delta n
     +\sum_{i=1}^3 u_i w\times \partial_i n\\
     &\quad +u\cdot \nabla(\Pi B)u
     -2\varepsilon \sum_{i=1}^3 \partial_i(\Pi B)\partial_i u,\\
F_4&=(u\cdot \nabla \chi)[w\times n+\Pi(Bu)]
      -\varepsilon \Delta \chi[w\times n+\Pi(Bu)]\\
      &\quad -2 \varepsilon \sum_{i=1}^3\partial_i \chi \partial_i[w\times n+\Pi(Bu)].
\end{aligned}
\end{equation*}
Multiplying \eqref{333} by $\eta$, it is easy to deduce that
\begin{equation}\label{334}
\frac{1}{2}\frac{d}{dt}\int |\eta|^2 dx+\varepsilon \int |\nabla \eta|^2 dx
=\int F\cdot \eta dx+\varepsilon \int \Delta(\Pi B)\cdot u \cdot \eta dx.
\end{equation}
It is easy to deduce that
\begin{equation}\label{335}
\begin{aligned}
&\|\chi[F_1 \times n+\Pi(BF_2)]\|
\le C_2(\|\nabla u\|_{L^\infty}\|\nabla u\|
    +\|\nabla d\|_{L^\infty}(\|\nabla \Delta d\|+\|\Delta d\|)+\|\nabla p\|),\\
&\|\chi F_3\|
\le \varepsilon \|\nabla^2 u\|+C_3(1+\|u\|_{L^\infty})(\|u\|+\|\nabla u\|).
\end{aligned}
\end{equation}
Notice that the term $F_4$ are supported away from the boundary, we can
control all the derivatives by the $\|\cdot\|_{\mathcal{H}^m}$. Hence, we find
\begin{equation}\label{336}
\|F_4\|\le \varepsilon \|\nabla^2 u\|+C_3(1+\|u\|_{L^\infty})\|u\|_{\mathcal{H}^{1}}.
\end{equation}
Integrating by parts, it is easy to deduce that
\begin{equation}\label{336-1}
\varepsilon \int \Delta(\Pi B)\cdot u \cdot \eta dx
\le \delta\varepsilon \int |\nabla \eta|^2  dx
+C_{\delta}C_{3}(\|\nabla u\|^2+\|u\|_{\mathcal{H}^{1}}^2).
\end{equation}
Substituting \eqref{335}-\eqref{336-1} into \eqref{334}
and integrating the resulting inequality over $[0, t]$, we have
\begin{equation}\label{337}
\begin{aligned}
&\frac{1}{2}\int |\eta|^2(t) dx+\varepsilon \int_0^t \int |\nabla \eta|^2 dxd\tau\\
&\le \frac{1}{2}\int |\eta_0|^2 dx
+ \int_0^t \|\nabla p\|\|\eta\|d\tau
+\delta \varepsilon^2 \int_0^t \|\nabla^2 u\|^2d\tau
+\delta \int_0^t\|\nabla \Delta d\|^2d\tau\\
&\quad  +C(1+Q(t))\int_0^t(\|u\|_{\mathcal{H}^{1}}^2+\|\nabla u\|^2)d\tau.
\end{aligned}
\end{equation}
Therefore, we complete the proof of lemma \ref{lemma3.3}.
\end{proof}

Next, we can establish the following conormal estimates for the equivalent quantity $\eta$.
\begin{lemm}\label{lemma3.4}
For $m \ge  1$ and a smooth solution to \eqref{eq1} and \eqref{bc2}, it holds that
for $\varepsilon \in (0, 1]$
\begin{equation}\label{341}
\begin{aligned}
&\underset{0\le \tau \le t}{\sup}
\|\eta(\tau)\|_{\mathcal{H}^{m-1}}^2
+\varepsilon \int_0^t \|\nabla \eta\|_{\mathcal{H}^{m-1}}^2 d\tau\\
&\le  C_{m+2}\left\{(\|u_0\|_{\mathcal{H}^{m}}^2
      +\|\nabla u_0\|_{\mathcal{H}^{m-1}}^2)
     +\delta \varepsilon^2 \int_0^t \|\nabla^2 u\|_{\mathcal{H}^{m-1}}^2 d\tau
     +\delta \int_0^t \|\nabla \Delta d\|_{\mathcal{H}^{m-1}}^2 d\tau\right\}\\
&\quad +C_{m+2}\left\{\int_0^t \|\nabla p\|_{\mathcal{H}^{m-1}}\|\eta\|_{\mathcal{H}^{m-1}} d\tau+C_\delta(1+P(Q(t))\int_0^t N_m(t)d\tau\right\}.
\end{aligned}
\end{equation}
\end{lemm}
\begin{proof}
The case for $m=1$ is already proved in Lemma \ref{lemma3.3}. Assume that
\eqref{341} is proved for $k=m-2$. We shall prove that is holds for $k=m-1 \ge 1$.
For $|\alpha|=m-1$, applying the operator $\mathcal{Z}^\alpha$ to the
equation \eqref{333} , we find
\begin{equation}\label{342}
\partial_t \mathcal{Z}^\alpha \eta+u\cdot \nabla \mathcal{Z}^\alpha \eta
-\varepsilon \mathcal{Z}^\alpha \Delta \eta=\mathcal{Z}^\alpha F
+\varepsilon \mathcal{Z}^\alpha (\Delta(\Pi B)\cdot u)+\mathcal{C}_2^{\alpha},
\end{equation}
where
\begin{equation*}
\begin{aligned}
&\mathcal{C}_2^{\alpha}=-[\mathcal{Z}^\alpha, u\cdot \nabla ]\eta,\\
&F=\chi[F_1 \times n+\Pi(BF_2)]+\chi F_3+F_4.
\end{aligned}
\end{equation*}
Multiplying \eqref{342} by $\mathcal{Z}^\alpha \eta$, it is easy to deduce that
\begin{equation}\label{343}
\begin{aligned}
&\frac{1}{2}\int |\mathcal{Z}^\alpha \eta(t)|^2 dx
-\frac{1}{2}\int |\mathcal{Z}^\alpha \eta_0|^2 dx\\
&=\varepsilon \int_0^t \int \mathcal{Z}^\alpha \Delta \eta \cdot \mathcal{Z}^\alpha \eta~ dxd\tau
+\int_0^t \int \mathcal{Z}^\alpha F \cdot \mathcal{Z}^\alpha \eta ~dxd\tau\\
&\quad +\varepsilon\int_0^t \int  \mathcal{Z}^\alpha (\Delta(\Pi B)\cdot u)
\cdot \mathcal{Z}^\alpha \eta ~dxd\tau
+\int_0^t \int \mathcal{C}_2^\alpha \cdot \mathcal{Z}^\alpha \eta ~dxd\tau.
\end{aligned}
\end{equation}
In the local basis, it holds that
\begin{equation*}
\partial_j=\beta_j^1 \partial_{y_1}+\beta_j^1 \partial_{y_1}+\beta_j^1 \partial_{z},
~~j=1,2,3,
\end{equation*}
for harmless functions $\beta_j^i, ~i,j=1,2,3$ depending on the boundary regularity and
weight function $\phi(z)$. Therefore, the following commutation expansion holds:
\begin{equation*}
\mathcal{Z}^\alpha \Delta \eta
=\Delta \mathcal{Z}^\alpha \eta
+\sum_{|\beta|\le m-2} C_{1\beta}\partial_{zz}\mathcal{Z}^\beta \eta
+\sum_{|\beta|\le m-1} (C_{2\beta}\partial_{z}\mathcal{Z}^\beta \eta
                        +C_{3\beta}Z_y\mathcal{Z}^\beta \eta).
\end{equation*}
Then, integrating by par and applying the Cauchy inequality, we obtain
\begin{equation}\label{344}
\begin{aligned}
&\varepsilon \int_0^t \int \mathcal{Z}^\alpha \Delta \eta \cdot \mathcal{Z}^\alpha \eta dxd\tau\\
&=\varepsilon\int_0^t \int \Delta \mathcal{Z}^\alpha \eta \cdot \mathcal{Z}^\alpha \eta dxd\tau
  +\sum_{|\beta|\le m-2}\varepsilon \int_0^t \int C_{1\beta}\partial_{zz} \mathcal{Z}^\beta \eta\cdot \mathcal{Z}^\alpha \eta dxd\tau\\
&\quad +\sum_{|\beta|\le m-1}\varepsilon \int_0^t
   \int(C_{2\beta}\partial_{z}\mathcal{Z}^\beta \eta
  +C_{3\beta}Z_y\mathcal{Z}^\beta \eta)\cdot \mathcal{Z}^\alpha \eta dxd\tau\\
&\le -\frac{3}{4}\varepsilon \int_0^t \|\nabla \mathcal{Z}^\alpha \eta\|^2 d\tau
     +C\varepsilon \int_0^t \|\nabla \eta\|_{\mathcal{H}^{m-2}}^2 d\tau
     +C_{m+2}\varepsilon\int_0^t  \|\eta\|_{\mathcal{H}^{m-1}}^2 d\tau.
\end{aligned}
\end{equation}
Note that there is no boundary term in the integrating by parts since $\mathcal{Z}^\alpha \eta$
vanishes one the boundary.
Substituting \eqref{344} into \eqref{343}, we find
\begin{equation}\label{345}
\begin{aligned}
&\frac{1}{2}\int |\mathcal{Z}^\alpha \eta(t)|^2 dx
+\frac{3}{4}\varepsilon \int_0^t \|\nabla \mathcal{Z}^\alpha \eta\|^2 d\tau\\
&\le\frac{1}{2}\int |\mathcal{Z}^\alpha \eta_0|^2 dx
+C\varepsilon \int_0^t \|\nabla \eta\|_{\mathcal{H}^{m-2}}^2 d\tau
+C_{m+2}\varepsilon\int_0^t  \|\eta\|_{\mathcal{H}^{m-1}}^2 d\tau\\
&\quad+\int_0^t \int \mathcal{Z}^\alpha F \cdot \mathcal{Z}^\alpha \eta dxd\tau
+\varepsilon\int_0^t \int  \mathcal{Z}^\alpha (\Delta(\Pi B)\cdot u)
\cdot \mathcal{Z}^\alpha \eta ~dxd\tau\\
&\quad  +\int_0^t \int \mathcal{C}_2^{\alpha}\cdot \mathcal{Z}^\alpha \eta dxd\tau.
\end{aligned}
\end{equation}
Similar to \eqref{335}-\eqref{336}, we apply the Proposition \ref{prop2.2} to deduce that
\begin{equation}\label{346}
\begin{aligned}
&\int_0^t \int \mathcal{Z}^\alpha(\chi F_1 \times n)\cdot \mathcal{Z}^\alpha \eta dxd\tau
\le  C_m\left\{ \delta \! \!\int_0^t  \! \!\|\nabla \Delta d\|_{\mathcal{H}^{m-1}}^2 d\tau
    +C_\delta(1+Q(t)) \! \!\int_0^t \! \! N_m(\tau)d\tau\right\},\\
&\int_0^t \int \mathcal{Z}^\alpha(\chi \Pi(BF_2))\cdot \mathcal{Z}^\alpha \eta dxd\tau\\
&\le C_{m+1}\left\{\int_0^t \|\nabla p\|_{\mathcal{H}^{m-1}}\|\eta\|_{\mathcal{H}^{m-1}} d\tau
+(1+Q(t))\int_0^t N_m (\tau)d\tau\right\},\\
&\int_0^t \int \mathcal{Z}^\alpha(\chi F_3)\cdot \mathcal{Z}^\alpha \eta dxd\tau
\le C_{m+2}\left\{\delta \varepsilon^2 \int \|\nabla^2 u\|_{\mathcal{H}^{m-1}}^2 d\tau
  +C_\delta(1+Q(t))\int_0^t N_m(\tau)d\tau\right\},
\end{aligned}
\end{equation}
and
\begin{equation}\label{347}
\int_0^t \int \mathcal{Z}^\alpha F_4 \cdot \mathcal{Z}^\alpha \eta dxd\tau
\le C_{m+1}\left\{\delta \varepsilon^2 \int_0^t \|\nabla^2 u\|_{\mathcal{H}^{m-1}}^2 d\tau
     +C_\delta(1+Q(t))\int_0^t N_m(\tau)d\tau\right\}.
\end{equation}
Then, the combination of \eqref{346}-\eqref{347} gives directly
\begin{equation}\label{348}
\begin{aligned}
&\int_0^t \int \mathcal{Z}^\alpha F \cdot \mathcal{Z}^\alpha \eta dxd\tau\\
&\le C_{m+2}\left\{ \delta \int_0^t \|\nabla \Delta d\|_{\mathcal{H}^{m-1}}^2 d\tau
+\delta \varepsilon^2 \int \|\nabla^2 u\|_{\mathcal{H}^{m-1}}^2 d\tau\right\}\\
&\quad +C_{m+2}\left\{
\int_0^t \|\nabla p\|_{\mathcal{H}^{m-1}}\|\eta\|_{\mathcal{H}^{m-1}} d\tau
     +C_\delta (1+Q(t))\int_0^tN_m(\tau)d\tau\right\}.
\end{aligned}
\end{equation}
Integrating by parts, one arrives at directly
\begin{equation}\label{349}
\varepsilon\int_0^t \int  \mathcal{Z}^\alpha (\Delta(\Pi B)\cdot u)
\cdot \mathcal{Z}^\alpha \eta ~dxd\tau
\le \delta \varepsilon^2 \int_0^t \|\nabla \eta\|_{\mathcal{H}^{m-1}}^2 d\tau
+C_\delta C_{m+2}\int_0^t N_m(\tau)d\tau.
\end{equation}
The remaining term are more involved, since it is desired to obtain an estimate
independent of $\partial_z \eta$.
First, it is easy to deduce that
\begin{equation}\label{3410-1}
\begin{aligned}
\mathcal{C}_2^{\alpha}
&=-\sum_{|\beta|\ge 1, \beta+\gamma=\alpha}\sum_{i=1}^2 C_{\alpha, \beta}
   \mathcal{Z}^{\beta}u_i \mathcal{Z}^{\gamma}\partial_{y_i}\eta
  -\sum_{|\beta|\ge 1, \beta+\gamma=\alpha}C_{\alpha, \beta}
   \mathcal{Z}^{\beta}(u\cdot N)\mathcal{Z}^{\gamma}\partial_{z}\eta\\
&\quad -\sum_{|\beta|\le m-2}C(\alpha, \beta, z)(u\cdot N)\partial_z \mathcal{Z}^{\beta}\eta,
\end{aligned}
\end{equation}
where $C(\alpha, \beta, z)$ are smooth functions depending on $\alpha, \beta$
and $\varphi(z)$.
By virtue of the Proposition \ref{prop2.2}, we obtain
\begin{equation}\label{3410-2}
\begin{aligned}
&\sum_{|\beta|\ge 1, \beta+\gamma=\alpha}\sum_{i=1}^2 C_{\alpha, \beta}
  \int_0^t \|\mathcal{Z}^{\beta}u_i \mathcal{Z}^{\gamma}\partial_{y_i}\eta\|^2 d\tau\\
&\le C\|\mathcal{Z}u\|_{L^\infty_{x,t}}^2\int_0^t \|Z_y \eta\|_{\mathcal{H}^{m-2}}^2d\tau
     +C\|Z_y \eta\|_{L^\infty_{x,t}}^2\int_0^t \|\mathcal{Z}u\|_{\mathcal{H}^{m-2}}^2d\tau
\end{aligned}
\end{equation}
and
\begin{equation}\label{3410-3}
\begin{aligned}
&\sum_{|\beta|\le m-2}C(\alpha, \beta, z)
\int_0^t \|(u\cdot N)\partial_z \mathcal{Z}^{\beta}\eta\|^2 d\tau\\
&\le C\sum_{|\beta|\le m-2}\left\|\frac{u\cdot N}{\varphi(z)}\right\|_{L^\infty_{x,t}}^2
     \int_0^t \|Z_3 \mathcal{Z}^{\beta}\eta\|^2 d\tau\\
&\le C\|u\|_{W^{1,\infty}_{x,t}}^2
     \int_0^t \|\eta\|_{\mathcal{H}^{m-1}}^2 d\tau,
\end{aligned}
\end{equation}
where we have used the Hardy inequality
$\|\frac{u\cdot N}{\varphi}\|_{L^\infty}\le C\|u\|_{W^{1,\infty}}$
due to the boundary condition $u\cdot n|_{\partial \Omega}=0$.
On the other hand, it is easy to deduce that
\begin{equation*}
\begin{aligned}
&-\sum_{|\beta|\ge 1, \beta+\gamma=\alpha}C_{\alpha, \beta}
   \mathcal{Z}^{\beta}(u\cdot N)\mathcal{Z}^{\gamma}\partial_{z}\eta\\
&=-\sum_{|\beta|\ge 1, \beta+\gamma=\alpha}C_{\alpha, \beta}
   \frac{1}{\varphi(z)}\mathcal{Z}^{\beta}
   (u\cdot N)\varphi(z)\mathcal{Z}^{\gamma}\partial_{z}\eta\\
&=\sum_{\widetilde{\beta}\le \beta, \widetilde{\gamma} \le \gamma}
  \mathcal{Z}^{\widetilde{\beta}} \left(\frac{u\cdot N}{\varphi(z)}\right)
   \mathcal{Z}^{\widetilde{\gamma}}(Z_3 \eta),
\end{aligned}
\end{equation*}
where $|\widetilde{\beta}|+|\widetilde{\gamma}|\le m-1, |\widetilde{\gamma}|\le m-2$,
and $C_{\alpha, \widetilde{\beta}, \widetilde{\gamma}}$ are some smooth bounded
functions depending on $\varphi(z)$.
If $\widetilde{\beta}=0$, and hence $|\widetilde{\gamma}|\le m-2$,
it holds that
\begin{equation}\label{3410-4}
\begin{aligned}
\int_0^t \|\mathcal{Z}^{\widetilde{\beta}}\left(\frac{u\cdot N}{\varphi(z)}\right)
          \mathcal{Z}^{\widetilde{\gamma}}(Z_3 \eta)\|^2 d\tau
\le C\|u\|_{W^{1,\infty}_{x,t}}^2
     \int_0^t \|\eta\|_{\mathcal{H}^{m-1}}^2 d\tau.
\end{aligned}
\end{equation}
If $\widetilde{\beta} \neq 0$, one attains that
\begin{equation}\label{3410-5}
\begin{aligned}
&\int_0^t \|\mathcal{Z}^{\widetilde{\beta}}\left(\frac{u\cdot N}{\varphi(z)}\right)
          \mathcal{Z}^{\widetilde{\gamma}}(Z_3 \eta)\|^2 d\tau\\
&\le \left\|\mathcal{Z}\left(\frac{u \cdot N}{\varphi(z)}\right)\right\|_{L^\infty_{x,t}}^2
     \int_0^t \|Z_3 \eta\|_{\mathcal{H}^{m-2}}^2 d\tau
     +\|Z_3 \eta\|_{L^\infty_{x,t}}^2\int\|\mathcal{Z}\left(\frac{u \cdot N}{\varphi(z)}\right)\|_{\mathcal{H}^{m-2}}^2 d\tau\\
&\le \|\partial_z(u\cdot N)\|_{\mathcal{H}^{1,\infty}}^2
     \int_0^t \|\eta\|_{\mathcal{H}^{m-1}}^2 d\tau
     +\|Z_3 \eta\|_{L^\infty_{x,t}}^2\int_0^t \|\partial_z(u\cdot N)\|_{\mathcal{H}^{m-1}}^2d\tau\\
&\le C(\|Z_y u\|_{\mathcal{H}^{1,\infty}}^2+\|Z_3 \eta\|_{L^\infty_{x,t}}^2)
     \int_0^t(\|\eta\|_{\mathcal{H}^{m-1}}^2+\|u\|_{\mathcal{H}^{m}}^2)d\tau,
\end{aligned}
\end{equation}
where we have used the divergence free condition for the velocity and
the Hardy inequality:
\begin{equation*}
\sum_{|\beta|\le 1}\|\mathcal{Z}^\beta\left(\frac{u\cdot N}{\varphi(z)}\right)\|_{\mathcal{H}^{m-2}}
\le C\|\partial_z(u\cdot N)\|_{\mathcal{H}^{m-1}}^2.
\end{equation*}
The combination of \eqref{3410-1}-\eqref{3410-5} yields directly
\begin{equation}\label{3411}
\int_0^t \int \mathcal{C}_2^{\alpha}\cdot \mathcal{Z}^\alpha \eta dxd\tau
\le C_m(1+P(Q(t)))\int_0^t N_m(t)d\tau.
\end{equation}
Substituting \eqref{348}, \eqref{349} and \eqref{3411} into \eqref{345}, we find
\begin{equation*}
\begin{aligned}
&\frac{1}{2}\int |\mathcal{Z}^\alpha \eta(t)|^2 dx
+\frac{3\varepsilon}{4} \int_0^t \|\nabla \mathcal{Z}^\alpha \eta\|^2 d\tau\\
&\le C_{m+2}\left\{\frac{1}{2}\int |\mathcal{Z}^\alpha \eta_0|^2 dx
+C\varepsilon \int_0^t \|\nabla \eta\|_{\mathcal{H}^{m-2}}^2 d\tau
+\delta \int_0^t \|\nabla \Delta d\|_{\mathcal{H}^{m-1}}^2d\tau\right\}\\
&\quad +C_{m+2}\left\{
     \delta \varepsilon^2 \int_0^t \|\nabla^2 u\|_{\mathcal{H}^{m-1}}^2 d\tau
     +\int_0^t \|\nabla p\|_{\mathcal{H}^{m-1}}\|\eta\|_{\mathcal{H}^{m-1}} d\tau\right\}\\
&\quad +C_{m+2}C_\delta (1+P(Q(t)))\int_0^t N_m(t)d\tau.
\end{aligned}
\end{equation*}
By the induction assumption, one can eliminate the term $\varepsilon \int_0^t \|\nabla \eta\|_{\mathcal{H}^{m-2}}^2 d\tau$.
Therefore, we complete the proof of the Lemma \ref{lemma3.4}.
\end{proof}

Similar to the analysis of \eqref{301}-\eqref{302}, it is easy to deduce that
\begin{equation*}
\begin{aligned}
\int_0^t \|\nabla^2 u\|_{\mathcal{H}^{m-1}}^2 d\tau
&\le C_{m+2}\left\{\int_0^t \|\nabla \eta\|_{\mathcal{H}^{m-1}}^2 d\tau
    +\int_0^t \|\nabla u\|_{\mathcal{H}^{m}}^2 d\tau
    +\int_0^t N_m(\tau) d\tau\right\},
\end{aligned}
\end{equation*}
which, together with \eqref{321}, \eqref{341-1} and \eqref{341}, yields directly
\begin{equation}\label{3201}
\begin{aligned}
&\underset{0\le \tau \le t}{\sup}
(\|(u, \nabla d)\|_{\mathcal{H}^{m}}^2+\|(\nabla u, \Delta d)\|_{\mathcal{H}^{m-1}}^2)
+\varepsilon \int_0^t \|\nabla u\|_{\mathcal{H}^{m}}^2 d\tau\\
&\quad  +\varepsilon \int_0^t \|\nabla^2 u\|_{\mathcal{H}^{m-1}}^2 d\tau
+\int_0^t \|\Delta d\|_{\mathcal{H}^{m}}^2 d\tau
+\int_0^t \|\nabla \Delta d\|_{\mathcal{H}^{m-1}}^2 d\tau\\
&\le C_{m+2}\left\{\|(u_0, \nabla d_0)\|_{\mathcal{H}^{m}}^2
    +\|(\nabla u_0, \Delta d_0)\|_{\mathcal{H}^{m-1}}^2
    +\int_0^t \|\nabla p\|_{\mathcal{H}^{m-1}}\|\eta\|_{\mathcal{H}^{m-1}} d\tau\right.\\
&\quad \quad
+\int_0^t\|\nabla^2 p_1\|_{\mathcal{H}^{m-1}}\|u\|_{\mathcal{H}^m}d\tau
+\varepsilon^{-1}\int_0^t\|\nabla p_2\|_{\mathcal{H}^{m-1}}^2d\tau
+\left.\left(1+P(Q(t))\right)\int_0^t N_m(t)d\tau\right\}.
\end{aligned}
\end{equation}

\subsection{Estimates for Pressure}
\quad
It remains to estimate the pressure and the $L^\infty-$norm on the right-hand side
of the estimates of Lemmas \ref{lemma3.2}, \ref{lemma3.4-1} and \ref{lemma3.4}.
The aim of this section is to give the estimate for the  pressure.

\begin{lemm}\label{lemma3.5}
For $m \ge 2$, we have the following estimates for the pressure:
\begin{equation}\label{351}
\begin{aligned}
&\int_0^t \|\nabla p_1\|_{\mathcal{H}^{m-1}}^2 d\tau
 +\int_0^t \|\nabla^2 p_1\|_{\mathcal{H}^{m-1}}^2 d\tau\\
&\le C_{m+2}Q(t)\int_0^t N_m(\tau) d\tau
+C_{m+2}Q(t)\int_0^t \|\nabla \Delta d\|_{\mathcal{H}^{m-1}}^2 d\tau,
\end{aligned}
\end{equation}
and
\begin{equation}\label{351-1}
\int_0^t \|\nabla p_2\|_{\mathcal{H}^{m-1}}^2 d\tau
\le C_{m+2} \varepsilon  \int_0^t N_m(\tau)d\tau.
\end{equation}
\end{lemm}
\begin{proof}
We recall that we have $p=p_1+p_2$, where
\begin{equation}\label{352}
\Delta p_1=-{\rm div}(u\cdot \nabla u+\nabla d \cdot \Delta d),
\quad x\in \Omega,
\quad \partial_n p_1=-(u\cdot \nabla u)\cdot n,
\quad x\in \partial \Omega,
\end{equation}
and
\begin{equation}\label{353}
\Delta p_2=0,
\quad x\in \Omega,
\quad \partial_n p_2=\varepsilon \Delta u \cdot n,
\quad x\in \partial \Omega.
\end{equation}
For $|\alpha_0|+|\alpha_1|=m-1$, by virtue of the standard elliptic regularity results with
Neumann boundary conditions, we get that
\begin{equation}\label{354}
\begin{aligned}
&\|\nabla \partial_t^{\alpha_0} p_1\|_{|\alpha_1|}
+\|\nabla^2 \partial_t^{\alpha_0} p_1\|_{|\alpha_1|}\\
&\le C(\|\partial_t^{\alpha_0}{\rm div}(u\cdot \nabla u
       +\nabla d\cdot \Delta d )\|_{|\alpha_1|}
+\|\partial_t^{\alpha_0}(u\cdot \nabla u+\nabla d\cdot \Delta d) \|)\\
&\quad +C|\partial_t^{\alpha_0}(u\cdot \nabla u)\cdot n|_{{H^{|\alpha_1|+\frac{1}{2}}}(\partial \Omega)}.
\end{aligned}
\end{equation}
Since $u\cdot n=0$ on the boundary, we note that
\begin{equation*}
(u\cdot \nabla u)\cdot n
=-(u\cdot \nabla n)\cdot u, \quad  x\in \partial \Omega,
\end{equation*}
and hence that
\begin{equation}\label{355}
|\partial_t^{\alpha_0}(u\cdot \nabla u)\cdot n|_{{H^{|\alpha_1|+\frac{1}{2}}}(\partial \Omega)}
\le C_{m+2}|\partial_t^{\alpha_0}(u\otimes u)|_{{H^{|\alpha_1|+\frac{1}{2}}}(\partial \Omega)}.
\end{equation}
In view of the trace theorem in Proposition \ref{prop2.3}, one arrives at
\begin{equation*}
|\partial_t^{\alpha_0}(u\otimes u)|_{{H^{|\alpha_1|+\frac{1}{2}}}(\partial \Omega)}
\lesssim \|\nabla \partial_t^{\alpha_0}(u\otimes u)\|_{|\alpha_1|}
    +\|\partial_t^{\alpha_0}(u\otimes u)\|_{|\alpha_1|+1},
\end{equation*}
which, together with \eqref{355}, reads
\begin{equation}\label{356}
|\partial_t^{\alpha_0}(u\cdot \nabla u)\cdot n|_{{H^{|\alpha_1|+\frac{1}{2}}}(\partial \Omega)}
\le C_{m+2}\left(\|u\cdot \nabla u\|_{\mathcal{H}^{m-1}}
+\|u\otimes u\|_{\mathcal{H}^{m}}\right).
\end{equation}
Substituting \eqref{356} into \eqref{354} and integrating the resulting inequality
over $[0, t]$, we find
\begin{equation}\label{357}
\begin{aligned}
&\int_0^t \|\nabla p_1\|_{\mathcal{H}^{m-1}}^2 d\tau
 +\int_0^t \|\nabla^2 p_1\|_{\mathcal{H}^{m-1}}^2 d\tau\\
&=\int_0^t \|\nabla \partial_t^{\alpha_0} p_1\|_{|\alpha_1|}^2 d\tau
+\int_0^t \|\nabla^2 \partial_t^{\alpha_0} p_1\|_{|\alpha_1|}^2 d\tau\\
&\le C_{m+2}\left\{\int_0^t \|\nabla u\cdot \nabla u\|_{\mathcal{H}^{m-1}}^2 d\tau
+\int_0^t \|u\cdot \nabla u\|_{\mathcal{H}^{m-1}}^2 d\tau
+\int_0^t \|u\otimes u\|_{\mathcal{H}^{m}}^2 d\tau\right\}\\
&\quad +C_{m+2}\left\{\int_0^t \|\Delta d\cdot \Delta d\|_{\mathcal{H}^{m-1}}^2 d\tau
+\int_0^t \|\nabla d\cdot \nabla \Delta d\|_{\mathcal{H}^{m-1}}^2 d\tau\right\}\\
&\le C_{m+2}(\|u\|_{W^{1,\infty}_{x,t}}^2+\|\Delta d\|_{L^\infty_{x,t}}^2)
   \int_0^t (\|u\|_{\mathcal{H}^{m}}^2
       +\|\nabla u\|_{\mathcal{H}^{m-1}}^2+\|\Delta d\|_{\mathcal{H}^{m-1}}^2) d\tau\\
&\quad  +C_{m+2}\int_0^t \|\nabla d\cdot \nabla \Delta d\|_{\mathcal{H}^{m-1}}^2 d\tau,
\end{aligned}
\end{equation}
where we have used the Proposition \ref{prop2.2} in the last inequality.
On the other hand, the application of Proposition \ref{prop2.2} yields
\begin{equation*}
\begin{aligned}
&\int_0^t \|\nabla d\cdot \nabla \Delta d\|_{\mathcal{H}^{m-1}}^2 d\tau\\
&=\sum_{|\beta|\ge 1, |\beta|+|\gamma|=m-1} \int_0^t
  \|\mathcal{Z}^{\beta-1}\mathcal{Z} \nabla d\cdot
     \mathcal{Z}^\gamma \nabla \Delta d\|^2d\tau
     +\int_0^t \|\nabla d \cdot \mathcal{Z}^\alpha \nabla \Delta d\|^2 d\tau\\
&\le \|\mathcal{Z}\nabla d\|_{L^\infty_{x,t}}^2
      \int_0^t \|\nabla \Delta d\|_{\mathcal{H}^{m-2}}^2 d\tau
    +\|\nabla \Delta d\|_{L^\infty_{x,t}}^2
      \int_0^t \|\mathcal{Z}\nabla d\|_{\mathcal{H}^{m-2}}^2 d\tau\\
&\quad
+\|\nabla d\|_{L^\infty_{x,t}}^2 \int_0^t \|\nabla \Delta d\|_{\mathcal{H}^{m-1}}^2 d\tau\\
&\le C_1 Q(t)\int_0^t \|\nabla \Delta d\|_{\mathcal{H}^{m-1}}^2 d\tau
     +Q(t)\int_0^t \|\nabla d\|_{\mathcal{H}^{m-1}}^2 d\tau,
\end{aligned}
\end{equation*}
which, together with \eqref{357}, reads
\begin{equation*}
\int_0^t \|\nabla p_1\|_{\mathcal{H}^{m-1}}^2 d\tau
 +\int_0^t \|\nabla^2 p_1\|_{\mathcal{H}^{m-1}}^2 d\tau
\le C_{m+2}Q(t)\int_0^t N_m(t) d\tau
+C_{m+2}Q(t)\int_0^t \|\nabla \Delta d\|_{\mathcal{H}^{m-1}}^2 d\tau.
\end{equation*}
Hence, we complete the proof of estimate \eqref{351}.
It remains to estimate $p_2$. By using the elliptic regularity for the Neumann
problem again, we get that for $|\alpha_0|+|\alpha_1|=m-1$,
\begin{equation}\label{358}
\|\nabla \partial_t^{\alpha_0}p_2\|_{|\alpha_1|}
\le C\varepsilon |\partial_t^{\alpha_0} \Delta u\cdot n|_{H^{|\alpha_1|-\frac{1}{2}}(\partial \Omega)}.
\end{equation}
From the results by Masmoudi and Rousset \cite{Masmoudi-Rousset},
we have the following fact that
\begin{equation}\label{359}
|\partial_t^{\alpha_0} \Delta u\cdot n |_{H^{|\alpha_1|-\frac{1}{2}}(\partial \Omega)}
\le C_{m+2}|\partial_t^{\alpha_0} u|_{H^{|\alpha_1|+\frac{1}{2}}(\partial \Omega)}.
\end{equation}
Applying the trace inequality in Proposition \ref{prop2.3}, we get
\begin{equation}\label{3510}
|\partial_t^{\alpha_0} u|_{H^{|\alpha_1|+\frac{1}{2}}(\partial \Omega)}
\le C(\|\nabla \partial_t^{\alpha_0} u\|_{|\alpha_1|}
      +
      \|\partial_t^{\alpha_0} u\|_{|\alpha_1|+1}).
\end{equation}
Substituting \eqref{359}- \eqref{3510} into \eqref{358}
and integrating the resulting inequality over $[0, t]$, one arrives at
\begin{equation*}
\int_0^t \|\nabla p_2\|_{\mathcal{H}^{m-1}}^2 d\tau
=\int_0^t \|\nabla \partial_t^{\alpha_0}p_2\|_{|\alpha_1|}^2 d\tau
\le C_{m+2} \varepsilon  \int_0^t (\|\nabla u\|_{\mathcal{H}^{m-1}}^2
+\|u\|_{\mathcal{H}^{m}}^2)d\tau.
\end{equation*}
Hence, we complete the proof of the estimate \eqref{351-1}.
\end{proof}

By virtue of \eqref{351}, the H\"{o}lder and Cauchy inequalities, it is easy to deduce that
\begin{equation}\label{3202}
\begin{aligned}
&\int_0^t\|\nabla^2 p_1\|_{\mathcal{H}^{m-1}}\|u\|_{\mathcal{H}^{m}}d\tau\\
&\le \left(\int_0^t \|\nabla^2 p_1\|_{\mathcal{H}^{m-1}}^2d\tau\right)^{\frac{1}{2}}
     \left(\int_0^t \|u\|_{\mathcal{H}^{m}}^2d\tau\right)^{\frac{1}{2}}\\
&\le C_{m+2}Q(t)^{\frac{1}{2}}
     \left(\int_0^t N_m(\tau) d\tau\right)^{\frac{1}{2}}
     \left(\int_0^t \|u\|_{\mathcal{H}^{m}}^2d\tau\right)^{\frac{1}{2}}\\
&\quad+
      C_{m+2}Q(t)^{\frac{1}{2}}
     \left(\int_0^t \|\nabla \Delta d\|_{\mathcal{H}^{m-1}}^2 d\tau\right)^{\frac{1}{2}}
     \left(\int_0^t \|u\|_{\mathcal{H}^{m}}^2d\tau\right)^{\frac{1}{2}}\\
&\le \delta \int_0^t \|\nabla \Delta d\|_{\mathcal{H}^{m-1}}^2 d\tau
     +C_{m+2}C_\delta Q(t)\int_0^t N_m(\tau) d\tau.
\end{aligned}
\end{equation}
Similarly, we also deduce
\begin{equation}\label{3203}
\int_0^t \|\nabla p\|_{\mathcal{H}^{m-1}}\|\eta\|_{\mathcal{H}^{m-1}} d\tau
\le \delta \int_0^t \|\nabla \Delta d\|_{\mathcal{H}^{m-1}}^2 d\tau
     +C_{m+2}C_\delta Q(t)\int_0^t N_m(\tau) d\tau.
\end{equation}
Substituting \eqref{351}, \eqref{351-1}, \eqref{3202} and \eqref{3203}
into \eqref{3201} and choosing $\delta$ small enough, one arrives at
\begin{equation}\label{3204}
\begin{aligned}
&\underset{0\le \tau \le t}{\sup}
(\|(u, \nabla d)\|_{\mathcal{H}^{m}}^2+\|(\nabla u, \Delta d)\|_{\mathcal{H}^{m-1}}^2)
+\varepsilon \int_0^t \|\nabla u\|_{\mathcal{H}^{m}}^2 d\tau\\
&\quad  +\varepsilon \int_0^t \|\nabla^2 u\|_{\mathcal{H}^{m-1}}^2 d\tau
+\int_0^t \|\Delta d\|_{\mathcal{H}^{m}}^2 d\tau
+\int_0^t \|\nabla \Delta d\|_{\mathcal{H}^{m-1}}^2 d\tau\\
&\le C_{m+2}\left\{(\|(u_0, \nabla d_0)\|_{\mathcal{H}^{m}}^2
    +\|(\nabla u_0, \Delta d_0)\|_{\mathcal{H}^{m-1}}^2)
    +C\left(1+P(Q(t))\right)\int_0^t N_m(t)d\tau\right\}.
\end{aligned}
\end{equation}

\subsection{$L^\infty-estimates$}

\quad In this subsection, we shall provide the $L^\infty-$estimates
of $(u, d)$ which are
needed to estimate on the right-hand side of the estimate \eqref{3204}.
\begin{lemm}\label{lemma3.6}
For a smooth solution to \eqref{eq1} and \eqref{bc2}, it holds that
\begin{equation}\label{361}
\|(u, \nabla d)\|_{L^\infty}^2 \le C_3 N_m(t),\quad m\ge 2,
\end{equation}
\begin{equation}\label{362-2}
\|u\|_{\mathcal{H}^{2,\infty}}^2 \le C N_m(t),\quad m\ge 4,
\end{equation}
\begin{equation}\label{362}
\|(u_t, d_t, \nabla d_t)\|_{L^\infty}^2 \le C_3 N_m^3(t),\quad m\ge 3,
\end{equation}
\begin{equation}\label{363}
\|\nabla^2 d\|_{L^\infty}^2 \le C_4 N_m^3(t),\quad m\ge 3,
\end{equation}
\begin{equation}\label{364}
\|\nabla \Delta d\|_{L^\infty}^2 \le C_4 N_m^3(t),\quad m\ge 3.
\end{equation}
\end{lemm}
\begin{proof}
By virtue of the Sobolev inequality in Proposition \ref{prop2.3}, one arrives at
\begin{equation}\label{365}
\|u\|_{L^\infty}^2\le C(\|\nabla u\|_{1}^2+\|u\|_2^2)
\end{equation}
and
\begin{equation}\label{366}
\|\nabla d\|_{L^\infty}^2\le C(\|\nabla^2 d\|_{1}^2+\|\nabla d\|_2^2).
\end{equation}
In view of the standard elliptic regularity results with Neumann boundary condition,
we get that
\begin{equation}\label{367}
\|\nabla^2 d\|_{m}^2\le C_{m+2}(\|\Delta d\|_m^2+\|\nabla d\|^2).
\end{equation}
Then, the combination of \eqref{366} and \eqref{367} yields directly
\begin{equation*}
\|\nabla d\|_{L^\infty}^2\le C_3(\|\Delta d\|_{1}^2+\|\nabla d\|_2^2),
\end{equation*}
which, together with \eqref{365}, complete the proof of \eqref{361}.
The estimates \eqref{362-2} follows directly from the application of
Sobolev inequality in Proposition \ref{prop2.3}.
In view of Sobolev inequality in Proposition \ref{prop2.3}, one arrives at
\begin{equation}\label{367-1}
\|u_t\|_{L^\infty}^2 \le C(\|\nabla u_t\|_1^2+\|u_t\|_2^2)
\le C N_m(t), ~{\rm for}~m \ge 3.
\end{equation}
By virtue of the equation \eqref{eq1}$_2$, we find
\begin{equation}\label{367-2}
\begin{aligned}
\|d_t\|_{L^\infty}^2
&\le C(\|\nabla d_t\|_1^2+\|d_t\|_2^2)\\
&\le C(\|\nabla d_t\|_1^2+\|\Delta d\|_2^2+\|u\cdot \nabla d\|_2^2+\||\nabla d|^2 d\|_2^2).
\end{aligned}
\end{equation}
By Proposition \ref{prop2.2} and estimate \eqref{361}, one attains
\begin{equation}\label{367-3}
\|u\cdot \nabla d\|_2^2
\le C(\|u\|_{L^\infty}^2 \|\nabla d\|_2^2+\|\nabla d\|_{L^\infty}^2 \|u\|_2^2)
\le C_3 N_m(t), ~{\rm for}~m \ge 2;
\end{equation}
and
\begin{equation}\label{367-4}
\begin{aligned}
\||\nabla d|^2 d\|_2^2
&\le \sum_{|\gamma|\ge 1,|\beta|+|\gamma|\le 2}
    \int |{Z}^{\beta}(|\nabla d|^2){Z}^\gamma d|^2 dx
     +\||\nabla d|^2\|_2^2\\
&\le \|Z d\|_{L^\infty}^2 \||\nabla d|^2\|_1^2
     +\||\nabla d|^2\|_{L^\infty}^2\|Z d\|_1^2
     +\|\nabla d\|_{L^\infty}^2\|\nabla d\|_2^2\\
&\le C_3 N_2^3(t).
\end{aligned}
\end{equation}
Hence, the combination of \eqref{367-2}-\eqref{367-4} gives directly
\begin{equation}\label{367-5}
\|d_t\|_{L^\infty}^2 \le C_3 N_m^3(t),~{\rm for}~m \ge 3.
\end{equation}
The application of Proposition \ref{prop2.3} and
the standard elliptic regularity results with Neumann boundary condition,
we obtain for $m \ge 3$
\begin{equation}\label{367-6}
\|\nabla d_t\|_{L^\infty}^2
\le C(\|\nabla^2 d_t\|_1^2+\|\nabla d_t\|_2^2)
\le C(\|\Delta d_t\|_1^2+\|\nabla d_t\|_2^2)
\le C(\|\Delta d\|_{\mathcal{H}^{2}}^2+\|\nabla d\|_{\mathcal{H}^{3}}^2).
\end{equation}
Then, the combination of \eqref{367-1}, \eqref{367-5} and \eqref{367-6}
completes the proof of \eqref{362}.
It is easy to deduce that
\begin{equation*}
\begin{aligned}
&\partial_{ii}=\partial_{y_i}^2-\partial_{y_i}(\partial_i \psi \partial_z)
                -\partial_i \psi \partial_z \partial_{y_i}
                +(\partial_i \psi)^2 \partial_z^2, \quad i=1,2,\\
&\partial_1 \partial_2=\partial_{y_1}\partial_{y_2}
               -\partial_{y_2}(\partial_1 \psi \partial_z)
               -\partial_2 \psi \partial_{y_1}\partial_z
               +\partial_2 \psi \partial_1 \psi \partial_z^2,\\
&\partial_i \partial_3=\partial_{y_i}\partial_z-\partial_i \psi \partial_z^2,\quad i=1,2.
\end{aligned}
\end{equation*}
Then, we find that
\begin{equation}\label{368}
\Delta=(1+|\nabla \psi|^2)\partial_{z}^2
        +\sum_{i=1,2}(\partial_{y_i}^2
        -\partial_{y_i}(\partial_i \psi \partial_z)
        -\partial_i \psi \partial_z \partial_{y_i}).
\end{equation}
and
\begin{equation}\label{369}
\begin{aligned}
\nabla^2
&=[(1+|\nabla \psi|^2)+\partial_2 \psi \partial_1 \psi
        -\partial_1 \psi-\partial_2 \psi]\partial_{z}^2
        +\partial_{y_1}\partial_{y_2}\\
&\quad  +\sum_{i=1,2}(\partial_{y_i}^2
        -\partial_{y_i}(\partial_i \psi \partial_z)
        -\partial_i \psi \partial_z \partial_{y_i})
        -\partial_{y_2}(\partial_1 \psi \partial_z)\\
&\quad  -\partial_2 \psi\partial_{y_1}\partial_z
        +\partial_{y_1}\partial_z
        +\partial_{y_2}\partial_z.
\end{aligned}
\end{equation}
The combination of \eqref{368} and \eqref{369} and Proposition \ref{prop2.3} yield that
\begin{equation}\label{3610}
\begin{aligned}
\|\nabla^2 d\|_{L^\infty}^2
&\le C_1(\|\Delta d\|_{L^\infty}^2+\|\partial_z \partial_{y_i} d\|_{L^\infty}^2
      +\|\partial_{y_i}\partial_{y_j}d\|_{L^\infty}^2)\\
&\le C_1(\|\nabla \Delta d\|_1^2+\|\Delta d\|_2^2
      +\|\nabla \partial_z \partial_{y_i} d\|_1^2
       +\|\partial_z \partial_{y_i} d\|_2^2)\\
&\quad   +C(\|\nabla \partial_{y_i}\partial_{y_j}d\|_1^2
       +\|\partial_{y_i}\partial_{y_j}d\|_2^2)\\
&\le C_1(\|\nabla \Delta d\|_1^2+\|\Delta d\|_2^2
      +\|\nabla^2 d\|_2^2+\|\nabla d\|_3^2)\\
&\le C_4(\|\nabla \Delta d\|_1^2+\|\Delta d\|_2^2+\|\nabla d\|_3^2),
\end{aligned}
\end{equation}
where we have used the estimate \eqref{367} in the last inequality.
In order to deal with the first term on the right hand side of \eqref{3610},
we apply the equation \eqref{eq1}$_2$ to attain that
\begin{equation}\label{3611}
\begin{aligned}
\|\nabla \Delta d\|_1^2
&\le \|\nabla(d_t+u\cdot \nabla d-|\nabla d|^2 d)\|_1^2\\
&\le \|\nabla d_t\|_1^2
     +\|\nabla u\cdot \nabla d\|_1^2
     +\|u\cdot \nabla^2 d\|_1^2
     +\|\nabla(|\nabla d|^2 d)\|_1^2.
\end{aligned}
\end{equation}
It is easy to deduce that
\begin{equation}\label{3612}
\|\nabla d_t\|_1^2 \le \|\nabla d\|_{\mathcal{H}^{2}}^2\le N_m(t), ~\text{for}~m\ge2,
\end{equation}
\begin{equation}\label{3613}
\|\nabla u\cdot \nabla d\|_1^2
\le \|(\nabla u, \nabla d)\|_{L^\infty}^2 \|(\nabla u,\nabla d)\|_1^2
\le C_3 N_m^2(t), ~\text{for}~m\ge2,
\end{equation}
and
\begin{equation}\label{3614}
\|u\cdot \nabla^2 d\|_1^2
\le \|u\|_{L^\infty}^2\|\nabla^2 d\|^2
     +\|Z u\|_{L^\infty}^2\|\nabla^2 d\|^2
     +\|u\|_{L^\infty}^2\|\nabla^2 d\|_1^2
\le CN_m^2(t), ~\text{for}~m\ge3.
\end{equation}
In view of the basic fact $|d|=1$(see \eqref{unit}), one arrives at
\begin{equation}\label{3615}
\|\nabla(|\nabla d|^2 d)\|_1^2
\le C N_m^3(t), ~\text{for}~m\ge3.
\end{equation}
Substituting \eqref{3612}-\eqref{3615} into \eqref{3611}, we find
\begin{equation*}
\|\nabla \Delta d\|_1^2 \le C_3 N_m^3(t), ~~m\ge3,
\end{equation*}
which, together with \eqref{3610}, completes the proof of \eqref{363}.
By virtue of the \eqref{eq1}$_2$, \eqref{361} and \eqref{363},
one attains for $m \ge 3$
\begin{equation*}
\begin{aligned}
\|\nabla \Delta d\|_{L^\infty}^2
&\le \|\nabla(d_t+u\cdot \nabla d-|\nabla d|^2 d)\|_{L^\infty}^2\\
&\le \|\nabla d_t\|_{L^\infty}^2
      +\|\nabla u\|_{L^\infty}^2\|\nabla d\|_{L^\infty}^2
      +\|u\|_{L^\infty}^2\|\nabla^2 d\|_{L^\infty}^2\\
&\quad  +\|\nabla d\|_{L^\infty}^2\|\nabla^2 d\|_{L^\infty}^2
      +\|\nabla d\|_{L^\infty}^4\|\nabla d\|_{L^\infty}^2\\
&\le \|\nabla d_t\|_{L^\infty}^2+C_4 N_m^3(t),
\end{aligned}
\end{equation*}
which, together with \eqref{367-6}, completes the proof of estimate \eqref{364}.
\end{proof}

In order to give the estimate for $\|\nabla u\|_{\mathcal{H}^{1,\infty}}$,
we need the lemma as follows, refer to \cite{Masmoudi-Rousset} or \cite{Wang-Xin-Yong}.
\begin{lemm}\label{lemma3.7}
Consider $\rho$ a smooth solution of
\begin{equation}\label{371}
\partial_t \rho+u\cdot \nabla \rho=\varepsilon \partial_{zz}\rho+\mathcal{S},
\quad z >0,\quad \rho(t,y,0)=0
\end{equation}
for some smooth divergence free vector field $u$ such that $u\cdot n$
vanishes on the boundary. Assume that $\rho$ and $\mathcal{S}$ are
compact supported in $z$. Then, we have the estimate:
\begin{equation}\label{372}
\|\rho(t)\|_{\mathcal{H}^{1,\infty}}
\le C\|\rho_0\|_{\mathcal{H}^{1,\infty}}
+C\int_0^t ((\|u\|_{\mathcal{H}^{2,\infty}}+\|\partial_z u\|_{\mathcal{H}^{1,\infty}})
          (\|\rho\|_{\mathcal{H}^{1,\infty}}+\|\rho\|_{\mathcal{H}^{m_0+3}})
          +\|\mathcal{S}\|_{\mathcal{H}^{1,\infty}})d\tau
\end{equation}
for $m_0 \ge 2$.
\end{lemm}

Finally, one gives the estimate for the quantity $\|\nabla u\|_{\mathcal{H}^{1,\infty}}$.
\begin{lemm}\label{lemma3.8}
For $m \ge 6$, we have the estimate
\begin{equation}\label{381}
\|\nabla u\|_{\mathcal{H}^{1,\infty}}^2
\le C_{m+2} \left\{N_m(0)+ \|u\|_{\mathcal{H}^{m}}^2
     +\delta \varepsilon \int_0^t \|\nabla^2 u\|_{\mathcal{H}^4}^2 d\tau
    +C_\delta \int_0^t N_m^4(\tau)d\tau\right\}.
\end{equation}
\end{lemm}
\begin{proof}
Away from the boundary, we clearly have by the classical isotropic Sobolev
embedding that
\begin{equation}\label{382}
\|\chi \nabla u\|_{\mathcal{H}^{1,\infty}}\lesssim \|u\|_{\mathcal{H}^{m}},
\quad m \ge 4,
\end{equation}
where the support of $\chi$ is away from the boundary.
Consequently, by using a partition of unity subordinated to the covering
we only have to estimate $\|\chi_j \nabla u\|_{\mathcal{H}^{1,\infty}}, ~j\ge 1$.
For notational convenience, we shall denote $\chi_j$ by $\chi$.
Similar to \cite{Masmoudi-Rousset}, we use the local parametrization in the
neighborhood of the boundary given by a normal geodesic system in which the
Laplacian takes a convenient form. Denote
\begin{equation*}
\Psi^n(y, z)=
\left(
\begin{array}{c}
y\\
\psi(y)
\end{array}
\right)
-zn(y)=x,
\end{equation*}
where
\begin{equation*}
n(y)=\frac{1}{\sqrt{1+|\nabla \psi(y)|^2}}
\left(
\begin{array}{c}
\partial_1 \psi(y)\\
\partial_2 \psi(y)\\
-1
\end{array}
\right)
\end{equation*}
is the unit outward normal. As before, one can extend $n$ and $\Pi$
in the interior by setting
\begin{equation*}
n(\Psi^n(y,z))=n(y), \quad \Pi(\Psi^n(y,z))=\Pi(y)=I-n\otimes n,
\end{equation*}
where $I$ is the unit matrix.
Note that $n(y,z)$ and $\Pi(y,z)$ have different definitions from
the ones used before.
The advantages of this parametrization is that in the associated local
basis $(e_{y_1}, e_{y_2}, e_z)$ of $\mathbb{R}^3$, it holds that
$\partial_z=\partial_n$ and
\begin{equation*}
\left.(e_{y_i})\right|_{\Psi^n(y,z)}
\cdot \left.(e_z)\right|_{\Psi^n(y,z)}=0,\quad i=1,2.
\end{equation*}
The scalar product on $\mathbb{R}^3$ induces in this coordinate system
the Riemannian metric $g$ with the norm
\begin{equation*}
g(y,z)=
\left(
\begin{array}{cc}
\widetilde{g}(y,z)& 0\\
0 & 1
\end{array}
\right).
\end{equation*}
Therefore, the Laplacian in this coordinate system has the form
\begin{equation}\label{383}
\Delta f=\partial_{zz}f+\frac{1}{2}\partial_z(\ln |g|)\partial_z f+\Delta_{\widetilde{g}}f,
\end{equation}
where $|g|$ denotes the determinant of the matrix $g$,
and $\Delta_{\widetilde{g}}$ is defined by
\begin{equation*}
\Delta_{\widetilde{g}}f
=\frac{1}{\sqrt|\widetilde{g}|}\sum_{i,j=1,2}
\partial_{y_i}(\widetilde{g}^{ij}|\widetilde{g}|^{\frac{1}{2}}\partial_{y_j}f),
\end{equation*}
which  only involves the tangential derivatives and
$\{\widetilde{g}^{ij}\}$ is the inverse matrix to $g$.

Next, thanks to \eqref{301}(in the coordinate system that
we have just defined) and Lemma \ref{lemma3.6},
we have for $m \ge 4$
\begin{equation}\label{384}
\begin{aligned}
\|\chi \nabla u\|_{\mathcal{H}^{1,\infty}}^2
&\le C_3(\|\chi \Pi \partial_n u\|_{\mathcal{H}^{1,\infty}}^2
    +\|u\|_{\mathcal{H}^{2,\infty}}^2)\\
&\le C_3(\|\chi \Pi \partial_n u\|_{\mathcal{H}^{1,\infty}}^2+N_m(t)).
\end{aligned}
\end{equation}
Consequently, it suffices to estimate $\|\chi \Pi \partial_n u\|_{\mathcal{H}^{1,\infty}}$.
To this end, it is useful to use the vorticity $w=\nabla \times u$,
see \cite{{Masmoudi-Rousset},{Xiao-Xin}, {Wang-Xin-Yong}}.
Indeed, it is easy to deduce that
\begin{equation*}
\Pi(w\times n)=\Pi((\nabla u-\nabla u^t)\cdot n)
=\Pi(\partial_n u-\nabla(u\cdot n)+\nabla n^t \cdot u),
\end{equation*}
which implies
\begin{equation}\label{385}
\begin{aligned}
\|\chi \Pi \partial_n u\|_{\mathcal{H}^{1,\infty}}^2
&\le C_3(\|\chi \Pi(w\times n)\|_{\mathcal{H}^{1,\infty}}^2+\|u\|_{\mathcal{H}^{2,\infty}}^2)\\
&\le C_3(\|\chi \Pi(w\times n)\|_{\mathcal{H}^{1,\infty}}^2+N_m(t)),
\end{aligned}
\end{equation}
where we have used the Lemma \ref{lemma3.6} in the last inequality.
In other words, we only need to estimate
$\|\chi \Pi(w\times n)\|_{\mathcal{H}^{1,\infty}}$.
It is easy to see that $w$ solves the vorticity equation
\begin{equation}\label{386}
w_t+(u\cdot \nabla)w-\varepsilon \Delta w=F_1,
\end{equation}
where $F_1\triangleq w\cdot \nabla u-\nabla \times(\nabla d\cdot \Delta d)$.
In the support of $\chi$, let
\begin{equation*}
\widetilde{w}(y, z)=w(\Psi^n(y,z)),
\quad \widetilde{u}(y, z)=u(\Psi^n(y,z)),
\widetilde{d}(y, z)=d(\Psi^n(y,z)),
\end{equation*}
The combination of \eqref{383} and \eqref{386} yields directly
\begin{equation}\label{387}
\partial_t {\widetilde{w}}+\widetilde{u}^1\partial_{y_1}{\widetilde{w}}
+\widetilde{u}^2\partial_{y_2}{\widetilde{w}}
+\widetilde{u}\cdot n \partial_z {\widetilde{w}}
=\varepsilon(\partial_{zz}{\widetilde{w}}
+\frac{1}{2}\partial_z (\ln |g|)\partial_z {\widetilde{w}}
+\Delta_{\widetilde{g}} {\widetilde{w}})
+\widetilde{F}_1
\end{equation}
and
\begin{equation}\label{388}
\begin{aligned}
\partial_t {\widetilde{u}}+\widetilde{u}^1\partial_{y_1}{\widetilde{u}}
+\widetilde{u}^2\partial_{y_2}{\widetilde{u}}
+\widetilde{u}\cdot n \partial_z {\widetilde{u}}
=\varepsilon(\partial_{zz}{\widetilde{u}}
+\frac{1}{2}\partial_z (\ln |g|)\partial_z {\widetilde{u}}
+\Delta_{\widetilde{g}} {\widetilde{u}})
+\widetilde{F}_2,
\end{aligned}
\end{equation}
where $\widetilde{F}_2 =F_2(\Psi^n(y,z))$
and $F_2\triangleq-\nabla p-\nabla d\cdot \Delta d$.
Similar to \eqref{305} , we define
\begin{equation*}\label{389}
\widetilde{\eta}=\chi(\widetilde{w}\times n+\Pi(B\widetilde{u})).
\end{equation*}
It is easy to deduce taht $\widetilde{\eta}$ satisfies
\begin{equation*}\label{3810}
\widetilde{\eta}(y, 0)=0.
\end{equation*}
and solves the equation
\begin{equation}\label{3811}
\begin{aligned}
&\partial_t \widetilde{\eta}+\widetilde{u}^1\partial_{y_1}\widetilde{\eta}
+\widetilde{u}^2\partial_{y_2}\widetilde{\eta}
+\widetilde{u}\cdot n \partial_z \widetilde{\eta}\\
&=\varepsilon\left(\partial_{zz}\widetilde{\eta}
+\frac{1}{2}\partial_z(\ln |g|)\partial_z \widetilde{\eta}\right)
+\chi (\widetilde{F}_1 \times n)
+\chi \Pi(B \widetilde{F}_2)+F^\chi+\chi F^{\kappa},
\end{aligned}
\end{equation}
where the source terms are given by
\begin{equation}\label{3812}
\begin{aligned}
F^\chi
&=\left[(\widetilde{u}^1\partial_{y_1}
+\widetilde{u}^2\partial_{y_2}
+\widetilde{u}\cdot n \partial_z)\chi\right]
(\widetilde{w}\times n+\Pi(Bu))\\
&-\varepsilon\left(\partial_{zz}\chi+2 \partial_z \chi \partial_z
      +\frac{1}{2}\partial_z(\ln |g|)\partial_z \chi\right)
      (\widetilde{w}\times n+\Pi(Bu)),
\end{aligned}
\end{equation}
and
\begin{equation}\label{3813}
\begin{aligned}
F^\kappa
=&(\widetilde{u}^1 \partial_{y_1}\Pi+\widetilde{u}^2 \partial_{y_2}\Pi)\cdot (B\widetilde{u})
+w\times (\widetilde{u}^1 \partial_{y_1}n+\widetilde{u}^2 \partial_{y_2}n)\\
&+\Pi\left[(\widetilde{u}^1 \partial_{y_1}+\widetilde{u}^2 \partial_{y_2}
            +u\cdot n \partial_{y_3})B\cdot u\right]
            +\varepsilon \Delta_{\widetilde{g}}\widetilde{w} \times n
            +\varepsilon \Pi (B \Delta_{\widetilde{g}}\widetilde{u}).
\end{aligned}
\end{equation}
Note that in the derivation of the source terms above, in particular,
$F^\kappa$, which contains all the commutators coming from the fact that
$n$ and $\Pi$ are not constant, we have used the fact that in the coordinate system
just defined, $n$ and $\Pi$ do not depend on the normal variable.
Since $\Delta_{\widetilde{g}}$ involves only the tangential derivatives,
and the derivatives of $\chi$ are compactly supported away from the boundary,
the following estimates hold for $m \ge 6$
\begin{equation}\label{3814}
\| F^\chi\|_{\mathcal{H}^{1,\infty}}^2
\le C_3(\|u\|_{\mathcal{H}^{1,\infty}}^2\|u\|_{\mathcal{H}^{2,\infty}}^2
+\varepsilon \|u\|_{\mathcal{H}^{3,\infty}}^2)
\le C_3 N^2_m(t),
\end{equation}
\begin{equation}\label{3815}
\|\chi (\widetilde{F}^1 \times n)\|_{\mathcal{H}^{1,\infty}}^2
\le C_2(\|\nabla u\|_{\mathcal{H}^{1,\infty}}^2
+\|\nabla d\|_{\mathcal{H}^{1,\infty}}^2\|\nabla \Delta d\|_{\mathcal{H}^{1,\infty}}^2)
\le C_2 N^4_m(t),
\end{equation}
\begin{equation}\label{3816}
\|\chi \Pi(B \widetilde{F}_2)\|_{\mathcal{H}^{1,\infty}}^2
\le C_3(\|\nabla p\|_{\mathcal{H}^{1,\infty}}^2
+\|\nabla d\|_{\mathcal{H}^{1,\infty}}^2\|\Delta d\|_{\mathcal{H}^{1,\infty}}^2)
\le C_3 N^3_m(t),
\end{equation}
and
\begin{equation}\label{3817}
\begin{aligned}
\|\chi F^\kappa\|_{\mathcal{H}^{1,\infty}}^2
&\le C_4\{\|u\|_{\mathcal{H}^{1,\infty}}^2\|\nabla u\|_{\mathcal{H}^{1,\infty}}^2
+\varepsilon^2(\|\nabla u\|_{\mathcal{H}^{3,\infty}}^2+\|u\|_{\mathcal{H}^{3,\infty}}^2)\}\\
&\le \delta \varepsilon \|\nabla^2 u\|_{\mathcal{H}^4}^2
+C_4\left\{\delta \varepsilon^3 N_m(t)+N^2_m(t)\right\}.
\end{aligned}
\end{equation}
It follows from \eqref{3812}-\eqref{3817} that
\begin{equation}\label{3818}
\|F\|_{\mathcal{H}^{1,\infty}}^2
\le \delta \varepsilon \|\nabla^2 u\|_{\mathcal{H}^4}^2
+C_4 \left\{C_\delta \varepsilon^3 N_m(t)+N^4_m(t)\right\},\quad m\ge 6,
\end{equation}
where $F\triangleq\chi (\widetilde{F}_1 \times n)
+\chi \Pi(B \widetilde{F}_2)+F^\chi+\chi F^{\kappa}$.
In order to be able to use Lemma \ref{lemma3.7}, we shall perform last change of
unknown in order to eliminate the term
$\partial_z(\ln |\widetilde{g}|)\partial_z \widetilde{\eta}$.
We set
$$
\widetilde{\eta}=\frac{1}{|g|^{\frac{1}{4}}}\overline{\eta}=\overline{\gamma} ~\overline{\eta}.
$$
Note that we have
\begin{equation}\label{3819}
\|\widetilde{\eta}\|_{\mathcal{H}^{1,\infty}}
\le C_3 \|\overline{\eta}\|_{\mathcal{H}^{1,\infty}},
\quad \|\overline{\eta}\|_{\mathcal{H}^{1,\infty}}
\le C_3 \|\widetilde{\eta}\|_{\mathcal{H}^{1,\infty}}
\end{equation}
and that, $\overline{\eta}$ solves the equation
\begin{equation*}
\begin{aligned}
&\partial_t \overline{\eta}+\widetilde{u}^1\partial_{y_1}\overline{\eta}
+\widetilde{u}^2\partial_{y_2}\overline{\eta}
+\widetilde{u}\cdot n \partial_z \overline{\eta}
-\varepsilon \partial_{zz}\overline{\eta}\\
&=\frac{1}{\overline{\gamma}}\left(\widetilde{F}
+\varepsilon\partial_{zz}\overline{\gamma}\cdot \overline{\eta}
+\frac{\varepsilon}{2}\partial_z(\ln |g|)\partial_z \overline{\gamma}\cdot \overline{\eta}
-(\widetilde{u}\cdot \nabla \overline{\gamma} )\overline{\eta}\right)
:=\mathcal{S}.
\end{aligned}
\end{equation*}
Hence, it is easy to deduce for $m \ge 6$
\begin{equation}\label{3820}
\begin{aligned}
\|\mathcal{S}\|^2_{\mathcal{H}^{1,\infty}}
&\le C_4(\|F\|^2_{\mathcal{H}^{1,\infty}}
+\|\overline{\eta}\|^2_{\mathcal{H}^{1,\infty}}
+\|\widetilde{u}\|^2_{\mathcal{H}^{1,\infty}}\|\overline{\eta}\|^2_{\mathcal{H}^{1,\infty}})\\
&\le C_4\left\{\delta \varepsilon \|\nabla^2 u\|_{\mathcal{H}^4}^2
+C_\delta \varepsilon^3 N_m(t)+N^4_m(t)\right\}.
\end{aligned}
\end{equation}
Consequently, by using Lemma \ref{lemma3.7}, we get that for $m \ge 6$
\begin{equation}\label{3821}
\begin{aligned}
\|\overline{\eta}\|_{\mathcal{H}^{1,\infty}}
&\lesssim \|\overline{\eta}_0\|_{\mathcal{H}^{1,\infty}}
+\!\!\int_0^t\!\! ((\|\widetilde{u}\|_{\mathcal{H}^{2,\infty}}
+\|\partial_z \widetilde{u}\|_{\mathcal{H}^{1,\infty}})
          (\|\overline{\eta}\|_{\mathcal{H}^{1,\infty}}
          \!+\!\|\overline{\eta}\|_{\mathcal{H}^{m_0+3}})
          \!+\!\|\mathcal{S}\|_{\mathcal{H}^{1,\infty}})d\tau\\
&\lesssim \|\widetilde{{\eta}}_0\|_{\mathcal{H}^{1,\infty}}
+\!\!\int_0^t\!\! ((\|u\|_{\mathcal{H}^{2,\infty}}+\|\nabla u\|_{\mathcal{H}^{1,\infty}})
          (\|\widetilde{{\eta}}\|_{\mathcal{H}^{1,\infty}}
          \!+\!\|\widetilde{{\eta}}\|_{H^{m_0+3}})
          \!+\!\|\mathcal{S}\|_{\mathcal{H}^{1,\infty}})d\tau\\
&\lesssim \|\widetilde{{\eta}}_0\|_{\mathcal{H}^{1,\infty}}
+\int_0^t (N_m(\tau)+\|\mathcal{S}\|_{\mathcal{H}^{1,\infty}})d\tau.
\end{aligned}
\end{equation}
Then, we deduce from  \eqref{3818}-\eqref{3821} that
\begin{equation*}
\|\eta(t)\|^2_{\mathcal{H}^{1,\infty}}
\le \|\eta_0\|^2_{\mathcal{H}^{1,\infty}}
+C C_4\left\{\delta \varepsilon \int_0^t \|\nabla^2 u\|_{\mathcal{H}^4}^2 d\tau
    +\int_0^t N_m^4(\tau)d\tau\right\},
\end{equation*}
which, together with  \eqref{382}, gives directly
\begin{equation*}
\|\nabla u\|_{\mathcal{H}^{1,\infty}}^2
\le C_4 \left\{N_m(0)+\|u\|_{\mathcal{H}^{m}}^2
+\delta \varepsilon \int_0^t \|\nabla^2 u\|_{\mathcal{H}^4}^2 d\tau
    +C_\delta \int_0^t N_m^4(\tau)d\tau \right\}.
\end{equation*}
Therefore, we complete the proof of the lemma \ref{lemma3.8}.
\end{proof}

\subsection{Proof of Theorem \ref{Theoream3.1}}
\quad
By virtue of the definition of $N_m(t)$ and $Q(t)$,
one deduces from the estimates in Lemma \ref{lemma3.6} that
\begin{equation*}
Q(t)\le C C_4 P(N_m(t)),
\end{equation*}
which, together with \eqref{3204}, gives directly
\begin{equation}\label{3205}
\begin{aligned}
&\underset{0\le \tau \le t}{\sup}
(\|(u, \nabla d)\|_{\mathcal{H}^{m}}^2+\|(\nabla u, \Delta d)\|_{\mathcal{H}^{m-1}}^2)
+\varepsilon \int_0^t \|\nabla u\|_{\mathcal{H}^{m}}^2 d\tau\\
&\quad  +\varepsilon \int_0^t \|\nabla^2 u\|_{\mathcal{H}^{m-1}}^2 d\tau
+\int_0^t \|\Delta d\|_{\mathcal{H}^{m}}^2 d\tau
+\int_0^t \|\nabla \Delta d\|_{\mathcal{H}^{m-1}}^2 d\tau\\
&\le CC_{m+2}\left\{\|(u_0, \nabla d_0)\|_{\mathcal{H}^{m}}^2
    +\|(\nabla u_0, \Delta d_0)\|_{\mathcal{H}^{m-1}}^2
    +P(N_m(t))\int_0^t N_m(t)d\tau \right\}.
\end{aligned}
\end{equation}
By virtue of the basic fact $|d|=1$(see \eqref{unit}),
the combination of \eqref{381} and \eqref{3205} yields that
\begin{equation*}
\begin{aligned}
&\underset{0\le \tau \le t}{\sup}N_m(\tau)
+\varepsilon \int_0^t \|\nabla u\|_{\mathcal{H}^{m}}^2 d\tau
+\varepsilon \int_0^t \|\nabla^2 u\|_{\mathcal{H}^{m-1}}^2 d\tau
+\int_0^t \|\Delta d\|_{\mathcal{H}^{m}}^2 d\tau\\
&\quad
+\int_0^t \|\nabla \Delta d\|_{\mathcal{H}^{m-1}}^2 d\tau
\le \widetilde{C}_2C_{m+2}\left\{N_m(0)+P(N_m(t))\int_0^t P(N_m(\tau)) d\tau\right\}.
\end{aligned}
\end{equation*}
Therefore, we complete the proof of Theorem \ref{Theoream3.1}.

\section{Proof of Theorem \ref{Theorem1.1} (Uniform Regularity)}\label{Proof1}
\quad
In this section, we will give the proof for the Theorem \ref{Theorem1.1}.
Indeed, we shall indicate how to combine the a priori estimates obtained
so far to prove the uniform existence result.
Fixing $m \ge 6$, we consider the initial data
$(u_0^\varepsilon, d_0^\varepsilon)\in X_{n,m}$ such that
\begin{equation}
\mathcal{I}_m(0)=\underset{0< \varepsilon \le 1}{\sup}
                 \|(u_0^\varepsilon, d_0^\varepsilon)\|_{X_{n,m}}
                 \le \widetilde{C}_0.
\end{equation}
For such initial data, since we are not aware of a local existence
result for \eqref{eq1} and \eqref{bc1}(or \eqref{bc2}),
we first establish the local existence of solution for \eqref{eq1} and \eqref{bc1}
with initial data $(u_0^\varepsilon, d_0^\varepsilon)\in X_{n,m}$.
For such initial data $(u_0^\varepsilon, d_0^\varepsilon)$,
it is easy to see that there exists a sequence of smooth approximate initial data
$(u_0^{\varepsilon, \delta}, d_0^{\varepsilon,\delta})\in X^{ap}_{n,m}$
($\delta$ being a regularity parameter), which has enough space regularity
so that the time derivatives at the initial time can be defined by the equations
\eqref{eq1} and the boundary compatibility conditions are satisfied.
Fixed $\varepsilon \in (0, 1]$, one constructs the approximate solutions as follows:\\
(1)Define $u^0=u_0^{\varepsilon, \delta}$
and $d^0=d_0^{\varepsilon,\delta}$.\\
(2)Assume that $(u^{k-1}, d^{k-1})$  has been defined for $k \ge 1$.
Let $(u^k, d^k)$ be the unique solution to the following linearized
initial data boundary value problem:
\begin{equation}\label{4eq1}
\left\{
\begin{aligned}
&u^k_t+u^{k-1}\cdot \nabla u^k-\varepsilon \Delta u^k+\nabla p^k
=-\nabla d^k \cdot \Delta d^k, & (x, t)\in \Omega \times (0, T),\\
&d^k_t-\Delta d^k=|\nabla d^{k-1}|^2 d^{k-1}-u^{k-1}\cdot \nabla d^{k-1},
& (x, t)\in \Omega \times (0, T),\\
&{\rm div}u^k=0, & (x, t)\in \Omega \times (0, T),\\
\end{aligned}
\right.
\end{equation}
with initial data
\begin{equation}\label{4id}
(u^k, d^k)|_{t=0}=(u_0^{\varepsilon,\delta}, d_0^{\varepsilon,\delta}),
\end{equation}
and Navier-type and Neumann boundary condition
\begin{equation}\label{4bc}
u^k\cdot n=0,
\quad n\times (\nabla \times u^k)=[Bu^k]_\tau,
\quad {\rm and}\quad
\frac{\partial d^k}{\partial n}=0,
\quad {\rm on}~\partial \Omega.
\end{equation}
Since $u^k$ and $d^k$ are decoupled, the existence of global unique smooth
solution $(u^k, d^k)(t)$ of \eqref{4eq1}-\eqref{4bc} can be obtained by using
classical methods, for example, refer to \cite{Wen-Ding} or \cite{Gao-Tao-Yao}.
On the other hand, by virtue of
$(u_0^{\varepsilon, \delta}, d_0^{\varepsilon, \delta})\in H^{4m}\times H^{4(m+1)}$,
one proves that there exists a positive time $\widetilde{T}_1=\widetilde{T}_1(\varepsilon)$
(depending on $\varepsilon$,
$\|u_0^{\varepsilon, \delta}\|_{H^{4m}}$
and
$\| d_0^{\varepsilon, \delta}\|_{H^{4(m+1)}}$) such that
\begin{equation}\label{41}
\|u^k(t)\|_{H^{4m}}^2
+\|d^k(t)\|_{H^{4(m+1)}}^2 \le \widetilde{C}_1 \quad {\rm for}~0\le t \le \widetilde{T}_1,
\end{equation}
where the constant $\widetilde{C}_1$ depends on $\widetilde{C}_0, \varepsilon^{-1}$,
$\|u_0^{\varepsilon,\delta}\|_{H^{4m}}$
and
$\|d_0^{\varepsilon, \delta}\|_{H^{4(m+1)}}$.
Based on the above uniform time $\widehat{T}_1(\le \widetilde{T}_1)$(independent of $k$)
such that $(u^k, d^k)$ converges to a limit
$(u^{\varepsilon, \delta}, d^{\varepsilon, \delta})$
as $k \rightarrow +\infty$ in the following strong sense:
$$
u^k \rightarrow u^{\varepsilon, \delta} ~{\rm in}~L^\infty(0, \widehat{T}_1; L^2)
\quad {\rm and}\quad \nabla u^k \rightarrow \nabla u^{\varepsilon, \delta}
~{\rm in}~L^2(0, \widehat{T}_1; L^2),
$$
and
$$
d^k \rightarrow d^{\varepsilon, \delta} ~{\rm in}~L^\infty(0, \widehat{T}_1; H^1)
\quad {\rm and}\quad \Delta d^k \rightarrow \Delta d^{\varepsilon, \delta}
~{\rm in}~L^2(0, \widehat{T}_1; L^2).
$$
It is easy to check that $(u^{\varepsilon, \delta}, d^{\varepsilon, \delta})$
is a classical solution to the problem \eqref{eq1} and \eqref{bc1}
with initial data $(u_0^{\varepsilon, \delta}, d_0^{\varepsilon, \delta})$.
In view of the lower semicontinuity  of norms, one can deduce from
the uniform bounds  \eqref{41} that
\begin{equation}\label{42}
\|u^{\varepsilon, \delta}(t)\|_{H^{4m}}^2
+
\|d^{\varepsilon, \delta}(t)\|_{H^{4(m+1)}}^2
\le \widetilde{C}_1 \quad {\rm for}~0\le t \le \widetilde{T}_1.
\end{equation}
Applying the a priori estimates given in Theorem \ref{Theoream3.1} to
the solution
$(u^{\varepsilon, \delta}, d^{\varepsilon, \delta})$, one can obtain a uniform
time $T_0$ and constant $C_3$(independent of $\varepsilon$ and $\delta$) such that
for all $t \in [0, \min\{T_0, \widehat{T}_1\}]$
\begin{equation}\label{42}
\underset{0\le \tau \le t}{\sup}N_m(\tau)
+\varepsilon \int_0^t (\|\nabla u\|_{\mathcal{H}^{m}}^2
+\|\nabla^2 u\|_{\mathcal{H}^{m-1}}^2) d\tau
+\int_0^t (\|\Delta d\|_{\mathcal{H}^{m}}^2
+\|\nabla \Delta d\|_{\mathcal{H}^{m-1}}^2) d\tau
\le \widetilde{C}_3,
\end{equation}
where $T_0$ and $\widetilde{C}_3$ depend only on $\widehat{C}_0$
and $\mathcal{I}_m(0)$. Based on the uniform estimate \eqref{42} for
$(u^{\varepsilon, \delta}, d^{\varepsilon, \delta})$, one can pass
the limit $\delta \rightarrow 0$ to get a strong solution
$(u^{\varepsilon}, d^{\varepsilon})$ of \eqref{eq1} and \eqref{bc1}
with initial data $(u_0^\varepsilon, d_0^\varepsilon)$ satisfying
\eqref{4eq1} by using a strong compactness arguments(see \cite{Simon}).
Indeed, it follows from
\eqref{42} that $(u^{\varepsilon, \delta}, \nabla d^{\varepsilon, \delta})$
is bounded uniformly in $L^\infty(0, \widetilde{T}_2; H_{co}^m)$,
where $\widetilde{T}_2=\min\{T_0, \widetilde{T}_1\}$, while
$(\nabla u^{\varepsilon, \delta}, \Delta d^{\varepsilon, \delta})$
is bounded uniformly in $L^\infty(0, \widetilde{T}_2; H_{co}^{m-1})$,
and $(\partial_t u^{\varepsilon, \delta}, \partial_t \nabla d^{\varepsilon, \delta})$
is bounded uniformly in $L^\infty(0, \widetilde{T}_2; H_{co}^{m-1})$.
Then, the strong compactness argument implies that
$(u^{\varepsilon, \delta}, \nabla d^{\varepsilon, \delta})$
is compact in $\mathcal{C}([0, \widetilde{T}_2]; H_{co}^{m-1})$.
In particular, there exists a sequence $\delta_n \rightarrow 0^+$
and $(u^\varepsilon, \nabla d^\varepsilon)\in \mathcal{C}([0, \widetilde{T}_2]; H_{co}^{m-1})$
such that
\begin{equation*}
(u^{\varepsilon, \delta_n}, \nabla d^{\varepsilon, \delta_n})
\rightarrow  (u^{\varepsilon, \delta}, \nabla d^{\varepsilon, \delta})
~{\rm in}~\mathcal{C}([0, \widetilde{T}_2]; H_{co}^{m-1})~{\rm as}~
\delta_n \rightarrow 0^+.
\end{equation*}
Moreover, applying the lower semicontinuity of norms to the bounds
\eqref{42}, one obtains the bounds \eqref{42} for $(u^\varepsilon, d^\varepsilon)$.
It follows from the bounds of \eqref{42} for $(u^\varepsilon, d^\varepsilon)$,
and the anisotropic Sobolev inequality \eqref{24} that
$$
\begin{aligned}
&\underset{0\le t\le \widetilde{T}_2}{\sup}
\|(u^{\varepsilon,\delta_n}-u^{\varepsilon}, d^{\varepsilon,\delta_n}-d^{\varepsilon})\|_{L^\infty}^2\\
&\le C\underset{0\le t\le \widetilde{T}_2}{\sup}
  \|\nabla(u^{\varepsilon,\delta_n}-u^{\varepsilon}, d^{\varepsilon,\delta_n}-d^{\varepsilon})\|_{H^1_{co}}
  \|(u^{\varepsilon,\delta_n}-u^{\varepsilon},
     d^{\varepsilon,\delta_n}-d^{\varepsilon})\|_{H^2_{co}}
  \rightarrow 0,
\end{aligned}
$$
and
$$
\underset{0\le t\le \widetilde{T}_2}{\sup}
\|\nabla(d^{\varepsilon,\delta_n}-d^{\varepsilon})\|_{L^\infty}^2
\le C\underset{0\le t\le \widetilde{T}_2}{\sup}
  \|\Delta(d^{\varepsilon,\delta_n}-d^{\varepsilon})\|_{H^1_{co}}
  \|\nabla(d^{\varepsilon,\delta_n}-d^{\varepsilon})\|_{H^2_{co}}
  \rightarrow 0,
$$
Hence, it is easy to check that $(u^\varepsilon, d^\varepsilon)$
is a weak solution of the nematic liquid crystal flows \eqref{eq1}.
The uniqueness of the solution $(u^\varepsilon, d^\varepsilon)$
comes directly from the Lipschitz regularity of solution.
Thus, the whole family $(u^{\varepsilon, \delta}, d^{\varepsilon, \delta})$
converge to $(u^{\varepsilon}, d^{\varepsilon})$.
Therefore,
we have established the local solution of equation \eqref{eq1}
and \eqref{bc1} with initial data
$(u_0^{\varepsilon}, d_0^{\varepsilon})\in X_{n,m},~t\in [0, T_2]$.

We shall use the local existence results to prove Theorem \ref{Theorem1.1}.
If $T_0 \le \widetilde{T}$, then Theorem \ref{Theorem1.1} follows
from \eqref{42} with $\widetilde{C}_1=\widetilde{C}_3$.
On the other hand, for the case $\widetilde{T} \le  T_0$,
based on the uniform estimate \eqref{42}, we can use the local existence
results established above to extend our solution step by step to the
uniform time interval $t\in [0, T_0]$.
Therefore, we complete the proof of Theorem \ref{Theorem1.1}.

\section{Proof of Theorem \ref{Theorem1.2} (Inviscid Limit)}\label{Proof2}

\quad In this section, we study the vanishing viscosity of solutions
for the equation \eqref{eq1} to the solution for the equation \eqref{eq2}
with a rate of convergence. It is easy to see that the solution
$(u, d)\in H^3 \times H^4$ of equation \eqref{eq1} and \eqref{bc1}
with initial data $(u_0, d_0)\in H^3 \times H^4$ satisfies
$$
\|u\|_{C([0, T_1]; H^{3})}
+\|d\|_{C([0, T_1]; H^{4})}
\le \widetilde{C}_4
$$
where $\widetilde{C}_4$ depends only on $\|(u_0, d_0)\|_{H^3 \times H^4}$.
On the other hand, it follows from the Theorem \ref{Theorem1.1}
that the solution $(u^\varepsilon, d^\varepsilon)$ of equation
\eqref{eq1} and \eqref{bc1} with initial data $(u_0, d_0)$
satisfies
$$
\|(u^\varepsilon, d^\varepsilon)\|_{X_m}\le \widetilde{C}_1,
\quad \forall t\in [0, T_0],
$$
where $T_0$ and $\widetilde{C}_1$ are defined in Theorem \ref{Theorem1.1}.
In particular, this uniform regularity implies the bound
$$
\|u^\varepsilon\|_{W^{1,\infty}}+\|d^\varepsilon\|_{W^{2,\infty}}\le \widetilde{C}_1,
$$
which plays an important role in the proof of Theorem \ref{Theorem1.2}.

Let us define
\begin{equation*}
v^\varepsilon=u^\varepsilon-u, \quad \varphi^\varepsilon=d^\varepsilon-d.
\end{equation*}
It then follows from \eqref{eq1} that
\begin{equation}\label{eq3}
\left\{
\begin{aligned}
&\partial_t v^\varepsilon+u\cdot \nabla v^\varepsilon
+\varepsilon \nabla \times (\nabla \times v^\varepsilon)+\nabla(p^\varepsilon-p)
=R_1^\varepsilon,\\
&{\rm div}v^\varepsilon=0,\\
&\partial_t \varphi^\varepsilon+u\cdot \nabla \varphi^\varepsilon
-\Delta \varphi^\varepsilon=R_2^\varepsilon,\\
\end{aligned}
\right.
\end{equation}
where
\begin{equation*}
\begin{aligned}
&R_1^\varepsilon\triangleq\varepsilon \Delta u-v^\varepsilon\cdot \nabla u^\varepsilon
                 -\nabla d^\varepsilon \cdot \Delta \varphi^\varepsilon
                 -\nabla \varphi^\varepsilon \cdot \Delta d,\\
&R_2^\varepsilon\triangleq-v^\varepsilon\cdot \nabla d^\varepsilon
                 +(\nabla \varphi^\varepsilon : \nabla (d^\varepsilon+d))d^\varepsilon
                 +|\nabla d|^2 \varphi^\varepsilon.
\end{aligned}
\end{equation*}
The boundary conditions to \eqref{eq3} are
\begin{equation}\label{bc3}
\left\{
\begin{aligned}
& v^\varepsilon \cdot n=0,
\quad n\times (\nabla \times v^\varepsilon)
=[B v^\varepsilon]_\tau+[Bu]_\tau-n\times w, \quad x\in \partial \Omega,\\
&\frac{\partial \varphi^\varepsilon}{\partial n}=0, \quad x\in \partial \Omega.
\end{aligned}
\right.
\end{equation}

\begin{lemm}\label{lemma5.1}
For $t \in [0, \min\{T_0, T_1\}]$, it holds that
\begin{equation}\label{511}
\underset{0\le \tau \le t}{\sup}(\|v^\varepsilon(\tau)\|_{L^2}
       +\|\varphi^\varepsilon(\tau)\|_{H^1}^2)
+\varepsilon \int_0^t \! \!\int |\nabla v^\varepsilon|^2 dx d\tau
+\int_0^t \! \! \int(|\nabla \varphi^\varepsilon|^2+|\Delta \varphi^\varepsilon|^2)dx d\tau
\le C \varepsilon^{\frac{3}{2}}.
\end{equation}
where $C>0$ depend only on $\widetilde{C}_0, \widetilde{C}_1$ and $\widetilde{C}_4$.
\end{lemm}
\begin{proof}
Multiplying \eqref{eq3}$_1$ by $v^\varepsilon$ and integrating by parts, we find
\begin{equation}\label{512}
\frac{1}{2}\frac{d}{dt}\int |v^\varepsilon|^2 dx
+\varepsilon \int \nabla \times (\nabla \times v^\varepsilon)\cdot v^\varepsilon dx
=\int R_1^\varepsilon \cdot v^\varepsilon dx.
\end{equation}
Integrating by part and applying the boundary condition \eqref{bc3},
one arrives at directly
\begin{equation}\label{513}
\begin{aligned}
&\int \nabla \times (\nabla \times v^\varepsilon)\cdot v^\varepsilon dx\\
&=\int_{\partial \Omega} n\times (\nabla \times v^\varepsilon)\cdot v^\varepsilon dx
   +\int |\nabla \times v^\varepsilon|^2 dx\\
&=\int_{\partial \Omega}([B v^\varepsilon]_\tau+[B u]_\tau-n\times w)\cdot v^\varepsilon d\sigma
  +\int |\nabla \times v^\varepsilon|^2 dx\\
&\le C|v^\varepsilon|_{L^2(\partial \Omega)}^2+C|v^\varepsilon|_{L^2(\partial \Omega)}
      +\int |\nabla \times v^\varepsilon|^2 dx.
\end{aligned}
\end{equation}
The application of Proposition \ref{prop2.1} gives directly
\begin{equation}\label{514}
\|\nabla v^\varepsilon\|_{L^2}^2
\le C(\|\nabla \times v^\varepsilon\|_{L^2}^2+\|v^\varepsilon\|_{L^2}^2).
\end{equation}
The combination of \eqref{512}-\eqref{514} yields that
\begin{equation}\label{515}
\begin{aligned}
&\frac{1}{2}\frac{d}{dt}\int |v^\varepsilon|^2 dx
+C_1 \varepsilon \int |\nabla v^\varepsilon|^2 dx\\
&\le \int R_1^\varepsilon \cdot v^\varepsilon dx
+C\varepsilon|v^\varepsilon|_{L^2(\partial \Omega)}^2
+C\varepsilon|v^\varepsilon|_{L^2(\partial \Omega)}
+C\varepsilon\|v^\varepsilon\|_{L^2}^2.
\end{aligned}
\end{equation}
By virtue of the H\"{o}lder and Cauchy inequalities, we obtain
\begin{equation*}
\begin{aligned}
\int R_1^\varepsilon \cdot v^\varepsilon dx
&\le \delta \|\Delta \varphi^\varepsilon\|_{L^2}^2
     +\varepsilon\|\Delta u\|_{L^\infty}\|v^\varepsilon\|_{L^2}
     +\|\nabla u^\varepsilon\|_{L^\infty}\|v^\varepsilon\|_{L^2}^2\\
&\quad  +C_\delta\|\nabla d^\varepsilon\|_{L^\infty}^2\|v^\varepsilon\|_{L^2}^2
        +\|\Delta d\|_{L^\infty}\|\nabla \varphi^\varepsilon\|_{L^2}\|v^\varepsilon\|_{L^2}\\
&\le \delta \|\Delta \varphi^\varepsilon\|_{L^2}^2
     +\varepsilon\|v^\varepsilon\|_{L^2}
     +C_\delta(\|\nabla \varphi^\varepsilon\|_{L^2}^2
     +\|v^\varepsilon\|_{L^2}^2),
\end{aligned}
\end{equation*}
which, together  with \eqref{515} , gives immediately
\begin{equation}\label{516}
\begin{aligned}
&\frac{1}{2}\frac{d}{dt}\int |v^\varepsilon|^2 dx
+C_1 \varepsilon \int |\nabla v^\varepsilon|^2 dx\\
&\le C\varepsilon|v^\varepsilon|_{L^2(\partial \Omega)}^2
+C\varepsilon|v^\varepsilon|_{L^2(\partial \Omega)}
+\delta \|\Delta \varphi^\varepsilon\|_{L^2}^2
     +\varepsilon\|v^\varepsilon\|_{L^2}
     +C_\delta(\|\nabla \varphi^\varepsilon\|_{L^2}^2
     +\|v^\varepsilon\|_{L^2}^2).
\end{aligned}
\end{equation}
Multiplying \eqref{eq3} by $-\Delta \varphi^\varepsilon$ and integrating over $\Omega$, we find
\begin{equation}\label{517}
\begin{aligned}
-\int \partial_t \varphi^\varepsilon \cdot \Delta \varphi^\varepsilon dx
+\int |\Delta \varphi^\varepsilon|^2 dx
=
-\int (u\cdot \nabla \varphi^\varepsilon) \cdot \Delta \varphi^\varepsilon dx
-\int R_2^\varepsilon \cdot \Delta \varphi^\varepsilon dx.
\end{aligned}
\end{equation}
Integrating by part and applying the boundary condition \eqref{bc3}, it holds that
\begin{equation}\label{518}
\begin{aligned}
-\int \partial_t \varphi^\varepsilon \cdot \Delta \varphi^\varepsilon dx
&=-\int_{\partial \Omega} \partial_t \varphi^\varepsilon \cdot (n\cdot \nabla \varphi^\varepsilon) d\sigma
+\frac{1}{2}\frac{d}{dt}\int |\nabla \varphi^\varepsilon|^2 dx
=\frac{1}{2}\frac{d}{dt}\int |\nabla \varphi^\varepsilon|^2 dx.
\end{aligned}
\end{equation}
Applying the Cauchy inequality, it is easy to deduce that
\begin{equation}\label{519}
\begin{aligned}
&-\int (u\cdot \nabla \varphi^\varepsilon) \cdot \Delta \varphi^\varepsilon dx
-\int R_2^\varepsilon \cdot \Delta \varphi^\varepsilon dx\\
&\le \delta \|\Delta \varphi^\varepsilon\|_{L^2}^2
     +C_\delta \|u\|_{L^\infty}^2\|\nabla \varphi^\varepsilon\|_{L^2}^2
     +C_\delta \|\nabla d^\varepsilon\|_{L^\infty}^2
     (\|\nabla \varphi^\varepsilon\|_{L^2}^2+\|v^\varepsilon\|_{L^2}^2)\\
&\quad   +C_\delta \|\nabla d\|_{L^\infty}^2
     (\|\nabla \varphi^\varepsilon\|_{L^2}^2+\|\varphi^\varepsilon\|_{L^2}^2)\\
&\le \delta \|\Delta \varphi^\varepsilon\|_{L^2}^2
     +C_\delta(\|v^\varepsilon\|_{L^2}^2+
          \|\varphi^\varepsilon\|_{L^2}^2+\|\nabla \varphi^\varepsilon\|_{L^2}^2).
\end{aligned}
\end{equation}
Substituting \eqref{518}-\eqref{519} into \eqref{517}, we obtain
\begin{equation}\label{5110}
\begin{aligned}
\frac{1}{2}\frac{d}{dt}\int |\nabla \varphi^\varepsilon|^2 dx
+\frac{3}{4}\int |\Delta \varphi^\varepsilon|^2 dx
\le
C(\|v^\varepsilon\|_{L^2}^2+
          \|\varphi^\varepsilon\|_{L^2}^2+\|\nabla \varphi^\varepsilon\|_{L^2}^2).
\end{aligned}
\end{equation}
In order to control the term $\int |\varphi^\varepsilon|^2 dx$ on the right hand side of
\eqref{5110}, we multiply the equation \eqref{eq3}$_3$ by
$\varphi^\varepsilon$ and integrating by part to get that
\begin{equation}\label{5111}
\frac{1}{2}\frac{d}{dt}\int |\varphi^\varepsilon|^2 dx
+\int |\nabla \varphi^\varepsilon|^2 dx
=\int R_2^\varepsilon \cdot \varphi^\varepsilon dx.
\end{equation}
In view of the H\"{o}lder inequality, one arrives at
\begin{equation*}
\begin{aligned}
\int R_2^\varepsilon \cdot \varphi^\varepsilon dx
&\le \|\nabla d^\varepsilon\|_{L^\infty}
     (\|v^\varepsilon\|_{L^2}\|\varphi^\varepsilon\|_{L^2}
      +\|\varphi^\varepsilon\|_{L^2}\|\nabla \varphi^\varepsilon\|_{L^2})\\
&\quad  + \|\nabla d\|_{L^\infty}
     \|\nabla \varphi^\varepsilon\|_{L^2}\|\varphi^\varepsilon\|_{L^2}
      +\|\nabla d\|_{L^\infty}^2\|\varphi^\varepsilon\|_{L^2}^2)\\
&\le
C(\|v^\varepsilon\|_{L^2}^2+
          \|\varphi^\varepsilon\|_{L^2}^2+\|\nabla \varphi^\varepsilon\|_{L^2}^2),
\end{aligned}
\end{equation*}
which, together with \eqref{516}, \eqref{5110} and \eqref{5111} , yields directly
\begin{equation}\label{5112}
\begin{aligned}
&\frac{1}{2}\frac{d}{dt}\int (|v^\varepsilon|^2
+|\varphi^\varepsilon|^2+|\nabla \varphi^\varepsilon|^2) dx
+C_1 \varepsilon\int |\nabla v^\varepsilon|^2 dx
+\frac{3}{4}\int (|\nabla \varphi^\varepsilon|^2+|\Delta \varphi^\varepsilon|^2) dx\\
&\le C\varepsilon|v^\varepsilon|_{L^2(\partial \Omega)}^2
+C\varepsilon|v^\varepsilon|_{L^2(\partial \Omega)}+\varepsilon^2
+C(\|v^\varepsilon\|_{L^2}^2+
          \|\varphi^\varepsilon\|_{L^2}^2+\|\nabla \varphi^\varepsilon\|_{L^2}^2).
\end{aligned}
\end{equation}
By virtue of the trace theorem in Proposition \ref{prop2.3}
and Cauchy inequality, one deduces that
\begin{equation}\label{5113}
|v^\varepsilon|^2_{L^2(\partial \Omega)}
\le \delta \|\nabla v^\varepsilon\|_{L^2}^2+C_\delta \|v^\varepsilon\|_{L^2}^2,
\end{equation}
and
\begin{equation}\label{5114}
\begin{aligned}
\varepsilon |v^\varepsilon|_{L^2(\partial \Omega)}
&\le \varepsilon \|v^\varepsilon\|_{L^2}^{\frac{1}{2}}
          \|\nabla v^\varepsilon\|_{L^2}^{\frac{1}{2}}
\le \delta \varepsilon \|\nabla v^\varepsilon\|_{L^2}^2
    +C_\delta \varepsilon \|\nabla v^\varepsilon\|_{L^2}^{\frac{2}{3}}\\
&\le \delta \varepsilon \|\nabla v^\varepsilon\|_{L^2}^2
    +C_\delta  \|\nabla v^\varepsilon\|_{L^2}^2
    +\varepsilon^{\frac{3}{2}}.
\end{aligned}
\end{equation}
Then the combination of \eqref{5112}-\eqref{5114} yields immediately
\begin{equation*}
\begin{aligned}
&\frac{d}{dt}\int (|v^\varepsilon|^2
+|\varphi^\varepsilon|^2+|\nabla \varphi^\varepsilon|^2) dx
+C_1 \varepsilon\int |\nabla v^\varepsilon|^2 dx
+\int (|\nabla \varphi^\varepsilon|^2+|\Delta \varphi^\varepsilon|^2) dx\\
&\le \varepsilon^{\frac{3}{2}}
+C(\|v^\varepsilon\|_{L^2}+
          \|\varphi^\varepsilon\|_{L^2}^2+\|\nabla \varphi^\varepsilon\|_{L^2}^2),
\end{aligned}
\end{equation*}
which, together with the Gr\"{o}nwall inequality,
completes the proof of the lemma \ref{lemma5.1}.
\end{proof}

\begin{rema}
By virtue of the nonlinear terms $\nabla d^\varepsilon \cdot \Delta d^\varepsilon$
and $\nabla d \cdot \Delta d$ on the right hand side of \eqref{eq1}$_1$
and \eqref{eq2}$_1$ respectively, we can not obtain the convergence rate for the quantity
$\|\nabla(u-u^\varepsilon)\|$ or $\|\nabla^2(d-d^\varepsilon)\|$ in this paper.
\end{rema}

\emph{\bf{Proof for Theorem \ref{Theorem1.2}:}}
Indeed, the estimate \eqref{122} just follows from the estimate \eqref{511}
in Lemma \ref{lemma5.1}.
On the other hand, by virtue of Sobolev inequality, uniform estimate \eqref{111}
and convergence rate \eqref{511}, it is easy to deduce
$$
\|u^\varepsilon-u\|_{L^\infty(0, T_2; L^\infty(\Omega))}
\le C\|u^\varepsilon-u\|_{L^2}^{\frac{2}{5}}
\|u^\varepsilon-u\|_{W^{1,\infty}}^{\frac{3}{5}}\le C\varepsilon^{\frac{3}{10}},
$$
and
$$
\|d^\varepsilon-d\|_{L^\infty(0,T^2; W^{1,\infty}(\Omega))}
\le C\|d^\varepsilon-d\|_{H^1}^{\frac{2}{5}}
\|d^\varepsilon-d\|_{W^{2,\infty}}^{\frac{3}{5}}\le C\varepsilon^{\frac{3}{10}},
$$
which implies the convergence rate \eqref{123}.
Therefore, we complete the proof of Theorem \ref{Theorem1.2}.

\section*{Acknowledgements}
The author Jincheng Gao would like to thank Yong Wang for fruitful discussion
and suggestion.

\phantomsection

\end{document}